\documentclass[11pt]{article}
\author{Bing-Long Chen$^{\dagger}$ $\&$ Ruo-Bin Wu$^{\dagger\dagger}$\\[8pt]}
\title{\textbf{On the existence of maximal foliations in general relativity}}
\date{Aug. 11, 2026}
\usepackage{latexsym}
\usepackage{amsmath}
\usepackage{amssymb,amscd}
\usepackage[mathscr]{eucal}
\newtheorem{theorem}{Theorem}[section]

\newtheorem{lemma}[theorem]{Lemma}
\newtheorem{proposition}[theorem]{Proposition}

\newtheorem{remark}{Remark}[section]

\numberwithin{equation}{section}
\newenvironment{proof}{{\noindent \it  Proof.}}{{\hfill$\Box$}\\}
\begin{document}
\maketitle

\let\thefootnote\relax\footnotetext{\noindent mcscbl@mail.sysu.edu.cn$^{\dagger}$;  wurb7@mail2.sysu.edu.cn$^{\dagger\dagger}$ \\ School of Mathematics, Sun Yat-sen University, Guangzhou, P.R.China\\   Mathematics Subject Classification 2020: 
83C45,  83C05, 53C44}

\begin{abstract}
It is well known that  the Einstein equations are tensor equations for a  Lorentzian metric. Hence, choosing a suitable  gauge condition is crucial for  solving them. As a powerful gauge condition, maximal foliations have  played a pivotal role in two groundbreaking works in general relativity: the  global stability of Minkowski spacetime \cite{CK93}  and the bounded  $L^2$  curvature conjecture \cite{KRS15}.  Nevertheless, if the initial hypersurface fails to be maximal, the usual elliptic  approach  for constructing maximal foliations encounters  fundamental difficulties  when   solving the Einstein equations.  

The purpose  of the paper is
to provide a construction of maximal foliations around any initial  asymptotically flat Cauchy surface satisfying vacuum Einstein constraint equations, under a suitable smallness condition on the mean curvature. 
   \end{abstract}


\section{Introduction}
\subsection{Bounded $L^2$ curvature conjecture}

 The Cauchy  problem for the vacuum Einstein equations refers to the following: Given a triplet $(\Sigma, g,h)$, where $(\Sigma,g)$ is a 3-dimensional Riemannian manifold and $h$ is a symmetric 2-tensor on $\Sigma$ satisfying the constraint equations,

\begin{equation}\label{ECE}
\left\{
\begin{aligned}
    & {R}+{H}^{2}-|{h}|^{2}=0,   \\
     &  div_{{g}}{h}-{\nabla}{H}=0.
\end{aligned}
\right. 
\end{equation} 
      The  problem  asks  whether there exist a 4-dimensional Lorentzian manifold $(\bar{M}, \bar{g})$ with 
      \begin{equation}\label{EVE}
          \bar{R}ic=0,
      \end{equation}
      and an isometric embedding  $(\Sigma, g)\subset (\bar{M}, \bar{g})$ such that $h$ is the second fundamental form.  
      
      The problem of  well-posedness  of the vacuum Einstein equations (\ref{ECE})(\ref{EVE}) asks:  for what kind of  initial data $(\Sigma, g,h)$,   the equation admits a unique solution?  Here,   the regularity requirement for $(g,h)$ is desired to be as low as possible.  The pioneering work of  Y. Choquet-Bruhat \cite{CB52} states that  the Cauchy problem for the Einstein equations is well-posed when $(g,h)$ is smooth. 
      The  classical result (see \cite{FM72}\cite{HKM77}),   as an improvement of   \cite{CB52},  asserts that the well-posedness holds  when $(g,h)\in H^{s}_{loc}(\Sigma)\times H^{s-1}_{loc}(\Sigma)$ for $s>\frac{5}{2}$.   The  bounded $L^2$  curvature conjecture, proposed by S. Klainerman in \cite{K99}, roughly states that the well-posedness holds when $Rm\in L^{2}_{loc}(\Sigma)$ and $h\in W^{1,2}_{loc}(\Sigma)$, where $Rm$ is the Riemann  curvature tensor of $g$.

After a series of groundbreaking  works (\cite{SJI},\cite{SJII},\cite{SJIII},\cite{SJIV},\cite{KRS15}),  the conjecture was ultimately proven by S. Klainerman, I. Rodnianski, and J. Szeftel in \cite{KRS15} under the existence of maximal foliations.

The proof in \cite{KRS15} consists of the following three major steps. \\
1)   Assume the spacetime  has  a maximal foliation, and each leaf is asymptotically flat.\\
2)  Using Cartan's moving frame method, rewrite the Einstein equations as a Yang-Mills equation for the connection. The time-like frame $e_0$ is naturally the unit normal to the foliation, while three space-like orthonormal frames $e_{i}$, $i=1,2,3,$  are chosen to satisfy a  Coulomb gauge condition.\\
3)  Bound the $L^2$
  integral of spacetime curvatures by applying energy estimate through the  Bel–Robinson tensor $Q$.

    For step 2),   if one  denote the connection as a $4\times 4$ matrix of 1-forms  $(A_{\alpha})_{\mu\nu}:=\bar{{g}}(D_{e_{\alpha}}e_{\nu}, e_{\mu})$, the maximal foliation enables one  to  solve an elliptic equation to obtain an orthonormal tetrad satisfying an equation of the form ${div} A=A^{2}$ for spacial component $A=\{A_i\}$.   As a result, the equation for $A_{0}$ reduces to an elliptic equation, which yields better regularity results for hyperbolic equations.  Specifically, using Littlewood-Paley theory, one can derive $W^{2,\frac{3}{2}}$ estimate  for $A_{0}$ as well as $L^{\infty}$ gradient estimate for the lapse function. 
 Furthermore, under this gauge, the hyperbolic equation for $A_{i}$ exhibits a null structure, which is crucial for the application of certain Strichartz estimates of step 3). 
  
 For step 3),  when directly computing the divergence  $D^{\alpha}(Q_{\alpha 000})$ and deriving the $L^{2}$ energy estimate for the curvature,  the result  contains some  terms of the form $\bar{R} * {h} * \bar{R}$.  It appears that the energy estimate for the curvature  requires the  boundedness of the second fundamental form,  which is incompatible with the framework of the bounded $L^2$ curvature conjecture. To overcome this difficulty, one need  to explicitly represent the solutions of hyperbolic equations by using  parametrices and derive some sharp form of Strichartz estimates (\cite{SJI},\cite{SJII},\cite{SJIII},\cite{SJIV},\cite{KRS15}).

This proof of bounded $L^2$ curvature conjecture   combines geometric PDE techniques (foliation theory) with gauge-theoretic methods (Yang-Mills reformulation) and harmonic analysis (parametrix construction and estimates).

Nevertheless,  recall that the Cauchy surfaces in the full form of the bounded $L^2$ curvature conjecture are generally not maximal and not even asymptotically flat. Because of the hyperbolicity of Einstein equations,   the bounded $L^2$ curvature conjecture asks essentially the well-posedness problem on a local surface. To exploit the  global techniques,   the  overall approach  is roughly to  glue  an asymptotically flat end to the exterior of the  surface.  This is possible if the regularity of $(g,h)$ is higher, see Remarks 2.7, 2.8 in \cite{CS06}. Here, we will not discuss the gluing precess.  Instead,  we assume the initial Cauchy surface is already  asymptotically flat.

By solving the  Cauchy problem  on the asymptotically flat initial surface,  it   yields a short-time solution, i.e., a 4-dimensional vacuum 
$(\bar{M}, \bar{g})$ with $\bar{M}\supset \Sigma$. 
 Naturally,  one may  desire to   construct a maximal hypersurface $L\subset \bar{M}$ such that 
$L$ lies within a foliation of maximal surfaces. 

The  issue is the following. To ensure the existence of the maximal hypersurface $L$, when applying the usual  construction based on elliptic techniques, we need some  barrier conditions (e.g.,  \cite{B84},\cite{BCM90}) to be satisfied in some   relevant regions. However, a priori we do not  know how large $\bar{M}$ is.  It appears that the elliptic techniques on constructing maximal foliations  have  some  unavoidable difficulties.  There are  a number of   works on constraint equations (\ref{ECE})  which aim to  construct maximal hypersurfaces within the conformal class of the initial metric, for related results we may  refer to references \cite{Li44}, \cite{CB74}, \cite{I95}, \cite{CIJ00} etc..  Nevertheless, these results appear insufficient to resolve the present problem. 
   
  Note that  the  global  stability of the Minkowski spacetime in \cite{CK93} was also proved under the gauge of maximal foliations (\cite{LR10} gave a proof under the harmonic gauge and higher regularity assumptions).   As we remarked previously, the construction   of  maximal foliations  is a nontrivial issue  when the initial Cauchy surface  is not maximal.

\subsection{Our strategy and main results}

\indent In the followings, we discuss our  strategy adopted in this paper.

The rough idea is as follows.  We attempt to deform $\Sigma$ in the direction of the mean curvature, i.e., we consider the mean curvature flow:
\begin{equation} \label{mcf}
\frac{\partial X}{\partial t}= \overset{\rightarrow}{H}
\end{equation}
where $X(\cdot, t):\Sigma\hookrightarrow \bar{M}$ is a one-parameter family of embeddings, $\overset{\rightarrow}{H}$ is the mean curvature vector of $X_t(\Sigma)$ in $\bar{M}$.  
As long as there is an open background spacetime $\bar{M}\supset\Sigma$,  one can try to deform 
$\Sigma$ within $\bar{M}$ by using the  mean curvature flow.

The point is that $\bar{M}$ can contain a large region at infinity of $\Sigma$, which ensures the mean curvature flow (\ref{mcf}) may  exist at least for a short time interval $[0,t_0)$ for $t_0>0$.   Actually,  the asymptotically flat end glued to the initial surface can be chosen to be a canonical spacelike slice in the Kerr spacetime. Hence, due to the  hyperbolicity of Einstein equations,  a large region of $\bar{M}$ at infinity is a part of the Kerr spacetime.

Let $[0,t_0)$ be the largest possible interval for which the mean curvature flow (\ref{mcf}) exists in $\bar{M}$. If $t_0<\infty$, as time $t$ increases to $t_0$, $X_t(\Sigma)$  may gradually approach the boundary of $\bar{M}$.  By taking limits, one can treat 
$X_{t_0}(\Sigma)$ as a new Cauchy surface. By solving the Cauchy problem with $X_{t_0}(\Sigma)$ as a new initial data,  the spacetime 
$\bar{M}$ can be extended to $\bar{M}^{\prime}\supset \bar{M}$ which contains a large region at infinity of $X_{t_0}(\Sigma)$.  Within $\bar{M}^{\prime}$,  we continue to deform  $X_{t_0}(\Sigma)$ by means of  the mean curvature flow (\ref{mcf}). 

 One may  repeat this process and expect that  within finite number of steps,  the mean curvature flow will  exist for all time and converge to a maximal hypersurface as  $t\rightarrow \infty$.

For more general results and further applications,  instead of having a Kerr end,   we may  assume  the initial data  $(M_0, g,h)$  is only asymptotically flat  in the following  very weak sense:  the complement  $M_{0}\setminus C_0$ of a compact set $C_0$ is diffeomorphic to $\mathbb{R}^3\setminus B_a$ for some $a>0$, let $(x^1,x^2,x^3)$ be the standard coordinate system in $\mathbb{R}^3$, as $r=|x|\rightarrow \infty$,  we assume  $(M_0\setminus C_0, g,h)$ satisfies the following  asymptotic conditions (for some small $\epsilon>0$):
\begin{equation}\label{AC}
    \begin{split}
        g_{ij}=\delta_{ij}+O_{4}(r^{-\epsilon}),\  h_{ij}=O_{3}(r^{-1-\epsilon}), \ H=O_3(r^{-2-\epsilon}),
    \end{split}
\end{equation}
where $F=O_i(r^j)$ indicates  that $F$ is a function or tensor satisfying  $|\nabla^IF|\leq Cr^{j-|I|} $ for all  $|I|\leq i$. 

We introduce two constants $\alpha_0, \alpha_1$ for $(M_0,g)$. 

The first constant $\alpha_0$ is the Sobolev constant:
\begin{equation} \label{II}
 ||f||_{L^{\frac{3}{2}}} \leq \alpha_0||\nabla f||_{L^1},  \ \ \forall \ \  f\in C^{\infty}_0(M_0).
 \end{equation}
 
 The second constant $\alpha_1$ is the volume ratio constant 
\begin{equation} \label{III}
\alpha_1^{-1}\leq \frac{vol(B(x_0,r))}{r^3} \leq \alpha_1,   \ \ \forall \ \  r>0,
\end{equation}
whose definition  may depend on the choice of a particular point $x_0$.

From (\ref{AC}), one can easily prove the integrability of the following integrals: 
\begin{equation}\label{IC} \sum_{j=0}^2||\nabla^j {R}m||_{L^2}+\sum_{j=1}^3||\nabla^j h||_{L^2} +\sum_{j=0}^2||\nabla^j \bar{R}m||_{L^2}<\infty,
\end{equation}
where $\bar{R}m$ is the spacetime curvature tensor which can be defined by  Gauss and Codazzi equations,
\begin{equation}\label{Gauss}
  \bar{R}_{ijkl}=R_{ijkl}+ h_{ik}h_{jl}-h_{il}h_{jk},
\end{equation} 
\begin{equation} \label{Codazzi}
\bar{R}_{ijk0}=\nabla_ih_{jk}-\nabla_jh_{ik}.
\end{equation}
 The component $\bar{R}_{i0j0}$ can be obtained by taking trace on (\ref{Gauss}): 
\begin{equation}\label{tr Gauss}
 \bar{R}_{i0j0} =R_{ij}+H h_{ij}-h^{2}_{ij}.
    \end{equation} 

The covariant derivative  $\nabla^j\bar{R}m$ in (\ref{IC}) is the j-th covariant derivative of  $\bar{R}m$ with respect to the intrinsic metric $g$  (the unit normal can be regarded as parallel with respect to the covariant differentiation $\nabla$).    

For the convenience of stating  the main result,  we introduce two quantities  (or notations): 
 \begin{equation} \label{E1}
      E^2_1 \triangleq  \int_{M_0}|\bar{R}m|^2+|\nabla h|^2+|h|^4,
\end{equation}
and 
\begin{equation} \label{E_2}
 E_2^2\triangleq \int_{M_0}|\nabla \bar{R}m|^2+|\nabla^2 h|^2+|h|^6+|\bar{R}m|^3+|\nabla h|^3,
\end{equation}
both of which are  finite by (\ref{AC}).  

From (\ref{AC}) again, there exists an exponent  $p_0\in (1,\frac{3}{2})$ such that $||H||_{L^{p_0}}<\infty$, e.g.,  all $p_0\in (\frac{3}{2+\epsilon}, \frac{3}{2})$ fulfills the requirement.  Let us fix such an exponent  $p_0$.

The main  result  of the  paper is  the following theorem:

\begin{theorem}\label{maxfol}
          Let $(M_{0},{g},{h})$ be an asymptotically flat initial data set (\ref{AC})  which satisfies the vacuum constraint equations (\ref{ECE}). For any $1<p_0<\frac{3}{2}$,     there exist  constants     $\epsilon_{0}>0$, $C>0$,  depending only on $p_0$ and $\alpha_0, \alpha_1$ in (\ref{II}) and (\ref{III})  satisfying the following properties.  If  
 \begin{equation}\label{HL_p1}
\|{H}\|_{L^{p_0}(M_{0})} (E_1E_2)^{\frac{3-p_0}{2p_0}}<\epsilon_0,
   \end{equation}        
\begin{equation}\label{HL_p2}
  \begin{split}
 \|{H}\|_{L^{p_0}(M_{0})}  E_1^{\frac{6}{p_0}-2}<\epsilon_0,\end{split}
      \end{equation}
       \begin{equation}\label{dHL_2}
   (||\nabla H||^2_{L^2(M_0)}+ ||H||_{L^{\infty}(M_0)})( ||\nabla^2 H||^2_{L^2(M_0)}+E_1E_2||H||_{L^{\infty}(M_0)})<\epsilon^4_0 E_1^8,
      \end{equation}
           then  there exist a vacuum spacetime $(\bar{M},\bar{g})$ in the maximal development $\mathcal{M}$ of $M_0$, and  a time function ${\tau}: \bar{M}\rightarrow (-\tau_0,\tau_0)$,   the level sets $\{\Sigma_\tau: \tau\in (-\tau_0,\tau_0)\}$ form  a foliation of maximal and asymptotic flat surfaces.  Moreover, the maximal foliation covers $M_0$, i.e.,  $M_0\subset \bar{M}$, and  we have 
  \begin{equation} \label{1.16}
\begin{split}
    &  \int_{\Sigma_{\tau}}|\bar{R}m|^2+|\nabla h|^2+|h|^4 \leq CE^2_1\\
& \int_{\Sigma_{\tau}}|\nabla \bar{R}m|^2+|\nabla^2 h|^2+|h|^6+|\bar{R}m|^3+|\nabla h|^3 \leq CE^2_2.
\end{split}
\end{equation}

 \end{theorem}

\begin{remark}
In Theorem \ref{maxfol},  the  only nontrivial conditions are (\ref{HL_p1}) (\ref{HL_p2}) and  (\ref{dHL_2}), which are scaling invariant and can be regarded as some smallness conditions on the mean curvature of the  initial Cauchy surface.  They are used to ensure the existence and smallness of the following  two integrals 
\begin{equation} \label{integrals}
\int_{0}^{\infty}||Hh||_{L^{\infty}}dt, \ \ \  \int_{0}^{\infty}||\nabla H||_{L^{\infty}}dt 
\end{equation}
during the mean curvature flow. The conditions  (\ref{HL_p1}) (\ref{HL_p2}) and  (\ref{dHL_2}) can be replaced by other conditions which can guarantee  the smallness of the two integrals in (\ref{integrals}).
\end{remark}

\begin{remark}     We remark that the  first inequality of (\ref{1.16}) is required   when applying the current  result to bounded $L^2$ curvature conjecture (see \cite{KRS15}). 
  \end{remark}

\begin{remark} As we mentioned previously,   the celebrated global stability of Minkowski spacetime was proven in \cite{CK93}  under the gauge of the maximal foliations.  From Theorem \ref{maxfol},  the existence of  maximal foliation can be guaranteed by conditions  (\ref{HL_p1}) (\ref{HL_p2}) and  (\ref{dHL_2}). As a consequence, the main result in \cite{CK93} can be correspondingly extended. 
\end{remark}

To prove Theorem \ref{maxfol} by adopting  the strategy of the mean curvature flow (\ref{mcf}),   the main task is  to prove the long time existence and  global convergence of the solution to (\ref{mcf}). The limit maximal hypersurface is uniquely and intrinsically determined by  the initial Cauchy surface.   This result has an   independent interest.  
\begin{theorem} \label{t1.5}

Under the assumptions of Theorem \ref{maxfol},  there exist a vacuum  spacetime $(\bar{M},\bar{{g}})\supset M_0$, so that the mean curvature flow (\ref{mcf}) admits a long time solution $X_t$ in $\bar{M}$,  converging  to a maximal spacelike hypersurface in $\bar{M}$ as $t\rightarrow \infty$.         
\end{theorem}

We mention  one  close  reference \cite{E93},  which dealt with the convergence of the mean curvature flow  under the existence of barrier functions on a given  ambient manifold.

  The arrangement of this paper is as follows. In Section 2, we establish the short time existence of the graphical  mean curvature flow.   In Section 3,  we  derive some preliminary asymptotic estimates of graphical mean curvature flow  at infinity.  In Section 4, we deduce  some fundamental  mean curvature estimates, and  transform the solutions of graphical mean curvature flow into those of the mean curvature flow.  In Section 5, we  reformulate    Theorem \ref{t1.5} in  small  data  version,  and set   up  the  necessary bootstrap assumptions for our arguments.   In Section 6,  we improve the mean curvature estimate  via the De Giorgi-Nash-Moser iteration.  In Section 7,  we obtain   the $L^2$ estimates of the mean curvature up to third-order derivatives  by means of   energy estimates.  In Section 8, we employ  the mean curvature estimates  obtained in Sections 4, 6  and 7 to refine  the energy estimates for the curvature and the second fundamental form. This is based on the null structure of the Einstein equations, which essentially allows us  to form a closed system of estimates when coupling with the mean curvature flow. In Section 9, we argue  that the uniform bootstrap assumptions  can be justified as long as the solution of the mean curvature flow exists.  In Section 10, we develop  some  higher order derivative estimates.  In  Section 11,  we accomplish  the convergence of the mean curvature flow  and  construct  the desired maximal foliations.  

We conclude this section  with  a remark.   In view of Theorem \ref{maxfol},    the justification of (\ref{HL_p1}) (\ref{HL_p2}) (\ref{dHL_2}) under the gluing process becomes crucial.  This problem  merits  further investigation.                          

\paragraph{Acknowledgment} This work is partially supported by grants  National Key R$\&$D Program of China(No.2022YFA1005400) and NSFC 12141106.

\section{Mean curvature flow}

To solve equation (\ref{mcf}),    we  first solve a graphical  mean curvature flow and  show  that its  solution must exist for   at least a 
short time interval $[0, \bar{T})$ with $\bar{T}>0$.  To this end, we begin with   some preliminary work.

\subsection{Preliminaries}\label{Pre}

Let $\mathcal{M}\supset M_0$ be the vacuum maximal development of $M_0$ in the sense of \cite{CG69}. Let $\tau$ be the function of affine parameters of all normalized timelike geodesics emanating from $M_0$ and normal to $M_0$.  
From condition (\ref{IC}), for any $\delta>0$, there exist $\delta_0, C_0$ depending $\delta,\alpha_0$ and $\sum_{j=0}^2||\nabla^jRm||_{L^2}$ such that for any point $P\in M_0$, we have  a harmonic coordinate system $\{ x\in \mathbb{R}^3:|x|<2\delta_0\}$ around $P$, i.e., $x^i(P)=0$, $\triangle_g x^i=0$, so that $g_{ij}=g(\frac{\partial}{\partial x^i}, \frac{\partial}{\partial x^j})$ satisfies 
\begin{equation} \label{Har}
||g_{ij}-\delta_{ij}||_{C^{1,\frac{1}{4}}(\{|x|<2\delta_0\})}<\frac{1}{2}\delta,\ \ \ \  ||g_{ij}-\delta_{ij}||_{W^{4,2}(\{|x|<2\delta_0\})}\leq \frac{1}{2}C_0. 
\end{equation}

We extend $x^i,i=1,2,3,$ constantly along geodesics perpendicular to $M_0$ so that $x^i,i=1,2,3$ can be defined on  spacetime  around $P$.  

 By solving the Cauchy problem (\ref{ECE}) (\ref{EVE}) (see \cite{HKM77}), there exists a coordinate system $\{ (y,z)\in \mathbb{R}^3\times\mathbb{R}:|y|<\delta_0, |z|<\delta_0\}$ around $P$ in $\mathcal{M}$ so that $z\mid_{M_0}=\tau\mid_{M_0}=0$,   $y^i\mid_{M_0}=x^i\mid_{M_0}$, $dy^i\mid_{M_0}=dx^i\mid_{M_0}$, $dz\mid_{M_0}=d\tau\mid_{M_0}$ and under this coordinate system we have 
\begin{equation}\label{wave}
||\bar{g}-\eta||_{C^{1,\frac{1}{4}}(\{|y|<\delta_0, |z|<\delta_0\})}<\delta, \ \ \  \sum_{|\alpha|=1}^4||\partial^{\alpha}\bar{g}||_{L^{2}(\{|y|<\delta_0\})} \leq C_0,
\end{equation}
where $\eta=diag\{-1,1,1,1\}$ is the Minkowski metric.

In particular, the components of the spacetime curvature $\bar{R}m$ in coordinates $\{y^{i},z\}$ are  uniformly  bounded. 

Let $N_0$ be the unit future timelike normal vector of  $M_0$ in $\mathcal{M}_1$. Let $\gamma(P,\tau)=\exp_P(\tau N_0)$,  be the normal geodesic starting from a point $P\in M_0$, where $\exp_{P}$ is  the exponential map at $P$ of the Lorentzian manifold $(\mathcal{M},\bar{g})$.

\textbf{Claim}: there exists $\delta_1>0$ such that for any fixed $\tau\in [-\delta_1, \delta_1]$, $\gamma(\cdot,\tau): M_0\rightarrow \mathcal{M}$ is injective.  

Denote by $T=\gamma^{\prime}$ the tangent vectors of these timelike  geodesics  $\{\gamma(P,\tau): P\in M_0, \tau\in [-\delta_1, \delta_1]\}$, $T$ will become a vector field if the \textbf{Claim} is true. 

Write the equation of geodesic $\gamma$ in coordinates $\{y^{\alpha}\}$, $\alpha=0,1,2,3$, with $y^0=z$:
\begin{equation} \label{2.3}
\frac{d^2 y^{\alpha}}{d\tau^2}+\bar{\Gamma}^{\alpha}_{\beta\gamma}\frac{dy^{\beta}}{d\tau}\frac{ dy^{\gamma}}{d\tau}=0,
\end{equation}
where  $\frac{dy^0}{d\tau}\mid_{\tau=0}=0$, $\frac{dz}{d\tau}\mid_{\tau=0}=1$. 

Since $\bar\Gamma^{\alpha}_{\beta\gamma}$ is bounded (\ref{wave}), integrating the above equation (\ref{2.3}) we obtain  
\begin{equation} \label{2.4}
|\frac{dz}{d\tau}-1|\leq Cr^{-1-\epsilon}|\tau|, \ \ \  |\frac{dy^i}{d\tau}|\leq Cr^{-1-\epsilon}|\tau|,
\end{equation}
i.e., $|T-\frac{\partial}{\partial z}| \leq  C\delta_1$, which implies that the spacetime curvature $\bar{R}m$,  measured by using the norm induced from  $T$, is also uniformly bounded.   

Along a geodesic $\gamma(P,\tau)$, the second fundamental form $\bar{h}_{ij}=\frac{1}{2}\frac{\partial \bar{g}_{ij}}{\partial \tau}$ of local foliations $\{\tau=const.\}$ (see (\ref{nor g})) satisfies the  Riccati equation \begin{equation} \label{nor evo}
\frac{d}{d\tau}\bar{h}_{ij}= \bar{g}^{pq}\bar{h}_{ip} \bar{h}_{jq}-\bar{R}_{i0j0}.
\end{equation}
By integrating the equation (\ref{nor evo}), since $\bar{R}m$ is uniformly bounded,  one can deduce  that $\bar{h}$ is bounded when $\tau$ is suitably small.  Combining with   $|DT|\leq |\bar{h}|$ (using the norm induced by $T$) we obtain   
\begin{equation}\label{DT}
    |DT|\leq C.
\end{equation} 
Consequently, for any two distinct  points $P_1, P_2 \in M_0$ in one coordinate system   (\ref{wave}), we have
\begin{equation}\label{TLip}
|T(P_1,\tau)-T(P_2,\tau)|\leq C |\gamma(P_1,\tau)-\gamma(P_2,\tau)|. 
\end{equation}
If $\gamma(P_1,\tau_0)=\gamma(P_2,\tau_0)$ for suitably small $\tau_0$, (\ref{TLip}) implies  $T(P_1,\tau_0)=T(P_2,\tau_0)$. By uniqueness of geodesics, we know $P_1=P_2$, which is a contradiction.  Thus,  \textbf{Claim} is proved. 

Now 
\begin{equation}\label{tau}
    \Sigma_{\tau}\triangleq \{\tau=const.\}, \ \ \tau\in (-\delta_1, \delta_1)
\end{equation}
defines a smooth  foliation around $M_0$,  and  the metric $\bar{g}$ on this foliation has the following expression: 
\begin{equation} \label{nor g}
\bar{g}=-d\tau^2+\bar{g}_{ij}(x,\tau)dx^idx^j,
\end{equation}
where $\bar{g}_{ij}(x,t)dx^idx^j$ is the induced Riemannian metric on $\Sigma_{\tau}$.

Due to the  asymptotically flat condition (\ref{AC}),   one can actually deduce   some decay estimates for  $\bar{R}m$ and $\bar{h}$. The argument is as follows.

 For any point $x_0$ in $M_0 \setminus C_0$,  in the asymptotically flat coordinates $\{x^1,x^2,x^3\}$ in (\ref{AC}),   we have 

 \begin{equation}\label{AC0}
\sum_{i=1}^{4}||\partial^i g||_{L^2(\{|x-x_0|<\frac{r}{2}\})}r^{i}+\sum_{i=0}^{3}||\partial^i h||_{L^2(\{|x-x_0|<\frac{r}{2}\})}r^{i+1}\leq C r^{-\epsilon} r^{\frac{3}{2}}.
\end{equation}

 By solving the Cauchy problem (\ref{EVE}),  there exists a local wave  coordinate system $\{(y,z)\in \mathbb{R}^3\times\mathbb{R}:|y|<\delta_0 r, |z|<\delta_0 r\}$ on $\mathcal{M}$ around $x_0$  so that $y^i\mid_{M_0}=x^i-x^i(x_0)$, $z\mid_{M_0}=0$, $dy^i\mid_{M_0}=dx^i\mid_{M_0}$ (extending $x^i$ constantly along $\tau$-geodesics),  $dz\mid_{M_0}=d\tau\mid_{M_0}$ and under this coordinate system we have 

  \begin{equation}\label{AC est}
||\bar{g}-\eta||_{L^{\infty}(|y|<\delta_0{r})}\leq Cr^{-\epsilon},  \ \ \ \ \sum_{|\alpha|=1}^{4}||\partial^{\alpha} \bar{g}||_{L^2(\{|y|<\delta_0{r}\})}r^{|\alpha|}\leq C r^{-\epsilon} r^{\frac{3}{2}}.
\end{equation}

In particular, by Sobolev embedding theorem, in the above wave coordinates $\{|y|<\delta_0 r, |z|<\delta_0 r\}$, we obtain 
\begin{equation} \label{ACR}
|\bar{R}m| \leq C r^{-2-\epsilon}.
\end{equation}

Because of (\ref{2.4}),  the estimate (\ref{ACR}) still holds in the frame given by $\tau-$ foliations (\ref{tau}).

By solving the equation (\ref{nor evo}) again, the estimate of the second fundamental form $\bar{h}$ in foliation (\ref{tau}) can also be improved: 
\begin{equation}\label{ACh}
|\bar{h}|\leq C r^{-1-\epsilon}.
\end{equation}

Moreover,  taking trace on (\ref{nor evo}) gives 
\begin{equation}\label{2.14}
    \frac{d}{d\tau}\bar{H}=-|\bar{h}|^2 
\end{equation}
which together  with (\ref{AC})  implies 
\begin{equation}\label{2.15}
|\bar{H}|\leq C r^{-2-\epsilon}.
\end{equation}

\subsection{Short time existence}\label{exist}

We consider the mean curvature flow equation (\ref{mcf}). 
Suppose   $M_{t}=X_t(M_0)$ can be written as an entire graph over  $M_{0}$ in $\tau$-foliation (\ref{tau}),  i.e.,  $M_t=\{(P, f(P,t)): P\in M_{0}, \tau=f(P,t)\}$.   Let $g$ be the induced  metric on $M_{t}$. Under  the  natural frame $DX( \frac{\partial}{\partial x^{i}})=\frac{\partial }{\partial x^{i}}+f_{i} \frac{\partial }{\partial \tau}$,  we have  $g_{ij}=\bar{g}_{ij}-f_{i}f_{j}$.
Let  $T=\frac{\partial}{\partial \tau}$ and $N$ be the unit future-directed normal vector fields  of $\varSigma_{\tau}$ and $M_t$ respectively.  Let    $\nabla f=(D\tau)^{\|} \in TM_t$, then $\nabla f=\nu N-T$ and  $\nu=-<N,  T>=\sqrt{1+|\nabla f|^2}$.   We   use $D$ and $\nabla$ to denote the covariant derivatives on spacetime $\bar{M}$ and submanifold $M_t$ respectively.

Note that  $N=\nu^{-1}(\nabla f+T)$,   $H=div_{M_{t}}(N)=\nu^{-1}div_{M_{t}}(\nabla f+T)$, $\frac{\partial X}{\partial t}=\frac{\partial f}{\partial t}T$.    The  deformation  $(\frac{\partial X(x,t)}{\partial t})^{\perp}=\overset{\rightarrow}{H}\mid_{{X}(x,t)}$ is equivalent to the following equation (graphical mean curvature flow equation): 
\begin{equation} \label{mcfg}
\left\{
\begin{aligned}
    &\frac{\partial f}{\partial t}=\nu^{-2}(div_{M_{t}}(\nabla f)+div_{M_{t}}(T)),  \\
     &f(P,0)=0,  \qquad P\in M_{0}.
\end{aligned}
\right. 
\end{equation}

On the other hand, suppose we have a solution  $X(\cdot,t)$ to the deformation $(\frac{\partial X(x,t)}{\partial t})^{\perp}=\overset{\rightarrow}{H}\mid_{{X}(x,t)}$,   after reparameterizing the domain manifold  with  a family of diffeomorphisms $\phi_t$,  the tangential part of the deformation can be killed, see (\ref{410})(\ref{411}). That is to say,  $\tilde{X}(x,t)=X(\phi_t(x),t)$ satisfies $\frac{\partial \tilde{X}(x,t)}{\partial t}=\overset{\rightarrow}{H}\mid_{\tilde{X}(x,t)}$. For this reason,   it suffices   to   solve the equation (\ref{mcfg}). 

If we extend $f\mid_{M_t}$ constantly along $\tau$-geodesics, $f$ can be defined on spacetime.  We have   $Df\perp T$, $ (Df)^{||}=\nabla f$, $\sqrt{1-|Df|^2}=\nu^{-1}$, and $N=\frac{1}{\sqrt{1-|Df|^2}}(T+Df)$. Consequently,  
\begin{equation} \label{215}
\nu^{-2}(div_{M_{t}}(\nabla f)+div_{M_{t}}(T)) =div_{M_t}(Df+T)=tr_{M_t} D^2f+div_{M_t}(T).
\end{equation}

\begin{theorem} \label{shorttime}
There exists $\bar{T}>0$, so that the graphical mean curvature flow (\ref{mcfg}) admits a short time  solution $f(P,t)$, for $t\in [0,\bar{T}]$. 
\end{theorem}

In the proof of Theorem \ref{shorttime}, we need an auxiliary function $\tilde{r}$, whose construction will be given as follows. We solve a linear wave equation

\begin{equation}
\left\{
\begin{aligned}
    &\Box  z^i=0,  \\
     & z^i\mid_{M_0\setminus K}=x^i,\\
     & dz^i\mid_{M_0\setminus K}=dx^i\mid_{M_0\setminus K},
\end{aligned}
\right. \end{equation}
where $\Box$ is the Laplace Beltrami operator of the Lorentzian metric $\bar{g}$,  the functions $x^i$  are obtained by constantly extending the  coordinate functions $x^i$ in (\ref{AC}) along $\tau$-geodesics, $K$ is a big compact set containing $C_0$.  We define a spacetime function $\tilde{r}=\sqrt{(z^1)^2+(z^2)^2+(z^3)^2}$. It is of importance to notice that in each coordinate system where  (\ref{AC est}) holds, we always have 
\begin{equation} \label{217}
z^i=y^i+ Const., 
\end{equation}
on   $\{|y|\leq \delta_0 r-C\}$. (\ref{217}) follows from the uniqueness of  the solutions of  hyperbolic equations.

\begin{lemma} \label{l2.2} The function $\tilde{r}$ satisfies the following properties: 

i) \ $|\frac{\partial z^{\alpha}}{\partial x^{\beta}}-\delta_{\alpha\beta}|\leq C r^{-1-\epsilon},$

ii) \  $r-Cr^{-1-\epsilon}\leq \tilde{r}\leq r+Cr^{-1-\epsilon}$,

iii) \  $ ||D\tilde{r}|^2-1|\leq C r^{-\epsilon}$,

iv) \  $ D^2\tilde{r}(\frac{\partial}{\partial y^i}, \frac{\partial}{\partial y^j})=\frac{1}{\tilde{r}}(\delta_{ij}-\frac{z_iz_j}{\tilde{r}^2})+O(r^{-1-\epsilon}),$  $D^2\tilde{r}(\frac{\partial}{\partial y^i}, \frac{\partial}{\partial y^0})=O(r^{-1-\epsilon}),$    

\ \  \ \ \ $D^2\tilde{r}(\frac{\partial}{\partial y^0}, \frac{\partial}{\partial y^0})=O(r^{-1-\epsilon}).$ 
\end{lemma}
\begin{proof} Differentiating (\ref{2.3}) with respect to $x^{i}$, we obtain 
$$ \frac{d^2}{d\tau^2} \frac{\partial y^{\alpha}}{\partial x^i}=-\partial_{\xi}\bar{\Gamma}^{\alpha}_{\beta\gamma}\frac{\partial y^{\xi}}{\partial x^i}\frac{d y^{\beta}}{d\tau}\frac{d y^{\gamma}}{d\tau}-2\bar\Gamma^{\alpha}_{\beta\gamma}\frac{d}{d\tau}(\frac{\partial y^{\beta}}{\partial x^{i}})\frac{d y^{\gamma}}{d\tau}.$$
Because  of  $ \frac{d}{d\tau}\frac{\partial y^{\alpha}}{\partial x^i}\mid_{\tau=0}=\frac{\partial }{\partial x^i}(\frac{d y^{\alpha}}{d\tau}\mid_{\tau=0})=0$, integrating the above equation yields  $|\frac{d}{d \tau} (\frac{\partial y^{\alpha}}{\partial x^i})|\leq C r^{-1-\epsilon}|\tau|$, whose further integration together with (\ref{2.4}) produces    i).

Differentiating $\tilde{r}$ with respect to $\tau$, we have   $\frac{\partial \tilde{r}}{\partial \tau}= \frac{z^i}{\tilde{r}} \frac{\partial y^i}{\partial \tau}$, which together with (\ref{2.4}) implies $|\frac{\partial \tilde{r}}{\partial \tau}|\leq Cr^{-1-\epsilon}|\tau|$. Integrating this inequality deduces  $|\tilde{r}-r|\leq C\tau^2r^{-1-\epsilon}$,  we obtain ii).   iv)  is a direct  consequence  of (\ref{wave}) (\ref{2.4}) (\ref{217}). Using  $\frac{d}{d\tau}|D\tilde{r}|^2=2D^2\tilde{r}(D\tilde{r},T)$, $|D\tilde{r}|^2\mid_{\tau=0}=1+O(r^{-\epsilon})$ and iv), we obtain iii). 
\end{proof}

\begin{proof} of Theorem \ref{shorttime}. 

Let $\Omega_i=\{ \tilde{r}<i\} \subset \mathcal{M}_1$, we consider the Cauchy-Dirichlet problem on $\Omega_i$:
\begin{equation} \label{CDP}
\left\{
\begin{aligned}
    &\frac{\partial f}{\partial t}=\nu^{-2}(\triangle f+div_{M_{t}}T),  \ \ \ \text{on} \ \Omega_i,\\
     &f(P,0)=0,  \qquad P\in \Omega_i\cap M_0,\\
     & f(P,t)=0, \qquad P\in \partial_{s} \Omega_i,
\end{aligned}
\right. 
\end{equation}
where $\partial_s\Omega_i=\partial \Omega_i \cap\{-\delta_1<\tau<\delta_1\}$ is the spacial boundary of $\Omega_i$. 

Clearly,   (\ref{CDP}) is a quasi-linear parabolic equation. It admits a  smooth solution on some time interval $[0, \bar{T}_i)$. We may assume $\bar{T}_i$ is maximal. 

Let $T_i=\sup\{ t<\bar{T}_i: \nu\leq 2\}$. We claim  that there is a positive constant $\bar{T}$ independent of $i$  such that  $T_i\geq \bar{T}$.  The argument is as follows. 

For each $t$, at an interior  maximal (minimal)  point of $f$, we have $\triangle f\leq 0 (\geq 0)$ respectively, $\nu=1$ and $div_{M_t}T=\bar{H}$.  From  (\ref{CDP}), we have  
\begin{equation}\label{fbdt} 
|f(P,t)|\leq \max_{\tau}|\bar{H}| t
\end{equation}
for all $P\in \Omega_i$, $t\in [0,T_i)$. 

We will use a barrier function technique to estimate the gradient of $f$ on the boundary.

Let domain  $\tilde{\Omega}_{i}=\{\frac{i}{2}<\tilde{r}<i\}$. From  Lemma \ref{l2.2},   for $i>>1$, if $\nu\leq 2$, we have 
\begin{equation}
\frac{1}{100}r^{-1}\leq  tr_{M_t}D^2 \tilde{r} \leq 100  r^{-1},
\end{equation}
which implies 
\begin{equation}
\frac{1}{C}\tilde{r}^{-1-\frac{\epsilon}{2}}\leq  tr_{M_t}D^2 \tilde{r}^{1-\frac{\epsilon}{2}} \leq C  \tilde{r}^{-1-\frac{\epsilon}{2}},
\end{equation}
provided that  $\epsilon$ is not too large.  It follows from  (\ref{ACh}) that 
\begin{equation} \label{ACdivT}
|div_{M_t}T|(P,t) \leq C |\bar{h}|(X(P,t)) \leq C r^{-1-\epsilon}.
\end{equation}

Note that  under the graphical mean  curvature flow  (\ref{CDP}), $\tilde{r}=\tilde{r}(X(p,t))$ satisfies  
\begin{equation}
\frac{\partial \tilde{r}}{\partial
t}=\frac{z^i}{\tilde{r}}\frac{\partial z^i}{\partial\tau} \frac{\partial f}{\partial t}+ \frac{z^i}{\tilde{r}}\frac{\partial z^i}{\partial x^k} \frac{\partial x^k}{\partial t}=\frac{z^i}{\tilde{r}}\frac{\partial z^i}{\partial\tau} \frac{\partial f}{\partial t}\end{equation}
since $\frac{\partial x^k}{\partial t}=0$ holds for the deformation $X=(x^1,x^2,x^3, f(x^1,x^2,x^3,t))$. 

Fix a  $\delta>0$.  On  $\tilde{\Omega}_i$ for $i>>1$,  we have 
\begin{equation}\label{f-i}
\begin{split}
& \frac{\partial}{\partial t}(f-\delta(i^{1-\frac{\epsilon}{2}}-\tilde{r}^{1-\frac{\epsilon}{2}}))\\
&= [1+\delta(1-\frac{\epsilon}{2})\frac{z^i}{\tilde{r}}\frac{\partial z^i}{\partial\tau}r^{-\frac{\epsilon}{2}} ](tr_{M_t} D^2 (f+\delta r^{1-\frac{\epsilon}{2}})-\delta tr_{M_t}D^2 \tilde{r}^{1-\frac{\epsilon}{2}}+div_{M_t}T) \\
&\leq [1+O(r^{-\frac{\epsilon}{2}})](tr_{M_t}D^2 (f+\delta r^{1-\frac{\epsilon}{2}}))-C^{-1}\delta \tilde{r}^{-1-\frac{\epsilon}{2}}+Cr^{-1-\epsilon})\\
&\leq [1+O(r^{-\frac{\epsilon}{2}}) ]tr_{M_t} D^2 (f-\delta(i^{1-\frac{\epsilon}{2}}-\tilde{r}^{1-\frac{\epsilon}{2}})).
\end{split}
\end{equation}

Analogously, we have 
\begin{equation}\label{f+i}
\frac{\partial}{\partial t}(f+\delta(i^{1-\frac{\epsilon}{2}}-\tilde{r}^{1-\frac{\epsilon}{2}}))
 \geq [1+O(r^{-\frac{\epsilon}{2}}) ]tr_{M_t}D^2 (f+\delta(i^{1-\frac{\epsilon}{2}}-\tilde{r}^{1-\frac{\epsilon}{2}}))
\end{equation}
on $\tilde{\Omega}_i$, for $i>>1$.

On the other hand,  owing to (\ref{fbdt}) and the boundary value in  (\ref{CDP}),   it is easy to see 
\begin{equation}\label{f-ibdy}
f(P,t)-\delta(i^{1-\frac{\epsilon}{2}}-\tilde{r}^{1-\frac{\epsilon}{2}})\mid_{\partial \tilde{\Omega}_i}\leq 0, \ \  f(P,t)-\delta(i^{1-\frac{\epsilon}{2}}-\tilde{r}^{1-\frac{\epsilon}{2}})\mid_{t=0}\leq 0 
\end{equation}
and 
\begin{equation}\label{f+ibdy}
f(P,t)+\delta(i^{1-\frac{\epsilon}{2}}-\tilde{r}^{1-\frac{\epsilon}{2}})\mid_{\partial \tilde{\Omega}_i}\geq  0,\ \ \   f(P,t)+\delta(i^{1-\frac{\epsilon}{2}}-\tilde{r}^{1-\frac{\epsilon}{2}})\mid_{t=0}\geq 0. 
\end{equation}

By Lemma \ref{l2.2},(\ref{f-i})(\ref{f+i})(\ref{f-ibdy})(\ref{f+ibdy}) and maximum principle, we have 
\begin{equation}\label{fbdr}
f(P,t)-\delta(i^{1-\frac{\epsilon}{2}}-\tilde{r}^{1-\frac{\epsilon}{2}})\leq 0,   f(P,t)+\delta(i^{1-\frac{\epsilon}{2}}-\tilde{r}^{1-\frac{\epsilon}{2}})\geq 0 
\end{equation}
on $\tilde{\Omega}_i$. In particular, we have 
\begin{equation}\label{dfbd1}
|\nabla f|=|\frac{\partial f}{\partial n}| \leq \delta(1-\frac{\epsilon}{2}) \tilde{r}^{-\frac{\epsilon}{2}}|\frac{\partial \tilde{r}}{\partial n}| \leq C\delta r^{-\frac{\epsilon}{2}},
\end{equation}
on $\partial \Omega_i$, for all $t\in [0,T_i)$, i.e., we have derived the gradient estimate of $f$ on the boundary.

In order to estimate $||\nabla f||_{L^{\infty}}$, we need to calculate the evolution equation of $\nabla f$. 

Differentiating (\ref{CDP}) we have   
\begin{equation}
\begin{split}
     \frac{\partial}{\partial t}\nabla_if
     =&\nu^{-2}(\triangle \nabla_if-R_{ik}\nabla_kf+ \nabla_i div_{M_t}T)-2\nu^{-3}\nabla_i\nu(\triangle f+div_{M_{t}}T)\\
     =&\nu^{-2}(\triangle \nabla_if-(\bar{R}_{i0k0}+h^{2}_{ik}-Hh_{ik})\nabla_kf+ \nabla_i div_{M_t}T)\\
     &-2\nu^{-3}\nabla_i\nu(\triangle f+div_{M_{t}}T)
\end{split}
\end{equation}
which implies 
\begin{equation}\label{|df|2 evo}
\begin{split}    
    \frac{\partial}{\partial t}|\nabla f|^2=& \nu^{-2}\triangle|\nabla f|^2-2\nu^{-2}|\nabla^2 f|^2-2\nu^{-2}(\bar{R}_{i0k0}+h^{2}_{ik})\nabla_kf\nabla_i f\\
    &+ 2\nu^{-2} \langle \nabla (div_{M_t}T), \nabla f\rangle-4\nu^{-3}(\triangle f+div_{M_{t}}T)\langle\nabla\nu,\nabla f\rangle. 
\end{split}
\end{equation}

Note that  from our assumptions,  $D\bar{R}$ is not necessarily uniformly bounded. In view of (\ref{nor evo}), we can not readily draw the conclusion  that the derivative  $\nabla (div_{M_t}T) $  is uniformly bounded. For this reason,   we adopt the method of  Nash-Moser iteration   to derive  the   $L^{\infty}$  bound of   $|\nabla f|$.

Let $v=(|\nabla f|^2-k)_+$, where $k=\sup_{\partial\Omega_i}|\nabla f|^2$. 
 With the aid of   (\ref{|df|2 evo}),  via  direct computations, we obtain 
\begin{equation} \label{dv}
\begin{split}
     \frac{d}{d t}\int_{\Omega_i\cap M_t}v^q\leq & \int_{\Omega_i\cap M_t} -2q\nu^{-2}v^{q-1}|\nabla^2 f|^2-q(q-1)\nu^{-2}|\nabla v|^2v^{q-2}+2q\nu^{-3}\langle\nabla \nu,\nabla |\nabla f|^2\rangle v^{q-1}\\
     &-2q \int_{\Omega_i\cap M_t}(\bar{R}_{i0k0}+h^{2}_{ik})\nu^{-2}\nabla_kf\nabla_i fv^{q-1}+2q\int_{\Omega_i}\nu^{-2} <\nabla div_{M_t}T,\nabla f>v^{q-1}\\
     &+\int_{\Omega_i\cap M_t}-4q\nu^{-3}(\triangle f+div_{M_{t}}T)\langle\nabla\nu,\nabla f\rangle v^{q-1}+(H^2-div(\nu^{-1}H\nabla f))v^q.
\end{split}
\end{equation}

Since $H= \nu^{-1}(\triangle f+div_{M_t}T)$, for $q\geq 8$, we may employ  integration   by parts in (\ref{dv}) to deduce  
\begin{equation} \label{dv1}
\begin{split}
  \frac{d}{d t}\int_{\Omega_i\cap M_t}v^q
     \leq & \int_{\Omega_i\cap M_t} -\frac{q(q-1)}{2}\nu^{-2}|\nabla v|^2v^{q-2}-q\nu^{-2}|\nabla^2 f|^2v^{q-1}\\
     &+ C\int_{\Omega_i\cap M_t}q \nu^{-2}|\bar{R}||\nabla f|^2v^{q-1}+q\nu^{-2}v^{q-1}|\bar{h}|^2 +q(q-1)\nu^{-2}|\nabla f|^2v^{q-2} |\bar{h}|^2\\
     &+\int_{\Omega_i\cap M_t}q\nu^{-4}|\nabla v|^2v^{q-1}+2q\nu^{-4}\langle\nabla v,\nabla f\rangle\triangle f v^{q-1}\\
     \leq &\int_{\Omega_i\cap M_t}-C^{-1}|\nabla v^{\frac{q}{2}}|^2+Cq|\bar{R}|v^{q-1}+Cq^2|\bar{h}|^2v^{q-2}.
\end{split}
\end{equation}

In view of  $\nu\leq 2$ and  $|\nabla f|^2\leq 3$, we have
\begin{equation} \label{dv2}
\frac{d}{d t}\int_{\Omega_i\cap M_t}v^q
     \leq  -C^{-1}||\nabla v^{\frac{q}{2}}||_{L^2}^2 +Cq^2|||\bar{R}|+|\bar{h}|^2||_{L^{\frac{q}{2}}}  ||v||_{L^q}^{q-2}.
\end{equation}
Integrating (\ref{dv2}) from $0$ to $T_i$ yields 
\begin{equation} \label{dv3}
\begin{split}
 ||v||_{L^q(\Omega_i\cap M_t)}^2 \leq Cq \int_{0}^{t}|||\bar{R}|+|\bar{h}|^2||_{L^{\frac{q}{2}}(\Omega_i\cap M_t)} dt. 
 \end{split}
\end{equation}

Employing the Sobolev inequality (see (\ref{StSob})) and integrating (\ref{dv2}) from $0$ to $T_i$ again,  we obtain 
\begin{equation}\label{v Moser}
\begin{split}
(\int_0^{T_i}\int_{\Omega_i\cap M_t}|v^{\frac{q}{2}}|^{\frac{10}{3}})^{\frac{3}{5}}& \leq Cq^2\int_0^{T_{i}}\int_{\Omega_i\cap M_t}(|\bar{R}|+|\bar{h}|^2)v^{q-2}\\
&\leq Cq^2T_i^{\frac{2}{q}} \max_{\tau}|||\bar{R}|+|\bar{h}|^2||_{L^{2}}^{\frac{4}{q}}\max_{\tau}|||\bar{R}|+|\bar{h}|^2||_{L^{\infty}}^{1-\frac{4}{q}}
(\int_0^{T_{i}}\int_{\Omega_i\cap M_t}v^{q})^{\frac{q-2}{q}}\\
&\leq Cq^2T_i^{\frac{2}{q}} \max_{\tau}(|||\bar{R}|+|\bar{h}|^2||_{L^{2}}+|||\bar{R}|+|\bar{h}|^2||_{L^{\infty}})
(\int_0^{T_{i}}\int_{\Omega_i\cap M_t}v^{q})^{\frac{q-2}{q}}.
\end{split}
\end{equation}

Note that our purpose is to bound  $T_i$ from below,  so we may assume $T_i\leq 1$. 

 Then (\ref{v Moser}) yields
\begin{equation}
\begin{split}
(\int_0^{T_i}\int_{\Omega_i\cap M_t}|v|^{\frac{5q}{3}})^{\frac{3}{5q}}\leq& \bar{C}^{\frac{1}{q}}q^{\frac{2}{q}}
(\int_0^{T_{i}}\int_{\Omega_i\cap M_t}v^{q})^{\frac{1}{q}-\frac{2}{q^2}},\\   
\end{split}
\end{equation}
where $\bar{C}=C\max_{\tau}(|||\bar{R}|+|\bar{h}|^2||_{L^{2}}+|||\bar{R}|+|\bar{h}|^2||_{L^{\infty}})+1\geq 1$.

Let $q_l=(\frac{5}{3})^l q$, $\beta_l=(\int_0^{T_i}\int_{\Omega_i\cap M_t}|v|^{q_l})^{q_l^{-1}}$, $y_l=\log \beta_l$, we have 
\begin{equation} \label{2.31}
    \beta_{l+1}\leq \bar{C}_l\beta_{l}^{1-2q_l^{-1}},\ \ \ \bar{C}_l=\bar{C}^{\frac{1}{q_l}}q_l^{\frac{2}{q_l}},
\end{equation}
and 
\begin{equation} \label{iter}
\begin{split}
    y_{l+1}\leq & \log \bar{C}_l+(1-\frac{2}{q_l})y_{l}\\
    \leq &\prod_{k=0}^{l}(1-(\frac{3}{5})^k2q^{-1})y_0+ \sum_k [\prod_{s=0}^{l-k}(1-(\frac{3}{5})^s2q^{-1})](\frac{\log \bar{C}+2\log q_k}{q_k}).
    \end{split}
\end{equation}
By letting $q=16$ in (\ref{dv3}), we have 

\begin{equation}
||v||^2_{L^{16}}\leq C t \max_{\tau} |||\bar{R}|+|\bar{h}|^2||_{L^8}
\end{equation}
and 
\begin{equation} \label{b0}
\beta_0\leq CT_i^{\frac{9}{16}} \max_{\tau} |||\bar{R}|+|\bar{h}|^2||^{\frac{1}{2}}_{L^8}\leq 1
\end{equation}
 if $T_i$ is not too large.

It is not hard to see 
\begin{equation} \label{2.35}
  \sum_k [\prod_{s=0}^{l-k}(1-(\frac{3}{5})^s2q^{-1})](\frac{\log \bar{C}+2\log q_k}{q_k})\leq  \sum_k (\frac{\log \bar{C}+2\log q_k}{q_k})\leq\frac{5}{2q}(\log \bar{C}+2\log q),
\end{equation}
and 
\begin{equation} \label{2.36}
    e^{-10q^{-1}}\leq \prod(1-(\frac{3}{5})^l2q^{-1})\leq 1.
\end{equation}
From (\ref{iter}) (\ref{b0}) (\ref{2.35}) and (\ref{2.36}), it  follows that 
\begin{equation}
   y_{l+1}\leq e^{-\frac{5}{8}} y_0+\frac{5}{32} \log \bar{C}+\frac{5}{4}\log 2.
\end{equation}
Let $l\rightarrow \infty$, we have 
\begin{equation} \label{2.38}
    ||v||_{L^{\infty}}\leq \bar{C}^{\frac{5}{32}}2^{\frac{5}{4}}(\int_{0}^{T_i}\int_{\Omega_i\cap M_t}v^{16} )^{\frac{e^{-\frac{5}{8}}}{16}},
\end{equation}
and 
\begin{equation}
\begin{split}
         ||\nabla f||_{L^{\infty}}^2&\leq \sup_{\partial_s\Omega_i}|\nabla f|^2+C\bar{C}^{\frac{5}{32}}(T_i^{\frac{9}{16}}\bar{C}^{\frac{1}{2}})^{e^{-\frac{5}{8}}}\\
         &\leq C\delta i^{-\epsilon}+C T_i^{\frac{9}{16}e^{-\frac{5}{8}}} \bar{C}^{\frac{5}{32}+\frac{1}{2}e^{-\frac{5}{8}}}.
\end{split}
\end{equation}
Therefore, 
\begin{equation}\label{nubd1}
\nu=\sqrt{1+|\nabla f|^2}\leq 1+ C\delta i^{-\epsilon}+C T_i^{\frac{9}{16}e^{-\frac{5}{8}}} \bar{C}^{\frac{5}{32}+\frac{1}{2}e^{-\frac{5}{8}}}\leq \frac{3}{2}\end{equation}
provided 
\begin{equation}\label{bart} T_i\leq \bar{C}^{-\frac{8}{9}-\frac{5}{18}e^{\frac{5}{8}}}(4C)^{-\frac{16}{9}e^{\frac{5}{8}}}, \ \ \  C\delta i^{-\epsilon}\leq \frac{1}{4}.
\end{equation}

Combining (\ref{nubd1}) (\ref{bart}) and  (\ref{fbdt}), we  conclude that 
\begin{equation} \label{2.48}
T_i\geq \bar{T}=\min\{\bar{C}^{-\frac{8}{9}-\frac{5}{18}e^{\frac{5}{8}}}(4C)^{-\frac{16}{9}e^{\frac{5}{8}}}, C^{-1}\bar{C}^{-1} \delta_1\}.
\end{equation}

(\ref{nubd1}) says that the equation (\ref{CDP}) is  uniformly parabolic (w.r.t. $\Omega_i$) in $[0, \bar{T}]$.   By standard  regularity estimates of parabolic equations and compactness theorem, one can extract a convergent subsequence $f_{i_k}$ of solutions $f_i$ over compact subsets of $M_0$ and over time interval   $[0,\bar{T}]$.  The limit $f=\lim f_{i_k}$  is a solution to the graphical mean curvature flow (\ref{mcfg}) defined on  $M_0\times [0,\bar{T}]$. 
\end{proof}

\begin{remark}  One can show  that  $\sum_{i=0}^2||\bar{\nabla}^i \bar{R}m||_{L^2(\Sigma_{\tau})}$ is bounded for each time $\tau\in [-\delta_1, \delta_1]$.   By differentiating (\ref{nor evo}) one or two times and using (\ref{tau}), we obtain $\sum_{i=1}^2||\bar{\nabla}^i \bar{h}||_{L^2(\Sigma_{\tau})}$ is bounded for each $\tau$.  In view of   $\nu\leq 2$  on the time interval $[0,\bar{T}]$,  we know $\sum_{i=0}^2||\nabla^i \bar{R}m||_{L^2(M_t)}$ is also bounded for each time $t\in [0,\bar{T}]$.

\end{remark}

\section{Asymptotic  estimates at infinity}\label{S2.3}

The aim of this section is to show that the solution to the graphical mean curvature flow is  sufficiently close to the initial data at infinity.

\begin{theorem}\label{t2.4}
    There exist some  large $\hat{r}_{0}$ and  small $\sigma_0>0$ such that  

    \begin{equation}\label{ACf}
        |f|(x,t)\leq Cr^{-\sigma_0},
    \end{equation}
when   $r(x)\geq \hat{r}_0$, for all $t\in [0, \bar{T}]$.
   \end{theorem}

\begin{proof}  The idea is to find a barrier function in the asymptotically flat region. Motivated by  Bartnik's construction (\cite{B84}, \S5), we consider a positive function $w=w(x)$ satisfying the differential equation 
\begin{equation}\label{253}
    x^{-2}(\frac{{x}^{2}w'}{\sqrt{1-w'^{2}}})'=-{x}^{-2-\sigma},
\end{equation}
which is the hight function  of a  graphical spacelike submanifold   with negative mean curvature ($-r^{-2-\sigma}$) in Minkowski space. 

A solution of (\ref{253}) can be given explicitly by the formula
\begin{equation}\label{254}
    w(x)=\int_{x}^{\infty}\frac{s^{1-\sigma}-{r}_0^{1-\sigma}}{\sqrt{(1-\sigma)^2s^4 +({r}_0^{1-\sigma}-s^{1-\sigma})^2}}ds
\end{equation}
on $[{r}_0, \infty).$
 One can prove that $w(x)$ is convex  on $[e {r}_0, \infty)$, non-increasing on  $[{r}_0, \infty)$ and    satisfies 
\begin{equation}\label{wbd}
     \frac{1}{\sqrt{2}(1-\sigma)}[\sigma^{-1}x^{-\sigma}-{r}_0^{1-\sigma}{x}^{-1}]\leq w(x)\leq \frac{1}{1-\sigma}[\sigma^{-1}{x}^{-\sigma}-{r}_0^{1-\sigma}{x}^{-1}],
\end{equation}
when ${x}\geq \max\{(1-\sigma)^{-\frac{1}{1+\sigma}},{r}_0\}$.  In particular, 
\begin{equation}\label{256}
    w ({r}_0)\geq\frac{{r}_0^{-\sigma}}{\sqrt{2}\sigma}\rightarrow \infty,  \ \  \text{as} \ \  \sigma\rightarrow 0. 
\end{equation}

Under wave coordinates $\{z^{i},z^{0}\}$, let $\tilde{r}=\sqrt{(z^1)^2+(z^2)^2+(z^3)^2}$, we consider the submanifold $M_w\triangleq  \{z^0=w(\tilde{r})\}$.  Since $C^{-1} \sigma^{-1}r^{-\sigma}\leq\tau\mid_{M_w}\leq C\sigma^{-1}r^{-\sigma}$,   for fixed large $r_0$ and small $\sigma$, there exists a $\tilde{r}_0=\tilde{r}_0(\sigma,{r}_0) \geq \frac{3{r}_0}{2}$ such that $\tau\mid_{M_{w}} \leq \delta_1$ when $\tilde{r}\geq \tilde{r}_0$, and $\max \tau_{M_{w}\cap\{\tilde{r}=\tilde{r}_0\}}=\delta_1$, where $\delta_1$ is the constant in (\ref{tau}).

 Since $w(\tilde{r})$ is non-increasing in $\tilde{r}$ and $\frac{\partial \tau}{\partial z^0}>0$ (Lemma \ref{l2.2}), we know  the part of the graph $\{(z^i, w(\tilde{r})): \tilde{r}\geq \tilde{r}_0\}$  lies  entirely in the region in (\ref{tau}).

Now we calculate the mean curvature $H_{M_w}$ of $M_{w}$.

First of all,  we present the spacetime metric $\bar{g}$ in  wave coordinates $\{z^{\alpha}\}$ as follows 
\begin{equation} \label{2.57}
    \bar{g}= -(\alpha^2-|\beta|^2)(dz^{0})^2+2\beta_idz^{i}dz^0+\bar{g}^{z}_{ij}dz^{i}dz^{j}.
\end{equation}

We  denote  the level sets  $\{z^0=const.\}$ of function $z^0$ by $M_{z^0}$, whose unit future directed normal vector can be given by $Z=\alpha^{-1}(\frac{\partial}{\partial z^{0}}-\beta)$. Then 
 \begin{equation}
 N=\nu(Z+U)
 \end{equation}
 is the unit normal vector of $M_{w}$, where  $U=(1+\langle D^{M_{z^0}} w, \beta \rangle )^{-1}\alpha D^{M_{z^0}} w$, $D^{M_{z^0}}w= (\bar{g}^z)^{ij} \frac{\partial w}{\partial z^i}\frac{\partial}{\partial z^j}$ and $\nu=(1-|U|^2)^{-\frac{1}{2}}$.

A straightforward  calculation gives 
\begin{equation} \label{2.59}
\begin{split}
H_{M_{w}}=&  div_{M_w} N=  div_{\bar{M}} N\\
=& div_{\bar{M}} \frac{U}{\sqrt{1-|U|^2}}+div_{M_{z^0}}( \frac{Z}{\sqrt{1-|U|^2}})-\langle D_{Z}(\frac{Z}{\sqrt{1-|U|^2}}), Z\rangle\\
=&  div_{\bar{M}} \frac{U}{\sqrt{1-|U|^2}}+\frac{H_{M_{z^0}}}{\sqrt{1-|U|^2}}+\frac{1}{2}(1-|U|^2)^{-\frac{3}{2}}Z(|U|^2).\end{split}
\end{equation}

Now we estimate  the mean curvature  $H_{M_{z^0}}$ of $\{z^0=const.\}$ by utilizing the formula
\begin{equation}
\begin{split}
H_{M_{z^0}}&=-{|D z^0|^{-2}} tr_{M_{z^0}} D^2 z^0=-|D z^0|^{-2}(\Box z^0-|Dz^0|^{-2} D^2 z^0(Dz^0,Dz^0))\\
&=(|Dz^0|^2)^{-2}D^2 z^0(Dz^0,Dz^0),
\end{split}
\end{equation}
where we have used $\Box z^0=0$.

Since $D^2 z^0=\partial \bar{g}\ast \bar{g}^{-1}=O(r^{-1-\epsilon})$, $\frac{\partial}{\partial z^0} D^2 z^0=O(r^{-2-\epsilon})$, we have 
\begin{equation}\label{260}
\frac{\partial}{\partial z^0} H_{M_{z^0}}=O(r^{-2-\epsilon})+O(r^{-2-2\epsilon})=O(r^{-2-\epsilon}).
\end{equation}

Integrating (\ref{260}) with  initial value $H_{M_{z^0}}\mid_{z^0=0}=O(r^{-2-\epsilon})$ (see (\ref{AC}))  yields 
\begin{equation}
H_{M_{z^0}}=O(r^{-2-\epsilon}).
\end{equation}

Since $w^{\prime}=O(r^{-1-\sigma})$, $w^{\prime\prime}=O(r^{-2-\sigma})$,  by direct computations, we have   $|U|^2=O(r^{-2-2\sigma})$, $Z(|U|^2)=O(r^{-3-2\sigma})$, $ div_{\bar{M}} \frac{U}{\sqrt{1-|U|^2}}=\frac{1}{\sqrt{1-|U|^2}}div_{\bar{M}}{U}+O(r^{-4-3\sigma})=div_{\bar{M}} {U}+O(r^{-4-3\sigma})$.

Since  $\{z^{\alpha}\}$ are wave coordinates,  we have  \begin{equation}
div_{\bar{M}}U=\bar{g}^{\alpha\beta}(\partial_{\beta}U_{\alpha}-\bar{\Gamma}^{\gamma}_{\beta\alpha}U_{\gamma})= \bar{g}^{\alpha\beta} \partial_{\beta}U_{\alpha}. \end{equation}
Then (\ref{2.59}) becomes 
\begin{equation}
\begin{split}
H_{M_{w}}&= \bar{g}^{\alpha\beta} \partial_{\beta}U_{\alpha}+O(r^{-2-\epsilon})\\
&= \alpha (1+\langle D^{M_{z^0}} w, \beta \rangle )^{-1} \bar{g}^{i\beta} \partial_{\beta}(D^{M_{z^0}}w)_{i}+O(r^{-2-\epsilon})+O(r^{-2-\epsilon-\sigma})\\
&=\partial_{i}\partial_{i}w+O(r^{-2-\epsilon})\\
&=w^{\prime\prime}+ \frac{2}{\tilde{r}}w^{\prime}+O(r^{-2-\epsilon}).
\end {split}
\end{equation}

Meanwhile,  the equation (\ref{253}) implies 
\begin{equation} \label{2.65}
\frac{2}{\tilde{r}}w'+w''= -\tilde{r}^{-2-\sigma}+O(\tilde{r}^{-3-\sigma}).\end{equation}

Consequently, 
\begin{equation}\label{2.66}
H_{M_{w}}= -r^{-2-\sigma}+O(r^{-2-\epsilon})<0
\end{equation}
provided  $\sigma<\epsilon$.

Now we proceed to prove the theorem.  Note that from the proof of Theorem \ref{shorttime}, the solution $f$ is obtained as the limit of a sequence of $f_i$ satisfying the Dirichlet boundary condition (\ref{CDP}).

 \textbf{Claim}:   $f_i(x,t)< \tau\mid_{M_w\cap\bar{M}_1}$ for all $t\in [0,\bar{T}]$,   $\tilde{r}_0\leq \tilde{r}(x)\leq i$.

  Note that on  $\{\tilde{r}=i\}$,  $\tau\mid_{M_w}\geq C^{-1}w(i)>0$.  Suppose the \textbf{Claim} is not true.  There exists a first  time $0<\tilde{T}_i\leq \bar{T}$  such that $f_i(\cdot,\cdot) <\tau\mid_{M_w}$ on  $[0,\tilde{T}_i)$, and $\tau\mid_{M_w}=f_i$ at some point  $P_0$ at time $\tilde{T}_i$.  

If $P_0$ is an interior point of $M_{w}\cap \bar{M}_1$, then the graph of $f_i(\cdot, \tilde{T}_i)$ is tangent to $M_{w}$ at  $P_0$.  Since the mean curvature of $M_{\tilde{T}_i}$ is not greater than  $H_{M_w}$ at $P_0$, the latter is negative according to  (\ref{2.66}), we have    $\frac{\partial f_i}{\partial t}\mid_{\tilde{T}_i}<0$ at $P_0$.   This implies that for slightly earlier time than $\tilde{T}_i$, the graph of  $f_i$ is not entirely lying below $M_w\cap{\bar{M}_1}$.  This contradicts  with the definition of  $\tilde{T}_i$, which is the first time when  the graph of $f_i(\cdot, \tilde{T}_i)$ touches  $M_{w}$ from below.  

Hence $P_0$ lies on the boundary of $M_{w}\cap \bar{M}_1$, i.e. $\tau(P_0)=\delta_1$.  In this case,  $P_0$ is still an interior point of $M_{w}$. Since $f_i\leq \delta_1$, the graph of $f_i(\cdot,\tilde{T}_i)$ is still tangent to $M_{w}$ at $P_0$. This leads to the same contradiction as above.  
 The \textbf{Claim} is proved. 

Let $i\rightarrow \infty$, by taking a subsequence limit, we obtain $f(x,t)\leq \tau\mid_{M_w}$ when $\tilde{r}(x)\geq \tilde{r}_0$ and $t\in [0,\bar{T}]$.

By taking the graph $\{(z^i,-w(\tilde{r}))\}$ as a  barrier  from below, we conclude that  $f(x,t) \geq  \tau\mid_{M_{-w}}$  when $\tilde{r}(x)\geq \tilde{r}_0$ and $t\in [0,\bar{T}]$. The result (\ref{ACf}) can be obtained from the estimate (\ref{wbd}) and  Lemma \ref{l2.2}. 
\end{proof}

\begin{theorem}\label{t2.5}
 There exists  large  $\hat{r}_{0}$  such that   
\begin{equation} \label{316}
\begin{split}
     |\nabla f|(x,t) \leq Cr^{-e^{-\frac{5}{8}}(\frac{2229}{2048}+\frac{21}{32}\epsilon)} \max\{\bar{T},1\}^{\frac{1177}{2048}}
\end{split}
    \end{equation} 
    holds for $t\in [0,\bar{T}]$ and  $r(x)\geq \max\{\hat{r}_0, \sqrt{\bar{T}}\}$. 
 \end{theorem}
\begin{proof} We need to adapt and localize the argument in Theorem \ref{shorttime}.  Let  $\phi$ be a cutoff function  such that $\phi=1$ on $[1,2]$, $\phi=0$ on $[0,\frac{1}{2}]\cup [\frac{5}{2},+\infty)$, and $|\phi'|\leq C \phi^{\frac{3}{4}}$.   From   (\ref{|df|2 evo}), integration by parts and H$\ddot{o}$lder inequality,  we obtain  
\begin{equation} \label{df L^4 evo}
\begin{split}
     \frac{d}{d t}\int_{M_t}\phi(\frac{r}{r_1})|\nabla f|^{q}\leq & \int_{M_t} -q\phi\nu^{-2}|\nabla^2 f|^2|\nabla f|^{q-2}-\frac{1}{4}q(q-2)\phi \nu^{-2}|\nabla|\nabla f|^2|^2|\nabla f|^{q-4}\\
     &-\int_{M_t}\nu^{-2}\langle\nabla \phi(\frac{r}{r_1}),\nabla|\nabla f|^{q}\rangle+ \int_{M_t}\frac{q}{2}\phi\nu^{-4}|\nabla|\nabla f|^2|^2|\nabla f|^{q-2}\\
     &+ \int_{M_t}C q\phi\nu^{-2}|\bar{R}||\nabla f|^{q}+q\phi\nu^{-2}|\nabla f|^{q-2}\langle \nabla div(T), \nabla f\rangle\\
     &-\int_{M_t} q \phi \nu^{-4}\langle \nabla |\nabla f|^2,\nabla f\rangle |\nabla f|^{q-2} (\triangle f+div_{M_t}T)\\
     \leq &\int_{M_t}-C^{-1}\phi|\nabla |\nabla f|^{\frac{q}{2}}|^2+\frac{Cq}{r_1^2}(\int_{M_t}|\nabla f|^{q})^{\frac{2}{q}}(\int_{M_t}\phi|\nabla f|^{q})^{\frac{q-2}{q}}\\
     &+C q^2(\int_{M_t}\phi(|\bar{R}|^{\frac{q}{2}}+|\bar{h}|^{q}))^{\frac{2}{q}}(\int_{M_t}\phi|\nabla f|^{q})^{\frac{q-2}{q}}\\
     \leq & Cq^2r_{1}^{-2+\frac{6}{q}-\epsilon}(\int_{M_t}\phi|\nabla f|^{q})^{\frac{q-2}{q}}.
\end{split}
\end{equation}
Integrating (\ref{df L^4 evo}) yields
\begin{equation} \label{2.50}
    (\int_{B_{2r_1}-B_{r_1}}\phi(\frac{r}{r_1})|\nabla f|^{q})^{\frac{2}{q}}\leq Cq\bar{T}r_{1}^{-2+\frac{6}{q}-\epsilon}.
\end{equation}
Fix a point  $x_1 \in M_t$ with  $ r(x_1)\triangleq 10 r_1\geq 100 r_0 $. Let $R_{0}=\frac{6r_1}{5},  R_{l}=\frac{3r_1}{5}+(\frac{3}{5})^{l+1}r_1$,  $\xi_l$ be a  cutoff function with  $\xi_{l}=1$ on $B(x_0,R_{l+1})$,  and $\xi_l=0$ on  $B(x_1,\frac{R_{l}+R_{l+1}}{2})^{c}$,  $|\nabla \xi_{l}|^{2}\leq \frac{C\xi_l^{\frac{3}{2}}}{(R_{l}-R_{l+1})^{2}}$. Then the same argument as in (\ref{dv})-(\ref{v Moser}) gives 
\begin{equation}
\begin{split}
& (\int_0^{\bar{T}}\int_{M_t}\xi_l|\nabla f|^{\frac{5q}{3}})^{\frac{3}{5}}\\&  \leq Cq^2\int_0^{\bar{T}}\int_{M_t}\xi_l(|\bar{R}|+|\bar{h}|^2)|\nabla f|^{q-2}+C\int_0^{\bar{T}}\int_{M_t}(R_l-R_{l+1})^{-2}\xi_l^{\frac{1}{2}}|\nabla f|^q\\
& \leq Cq^2\bar{T}^{\frac{2}{q}} r_1^{\frac{6}{q}-2-\epsilon}(\int_0^{\bar{T}}\int_{M_t}\xi_l|\nabla f|^{q})^{\frac{q-2}{q}}+C(\frac{5}{3})^{2l}r_1^{-2}\int_0^{\bar{T}}\int_{M_t}\xi_l^{\frac{1}{2}}|\nabla f|^q,
\end{split}
\end{equation}
which deduces  
\begin{equation} \label{2.52}
\begin{split}
       &  (\int_0^{\bar{T}}\int_{B(x_1,R_{l+1})}|\nabla f|^{\frac{5q}{3}})^{\frac{3}{5}}\\ &  \leq  Cq^2\bar{T}^{\frac{2}{q}} r_1^{\frac{6}{q}-2-\epsilon}(\int_0^{\bar{T}}\int_{B(x_1,R_l)}|\nabla f|^{q})^{\frac{q-2}{q}} +C(\frac{5}{3})^{2l}r_1^{-2}\int_0^{\bar{T}}\int_{B(x_1,R_l)}|\nabla f|^q\\
       & \leq  [Cq^2\bar{T}^{\frac{2}{q}} r_1^{\frac{6}{q}-2-\epsilon}+ Cq (\frac{5}{3})^{2l}\bar{T}^{1+\frac{2}{q}}  r_1^{\frac{6}{q}-2-\epsilon}r_1^{-2}](\int_0^{\bar{T}}\int_{B(x_1,R_l)}|\nabla f|^{q})^{\frac{q-2}{q}}       
       \end{split}
\end{equation}
taking (\ref{2.50}) into account. Let $q=q_l=16\times (\frac{5}{3})^l$ in (\ref{2.52}),   we have 
\begin{equation} \label{2.53}
      (\int_0^{\bar{T}}\int_{B(x_1,R_{l+1})}|\nabla f|^{\frac{5q_l}{3}})^{\frac{3}{5q_l}} \leq C^{\frac{1}{q_l}}(\frac{5}{3})^{\frac{l}{q_l}}(\bar{T})^{\frac{2}{q^2_l}}q_l^{\frac{2}{q_l}}r_1^{\frac{6}{q_l^2}-\frac{(2+\epsilon)}{q_l}}(\int_0^{\bar{T}}\int_{B(x_1,R_l)}|\nabla f|^{q_l})^{\frac{q_l-2}{q_l^2}}.
\end{equation}

Iterating (\ref{2.53}) and using same  argument from (\ref{2.31}) to (\ref{2.38}), we obtain: 
\begin{equation} \label{2.54}
\begin{split}
     ||\nabla f||_{L^{\infty}(B(x_1,\frac{3r_1}{5}))}&\leq Cr_1^{e^{-\frac{5}{8}}[\frac{75}{2048}-\frac{5}{32}(2+\epsilon)]}\max\{\bar{T},1\}^{\frac{25}{2048}}(\int_{0}^{\bar{T}}\int_{B(x_1,\frac{6r_1}{5})}|\nabla f|^{16} )^{\frac{e^{-\frac{5}{8}}}{16}}\\
     &\leq Cr_1^{-(\frac{2229}{2048}+\frac{21}{32}\epsilon)e^{-\frac{5}{8}}}\max\{\bar{T},1\}^{\frac{25}{2048}+\frac{9}{16}e^{-\frac{5}{8}}},
\end{split}
    \end{equation}
provided $\int_{0}^{\bar{T}}\int_{B(x_1,\frac{6r_1}{5})}|\nabla f|^{16}\leq 1$.  If $\int_{0}^{\bar{T}}\int_{B(x_1,\frac{6r_1}{5})}|\nabla f|^{16}\geq  1$,  the  iteration argument deduces  
\begin{equation} \label{2.55}
\begin{split}
     ||\nabla f||_{L^{\infty}(B(x_1,\frac{3r_1}{5}))}
     &\leq Cr_1^{-(\frac{565}{2048}+\frac{5}{32}\epsilon)e^{-\frac{5}{8}}-\frac{13}{16}-\frac{\epsilon}{2}}\max\{\bar{T},1\}^{\frac{1177}{2048}}.
\end{split}
    \end{equation}
The  combination of   (\ref{2.54})  and (\ref{2.55}) gives  (\ref{316}).  This completes the proof.

\end{proof}

\section{A fundamental  mean curvature estimate}

From this section onward, we will derive several a priori estimates for the mean curvature. 

To begin with,  we observe that   the following    mean curvature estimate is fundamental.  This estimate  may have been noted in previous literature  in different circumstances (e.g. \cite{E93}).

\begin{proposition}\label{P41} Let $f(x,t)$ be a solution to (\ref{mcfg}) on $[0,\bar{T})$. Then    \begin{equation}\label{41}
    \mathop{sup}\limits_{M_{t}}|H| \leq  \frac{1}{\sqrt{\frac{2t}{3}+\frac{1}{\sup\limits_{M_0}|H|^2}}}.
    \end{equation}
\end{proposition}
  To accomplish the proof by virtue of the  maximum principle,  we need to calculate the evolution equations of the metric $g_{ij}$,  the second fundamental form $h_{ij}$ and the  mean curvature $H$,  under  the deformation (\ref{mcfg}).  These evolution equations are standard, but for completeness, we will  give a brief proof.
  
  \begin{proposition} \label{P42}
Let $f(x,t)$ be a solution to (\ref{mcfg}) on $[0,\bar{T})$.  Then 
\begin{equation} \label{42}
\begin{split}
 \frac{\partial g_{ij}}{\partial t} =&2Hh_{ij}-\nabla_i(\nu^{-1}H\nabla_jf)- \nabla_j(\nu^{-1}H\nabla_if)\\
 \frac{\partial h_{ij}}{\partial t}=& \nabla_{i}\nabla_j H+H(h^2_{ij}-\bar{R}_{i0j0})-\nabla_{k}h_{ij}\nu^{-1}H \nabla_kf-h_{kj}\nabla_{i}(\nu^{-1}H\nabla_kf)\\ & -h_{kj}\nabla_{i}(\nu^{-1}H\nabla_kf)\\
\frac{\partial H}{\partial t}=& \triangle H -H|h|^{2}-\nu^{-1}H \langle \nabla H,\nabla f\rangle. 
\end{split}
\end{equation}
\end{proposition}

\begin{proof}
Let $X(\cdot,t)=(x^1,x^2,x^3, f(x^1,x^2,x^3,t))$ be the graph  map of $M_t$, where $\{x^i\}$ is local coordinate system  of $M_0$. Then  
\begin{equation} \label{43}
\frac{\partial X}{\partial t}=\frac{\partial f}{\partial t}\frac{\partial}{\partial \tau}=HN-\nu^{-1}H \nabla f.
\end{equation}

Denote $X_i=\frac{\partial X}{\partial x^i}$, $X_t=\frac{\partial X}{\partial t}$, $g_{ij}=\bar{g}(X_i,X_j)$, $h_{ij}=\bar{g}(D_{X_i}N,X_j)$, $H=g^{ij}h_{ij}$. Then 
\begin{equation}\begin{split} \frac{\partial g_{ij}}{\partial t}&=\langle D_{X_t}X_i,X_j\rangle+\langle X_i,  D_{X_t}X_j\rangle\\
&=\langle D_{X_i}(HN-\nu^{-1}H\nabla f),X_j\rangle+\langle X_i,  D_{X_j}(HN-\nu^{-1}H\nabla f)\rangle\\
&=2Hh_{ij}-\nabla_i(\nu^{-1}H\nabla_jf)- \nabla_j(\nu^{-1}H\nabla_if),\end{split}
\end{equation}
which gives the first formula. 
By straightforward computations,  
\begin{equation} \label{44}
\begin{split} 
\frac{\partial h_{ij}}{\partial t}=&\langle D_{X_t} D_{X_i}N,X_j\rangle
+\langle D_{X_i}N,  D_{X_j} X_t\rangle\\
=&-\bar{R}(HN-\nu^{-1}H\nabla f, X_i, N,X_j)+\langle D_{X_i}N,  D_{X_j} (HN-\nu^{-1}H\nabla f)\rangle\\
&+\langle D_{X_i} D_{X_t}N, X_j\rangle.\end{split}
\end{equation}
From  $\langle D_{X_t} N,X_j\rangle=-\langle N, D_{X_j} (HN-\nu^{-1}H\nabla f)\rangle=X_j(H)-h_{jk}\nu^{-1}H\nabla_kf$,  we  know  $D_{X_t} N=\nabla H-h(\nu^{-1}H\nabla f, \cdot)$. Substituting it into (\ref{44}) and employing Codazzi equation $\bar{R}_{ijk0}=\nabla_ih_{jk}-\nabla_jh_{ik}$,  the second  equation in (\ref{42}) follows. The third  equation  is obtained by  taking trace on the second equation.  
\end{proof}

\begin{proof} of Proposition \ref{P41}. From (\ref{42}), one can prove 
\begin{equation}\label{H2evo}
\begin{split}
(\frac{\partial }{\partial t}-\triangle)H^{2}& =-2|\nabla H|^{2}-2H^{2}|h|^{2}-\nu^{-1}H\langle \nabla H^2,\nabla f\rangle\\
&\leq -2|\nabla H|^{2}-\frac{2}{3}H^{4}-\nu^{-1}H\langle \nabla H^2,\nabla f\rangle,
\end{split}
\end{equation}
 where we have used the inequality  $|h|^{2}\geq \frac{H^{2}}{3}$. 
 
Let $\xi$ be a smooth cut-off function on $\mathbb{R}$ with $0\leq \xi\leq 1$,  which is equal to 1 on $[-\frac{1}{2}, \frac{1}{2}]$,  $0$ outside $[-1,1]$.  For any large  $a$, we consider the function $u=\xi(\frac{\tilde{r}}{a})H^2$. For each time $t$,  the maximum of $u$ must  be achieved at some interior point $P\in \{\tilde{r}<a\}$. The  maximum principle asserts that  $\triangle u(P,t)\leq 0$ and $\nabla u(P,t)=0$.  Now  we compute the evolution equation of $u$ at $(P,t)$:
\begin{equation}\label{47}
\begin{split}
(\frac{\partial }{\partial t}-\triangle)u&=H^2(\frac{\partial }{\partial t}-\triangle)\xi +\xi (\frac{\partial }{\partial t}-\triangle)H^2-2 \nabla \xi \cdot \nabla H^2\\
&\leq H^2(a^{-1}\xi^{\prime}(\frac{\partial}{\partial t}-\triangle)\tilde{r}-a^{-2}\xi^{\prime\prime}|\nabla \tilde{r}|^2 )-\frac{2}{3} \xi H^4-\nu^{-1}H\xi \langle \nabla H^2,\nabla f\rangle+2H^2 \frac{|\nabla \xi|^2}{\xi}\\
&\leq-\frac{2}{3}\xi H^4+\nu^{-1}H^3 \langle \nabla\xi,\nabla f\rangle+C a^{-2-\epsilon} |\xi^{\prime}||H|^3+C a^{-2}(|\xi^{\prime}|+|\xi^{\prime\prime}|+(\xi^{\prime})^2\xi^{-1})H^2 \\
&\leq - \frac{2}{3}\xi H^4+ Ca^{-1} \xi^{\frac{3}{4}}|H|^3+Ca^{-2}\xi^{\frac{1}{2}}H^2,
\end{split}
\end{equation}
where we have used  $|\nabla \tilde{r}|+\tilde{r}|\triangle \tilde{r}| \leq C$, $|\xi^{\prime}|+|\xi^{\prime\prime}|+(\xi^{\prime})^2\xi^{-1} \leq C\xi^{\frac{3}{4}}$,  and $|\frac{\partial \tilde{r}}{\partial t}|\leq Cr^{-1-\epsilon}\nu^{-1}|H|$. The last inequality  follows from the chain rule  $$\frac{\partial \tilde{r}}{\partial t}=\frac{z^{i}}{\tilde{r}}\frac{\partial z^{i}}{\partial \tau} \frac{\partial f}{\partial t}.$$

Applying H$\ddot{o}$lder inequality in (\ref{47}), for any $0<\delta<1$,  we have 
\begin{equation}
\begin{split}
(\frac{\partial }{\partial t}-\triangle)u
\leq - \frac{2-2\delta}{3}\xi H^4- ( \frac{\delta}{3}u^2-C\delta^{-3}a^{-4}).
\end{split}
\end{equation}

One of the followings must be true:

i) $\frac{\delta}{3}u^2(P,t)\leq C \delta^{-3}a^{-4}$;  ii) $\frac{\delta}{3}u^2(P,t)> C\delta^{-3}a^{-4}$ and $\frac{\partial}{\partial t}u(P,t)\leq -\frac{2(1-\delta)}{3}u^2(P,t).$

Let $u_{\max}(s)=\sup_{\{\tilde{r}<a\}} u(\cdot,s)$, define $\frac{d^{+}}{ds}u_{\max}\triangleq \limsup_{h\searrow 0} \frac{u_{\max}(s+h)-u_{max}(s)}{h}$. In the case ii),  we have $\frac{d^{+}}{dt}u_{\max}(t)\leq -\frac{2(1-\delta)}{3}u^2_{\max}(t)$ (see Lemma 3.5 in \cite{H86}). Integrating this   differential inequality we obtain 
\begin{equation}\label{u est}
u_{\max}(s)\leq \max\{ \frac{\sqrt{3C}}{\delta^2 a^2},  \frac{1}{\frac{2(1-\delta)s}{3}+\frac{1}{u_{\max}(0)}}\}.
\end{equation}
The result (\ref{41}) is  obtained by first letting  $a\rightarrow \infty$, then letting  $\delta\rightarrow 0$. 
\end{proof}

Now we are in a position to find a  solution to (\ref{mcf}) from graphical mean curvature flow (\ref{mcfg}).  

  Let $\phi(x,t)$ be a family of diffeomorphisms of $M_0$ generated from the vector field $-\nu^{-1}H \nabla f$,  with $\phi(\cdot,0)=id$,  i.e. 
\begin{equation} \label{410}
\frac{\partial \phi^i(x,t)}{\partial t}= \nu^{-1} H g^{ik} f_k, \ \ \ \phi^i\mid_{t=0}=x^i
\end{equation}
in local coordinates $\{x^i\}$. Because of Proposition \ref{P41} and Theorem \ref{t2.5}, the equation (\ref{410}) can always be solved on $M_0\times [0,\bar{T}]$.  Let $\tilde{X}(x,t)=X(\phi(x,t), t)$,  from (\ref{43}) and direct computations, we obtain 
\begin{equation} \label{411}
\frac{\partial \tilde{X}}{\partial t} = \overset{\rightarrow}{H}, \ \ \ \tilde{X}\mid_{t=0}=X_0.
\end{equation}

That is to say, $\tilde{X}(x,t)$ is a solution to  (\ref{mcf}). Under the deformation (\ref{411}), we have 
\begin{proposition}\label{P43} Let $g_{ij}$, $h_{ij}, H$ be the induced metric, the  second fundamental form and the mean curvature of deformation (\ref{411}) respectively.  Then  
\begin{equation} \label{412}
\begin{split}
& \frac{\partial g_{ij}}{\partial t} =2Hh_{ij},\\
& \frac{\partial h_{ij}}{\partial t} =\nabla_{i}\nabla_j H+H(h^2_{ij}-\bar{R}_{i0j0}),\\
&(\frac{\partial }{\partial t}-\triangle)H =-H|h|^{2}.
\end{split}
\end{equation}
\end{proposition}

Proposition \ref{P43} can be readily deduced from Proposition \ref{P42}. 

For simplicity of notations, in the subsequent sections, we still use $X(x,t)$ to denote the solutions to (\ref{mcf}).

  Note that under (\ref{411}) or (\ref{mcf}), we have $\frac{\partial \tilde{r}}{\partial t}=HN(\tilde{r} )$, and     
\begin{equation}
    (\partial_t-\triangle)\tilde{r}=-tr_{M_t}D^2\tilde{r}. 
\end{equation}

Let $\xi$ be a smooth cut-off function on $\mathbb{R}$ with $0\leq \xi\leq 1$,  which is equal to 1 on $[1, 2]$,  $0$ outside $[\frac{1}{2},\frac{5}{2}]$.  We observe that the calculations in (\ref{47}) also hold  for $u=\xi(\frac{\tilde{r}}{a})H^2$, i.e., 

\begin{equation}
    (\frac{\partial }{\partial t}-\triangle)u\leq - \frac{2}{3}\xi H^4+ Ca^{-1} \xi^{\frac{3}{4}}|H|^3+Ca^{-2}\xi^{\frac{1}{2}}H^2. 
\end{equation}

Applying  the maximum  principle as before, we arrive at the estimate  

\begin{equation}
\begin{split}
\sup_{\{a \leq \tilde{r}\leq 2a \}}|H|(\cdot,t) \leq & \max\{C{a}^{-1}, (\frac{2t}{3}+\frac{1}{\sup_{\{\frac{a}{2} \leq \tilde{r}\leq \frac{5a}{2} \}}|H|^2(\cdot,0)})^{-\frac{1}{2}}\}\\
\leq & \max\{{C}{a}^{-1}, \sup |H|(\cdot,0)\mid_{\{\frac{a}{2} \leq \tilde{r}\leq \frac{5a}{2} \}}\},
\end{split}
\end{equation}
where $C$ is a universal constant.

Combining with the initial data (\ref{AC}), we obtain:

\begin{theorem}\label{t44} The mean curvature of the solution $X(x,t)$ to (\ref{mcf}) satisfies  \begin{equation}\label{416}
    \sup_{\{a \leq \tilde{r}\leq 2a \}}|H|(\cdot,t)\leq \frac{C}{a},
    \end{equation}
when $a$ is large, for all $t\in [0,\bar{T})$.  
\end{theorem}

\section{Bootstrap assumptions}

A  meaningful  observation is  that the Cauchy problem  for the vacuum Einstein equation is scaling invariant.  More explicitly,  if $(M_0,g,h)$ is a initial Cauchy surface  satisfying the constraint equations (\ref{ECE}), and $(\bar{M}, \bar{g})$ is a  vacuum spacetime developed from $(M_0,g,h)$, then for any $\lambda>0$, $(M_0, \lambda g, \sqrt{\lambda}h)$  satisfies   the constraint equations  and $(\bar{M}, \lambda \bar{g})$ is the corresponding  vacuum spacetime developed from  $(M_0, \lambda g, \sqrt{\lambda}h)$.

Moreover,  if $X(x,t)$ is the solution to the mean curvature flow (\ref{mcf}), let $\tilde{t}=\lambda t$, then $\tilde{X}(x,\tilde{t})=X(x,\lambda^{-1}\tilde{t})$ solves the same  mean curvature flow equation (\ref{mcf}) in  the ambient manifold $(\bar{M}, \lambda \bar{g})$. 

Based on this observation, 
we will first prove Theorem \ref{t1.5} when initial data are  suitably small.  Namely, we will prove the following: 

\begin{theorem} \label{max exist}
    Under the assumptions of Theorem \ref{maxfol}, if    
\begin{equation}\label{d2gM0}
\begin{split}
        \int_{M_0}|\bar{R}m|^2+|\nabla h|^2+|h|^4<\epsilon_0^2,
\end{split}
\end{equation}
\begin{equation}\label{d3gM0}
\begin{split}
 \int_{M_0}|\nabla \bar{R}m|^2+|\nabla^2 h|^2+|h|^6+|\bar{R}m|^3+|\nabla h|^3< \epsilon_0^2,
\end{split}
\end{equation}

  then there exist a vacuum  spacetime $(\bar{M},\bar{{g}})\supset M_0$, so that the mean curvature flow (\ref{mcf}) admits a long time solution $X_t$ in $\bar{M}$, which converges to a maximal spacelike hypersurface in $\bar{M}$ as $t\rightarrow \infty$.     

\end{theorem}

From  Section 2.2, one may solve the Cauchy problem for the vacuum Einstein equation so that the spacetime admits a foliation (\ref{nor g}) for $\tau\in [-\delta_1, \delta_1]$.  Moreover, one can  solve the mean curvature flow equation  (\ref{mcf})  for a short time interval $[0,T]$ for some $T>0$.  Let $M_t=X_t(M_0)$, $\mathcal{M}_t=\cup_{s\in[0,  t]}M_{s}$.

We introduce the following bootstrap assumptions  

\begin{equation}\label{bootd1g}
      \frac{1}{2} g_{ij}(x,0)\leq g_{ij}(x,t)\leq 2 g_{ij}(x,0),
\end{equation}
\begin{equation}\label{bootd2g}
        \int_{M_t}|\bar{R}m|^2+|\nabla h|^2+|h|^4\leq  P^2 E_1^2,
\end{equation}
\begin{equation}\label{bootd3g}
    \int_{M_t}|\nabla \bar{R}m|^2+|\nabla^2 h|^2+|h|^6+|\bar{R}m|^3+|\nabla h|^3\leq P^2E_2^2
\end{equation}
on $[0,\hat{T})$ for some $\hat{T}>0$ and  $P>1$.

   One can use $E^2_1(M_t)$ and $E^2_2(M_t)$ to denote the left hand side quantities in (\ref{bootd2g})  and (\ref{bootd3g}) respectively. We will prove eventually that   $\frac{d}{dt}E^2_1(M_t)$ and $\frac{d}{dt}E^2_2(M_t)$ are bounded from above. In particular,  (\ref{bootd2g}) and (\ref{bootd3g}) must hold for some $\hat{T}$ and $P>1$. 

   In Section 9\textit{}, we will prove that  (\ref{bootd1g}) (\ref{bootd2g}) (\ref{bootd3g}) will always hold if  we choose $P=\epsilon_0^{-\frac{1}{8}}$.

   In the remainder of this section, we will derive some Sobolev inequalities, which will be employed frequently in the subsequent sections.

   These inequalities are guaranteed by our bootstrap assumptions.

\begin{lemma} \label{Sobfun}
    For any  $f\in W^{1,1}(M_{t}),$ we have     $$\|f\|_{L^{\frac{3}{2}}(M_{t})}\leq 100\alpha_0\|\nabla f\|_{L^{1}(M_{t})}.$$
\end{lemma}
\begin{proof} This follows from (\ref{bootd1g}) and (\ref{II}). 
\end{proof}

\begin{lemma}\label{L3Sob}
   For  any $W^{1,\frac{3}{2}}$ or $W^{1,2}$- tensor field $F$ on $M_{t}$, we have: 
    \begin{equation}
    \|F\|_{L^{3}(M_{t})}\leq C\|\nabla F\|_{L^{\frac{3}{2}}(M_{t})},
    \end{equation}
    or 
    \begin{equation}\label{L6Sob}
    \|F\|_{L^{6}(M_{t})}\leq C\|\nabla F\|_{L^{2}(M_{t})}.
    \end{equation}
\end{lemma}
\begin{proof}
    The proof is accomplished by substituting   $f=|F|^{2}$ and  $f=|F|^4$ in Lemma \ref{Sobfun}. 
   
\end{proof}

\begin{lemma}\label{LinftySob}
   For  any $W^{1,q}$- tensor field $F$ on $M_{t}$ ($q>3$), we have 
    \begin{equation}\label{inftySob1}
    \|F\|_{L^{\infty}(M_{t})}\leq C||F||^{1-\frac{3}{q}}_{L^q(M_t)}||\nabla F||^{\frac{3}{q}}_{L^{q}(M_t)}\leq C(\|\nabla F\|_{L^{q}(M_{t})}+\|F\|_{L^{q}(M_{t})}),
    \end{equation}
    where $C=C(q,\alpha_0, \alpha_1)$.   
\end{lemma}
\begin{proof}  From  Lemma \ref{Sobfun} and an iteration argument,  we have  (e.g. see \cite{GT77}, P157)
\begin{equation} \label{xf}
||\xi F||_{L^{\infty}(\Omega)}\leq C ||\nabla (\xi F)||_{L^q(\Omega)}|\Omega|^{\frac{1}{n}-\frac{1}{q}},
\end{equation}
where $\xi$ is a cutoff function on $\Omega$. By substituting  $\Omega=B(x_1,a)$ in (\ref{xf}) and making use of  (\ref{III}) and (\ref{bootd1g}),   we obtain 
\begin{equation} \label{3.12}
\|F\|_{L^{\infty}(B(x_1,\frac{a}{2}))}
\leq C(\|\nabla F\|_{L^{q}(M_{t})}+a^{-1}\|F\|_{L^{q}(M_{t})})a^{1-\frac{n}{q}},
\end{equation}
for any $a>0$. Letting $a=\frac{||F||_{L^q}}{||\nabla F||_{L^q}}$ in (\ref{3.12}) and using the arbitrariness of $x_1$, the first inequality in (\ref{inftySob1}) follows. The second  is just the H$\ddot{o}$lder inequality. The proof is completed. 
\end{proof}

 Letting   $q=6$ in (\ref{inftySob1}),  from  (\ref{L6Sob}),  we obtain \begin{equation}\label{inftySob2}
    \|F\|_{L^{\infty}(M_{t})}\leq C \|\nabla^{2} F\|^{\frac{1}{2}}_{L^{2}(M_{t})}\|\nabla F\|^{\frac{1}{2}}_{L^{2}(M_{t})}.
    \end{equation}    

    Under bootstrap assumptions, one can even establish  a stronger statement: 

    \begin{lemma}\label{l3.5}
   For  any $W^{2,2}$-tensor field $F$ on $M_{t}$, we have      \begin{equation}\label{inftySob5}
    \|F\|_{L^{\infty}(M_{t})}\leqslant C\|\Delta F\|^{\frac{1}{2}}_{L^{2}(M_{t})} \|\nabla F\|^{\frac{1}{2}}_{L^{2}(M_{t})},
    \end{equation}
    where  $C=C(\alpha_0, \alpha_1)$.
    \end{lemma}
     \begin{proof}
    To prove (\ref{inftySob5}), we may assume the tensor field $F$ has compact support. By integration by parts, exchanging derivatives and making use of  Gaussian equations,  we have
    \begin{equation} \label{3.13}
    \begin{split}
            \int_{M_{t}}|\nabla^{2}F|^{2} 
            &=-\int_{M_{t}}\nabla_{j}F\nabla^{j}\nabla_{i}\nabla^{i}F+\nabla^{2}F*F*Rm+Rm*\nabla F*\nabla F\\
            &=\int_{M_{t}}(\Delta F)^{2}+(\bar{R}+Hh+h^{2})*(\nabla F*\nabla F+\nabla^{2}F*F).
    \end{split}
    \end{equation}
     On the other hand,  employing  (\ref{L6Sob}) and (\ref{inftySob2}) we obtain    
     \begin{equation}\label{3.14}
    \begin{split}
            &\int_{M_{t}}(\bar{R}+Hh+h^{2})*(\nabla F*\nabla F+\nabla^{2}F*F)\\ & \leq C|||\bar{R}m|+|h|^2||_{L^2}(||\nabla F||^{\frac{1}{2}}_{L^2}||\nabla F||^{\frac{3}{2}}_{L^6}+|| F||_{L^{\infty}}||\nabla^2 F||_{L^2})\\
             & \leq C|||\bar{R}m|+|h|^2||_{L^2}[||\nabla F||^{\frac{1}{2}}_{L^2}||\nabla^2F||^{\frac{3}{2}}_{L^2}+(||\nabla^2 F||_{L^2}+||\nabla F||_{L^2})||\nabla^2 F||_{L^2}]\\     
            &\leq C|||\bar{R}m|+|h|^2||_{L^2}(||\nabla^2 F||^2_{L^2}+||\nabla F||^2_{L^2}).
    \end{split}
    \end{equation}
    Since $ C|||\bar{R}m|+|h|^2||_{L^2}\leq CPE_1\leq C\epsilon_0^{\frac{7}{8}}<\frac{1}{2}$, we have 
\begin{equation}\nonumber
 ||\nabla^{2}F||_{L^2}^{2} 
            \leq C (||\triangle F||_{L^2}^{2}+  ||\nabla F||^2_{L^{2}}).
    \end{equation} 
    Consequently, 
\begin{equation}\label{inftySob3}
 ||F||_{L^{\infty}}
            \leq C (||\triangle F||_{L^2}+  \|\nabla F\|_{L^{2}}).
    \end{equation}     
Note that the constant $C$ in (\ref{inftySob3}) depends only on $\alpha_0$ and $\alpha_1$. Thanks to the scaling invariance of the  conditions (\ref{II}) and (\ref{III}), the estimate (\ref{inftySob3}) also holds for all metrics $\lambda g$ for any $\lambda>0$,  which  concludes  
\begin{equation}\label{inftySob4}
    \|F\|_{L^{\infty}(M_{t})}\leq C(\lambda^{\frac{n}{4}-1}\|\Delta F\|_{L^{2}(M_{t})}+\lambda^{\frac{n}{4}-\frac{1}{2}} \|\nabla F\|_{L^{2}(M_{t})}).
    \end{equation}  
By choosing $\lambda=\frac{||\triangle F||^2_{L^2}}{||\nabla F||^2_{L^2}}$ in (\ref{inftySob4}),  (\ref{inftySob5})  follows.

\end{proof}

\begin{lemma}\label{StSob}
    For any $L^2_tW^{1,2}(M_t)$-tensor field $F$, we have 
    \begin{equation} \label{3.17}
        \|F\|_{L^{\frac{10}{3}}(M_{0}\times [0,  T])}\leq C \|F\|^{\frac{2}{5}}_{L^{\infty}_{t}L^{2}(M_{t})}\|\nabla F\|^{\frac{3}{5}}_{L^{2}(M_{0}\times [0,  T])},
    \end{equation}
    where  $\|F\|_{L^{\frac{10}{3}}(M_{0}\times [0,  T])}=[\int_{0}^{T}\int_{M_t}|F|^{\frac{10}{3}}]^{\frac{3}{10}}$,  $C=C(\alpha_0)$.    
    \end{lemma}
\begin{proof}
    By  H$\ddot{o}$lder inequality and  (\ref{L6Sob}), 
    \begin{equation}
           \|F\|_{L^{\frac{10}{3}}(M_{t})}^{\frac{10}{3}}\leq \|F\|_{L^{2}(M_{t})}^{\frac{4}{3}}\|F\|_{L^{6}(M_{t})}^{2}\leq   C\|F\|_{L^{2}(M_{t})}^{\frac{4}{3}}\|\nabla F\|_{L^{2}(M_{t})}^{2}.
    \end{equation}
    
   Integrating the above inequality over  $[0,T]$ we obtain  (\ref{3.17}). 
\end{proof}

\section{Improved mean curvature estimate}

Let $X(x,t)$ be the  solution to the mean curvature flow (\ref{mcf}) constructed in Section 2-4.  The fundamental  mean curvature estimate in Proposition \ref{P41} clearly also holds  for $X(x,t)$, i.e., we have   

\begin{theorem}\label{} The mean curvature of the solution $X(x,t)$ to (\ref{mcf}) satisfies  \begin{equation}\label{61}
    \mathop{sup}\limits_{M_{t}}|H| \leq  \frac{1}{\sqrt{\frac{2t}{3}+\frac{1}{\sup\limits_{M_0}|H|^2}}}.
    \end{equation}
\end{theorem}

 Nevertheless, the desired convergence of the position function  demands the mean curvature to have a higher decay rate in time (faster  than $1/t$), which certainly requires a little more efforts to make.  The purpose of this section is to derive such an improved mean curvature estimate.

 Our starting point is some  preliminary monotonicity formulas   for the mean curvature integrals.

\begin{theorem}\label{HLpbd1}
For any $p\geq 3$, we have 
\begin{equation}\label{HLpbd2}
 \frac{d^{+}}{dt}\int_{M_{t}} |H|^{p}
     \leq  0.
     \end{equation}
\end{theorem}
\begin{proof}
 From (\ref{412}), for $\alpha> 0$, we have 
 \begin{equation} \label{63}
 \begin{split}
 (\frac{\partial}{\partial t}-\triangle)|H|^{2\alpha}=&\alpha|H|^{2\alpha-2}(\frac{\partial}{\partial t}-\triangle)|H|^{2}-4\alpha(\alpha-1)|H|^{2\alpha-2}|\nabla|H||^2\\
 = &\alpha|H|^{2\alpha-2}(-2|\nabla H|^2-2H^2|h|^2)-4\alpha(\alpha-1)|H|^{2\alpha-2}|\nabla|H||^2\\ \leq & -2\alpha|H|^{2\alpha}|h|^{2}+(2\alpha-4\alpha^2) |H|^{p-2}|\nabla |H||^{2}.
 \end{split}
 \end{equation}

Let $\xi:\mathbb{R}\rightarrow \mathbb{R}$ be  the cut-off function in (\ref{47}).  Applying (\ref{63}) for $p=2\alpha$ and Lemma \ref{l2.2}, we obtain
\begin{equation}\label{HLpcut}
\begin{split}
 \frac{d}{d t}\int_{M_{t}} \xi(\frac{\tilde{r}}{a})|H|^{p}
     \leq & \int_{M_{t}}\xi(\frac{\tilde{r}}{a})[-p|H|^{p}|h|^{2}-p(p-1)|H|^{p-2}|\nabla H|^{2}+|H|^{p+2}]\\
     &+\int_{M_t} |H|^p(\xi^{\prime} a^{-1}(\frac{\partial}{\partial t}-\triangle) \tilde{r}+\xi^{\prime\prime}\frac{|\nabla \tilde{r}|^2}{a^2})\\
     \leq & \int_{M_{t}}\xi(\frac{\tilde{r}}{a})[-(\frac{p}{3}-1)|H|^{p+2}-p(p-1)|H|^{p-2}|\nabla H|^{2}]\\
     &+ \frac{C}{a^2}\int_{\{\frac{a}{2}\leq \tilde{r} \leq a\}} |H|^p.     \end{split}
   \end{equation}

   On the other hand, from Theorem \ref{t44}, if $p>1$, we have 
\begin{equation}
\frac{1}{a^2}\int_{0}^{\bar{T}}\int_{\{\frac{a}{2}\leq \tilde{r} \leq a\}} |H|^p \leq C a^{1-p} \rightarrow 0, \ \ \text{as}\ \  a\rightarrow \infty.
\end{equation}
   
If $p>1$, integrating (\ref{HLpcut})  over a time interval $[t_1,t_2]$, and letting $a\rightarrow \infty$, we obtain  

\begin{equation}\label{66}
\int_{M_{t_2}}|H|^p-\int_{M_{t_1}}|H|^p\leq \int_{t_1}^{t_2}\{\int_{M_t}-(\frac{p}{3}-1)|H|^{p+2}-p(p-1)|H|^{p-2}|\nabla H|^{2}\} dt.\end{equation}
If $p\geq 3$, the right-hand side of (\ref{66}) is non-positive. This  deduces  (\ref{HLpbd2}). 
 \end{proof}

\begin{theorem}\label{H Lp 1<p<2}
    Under the bootstrap assumptions (\ref{bootd1g}) (\ref{bootd2g}) (\ref{bootd3g}), if $\epsilon_0$ is small,  then
    \begin{equation}\label{Hpm}
    \frac{d^{+}}{dt}\int_{M_t}|H|^{p_0}\leq 0,
    \end{equation}
    where $p_0\in (1,\frac{3}{2})$ is the exponent in Theorem \ref{maxfol}. 
\end{theorem}
\begin{proof}  From (\ref{66}), we have 
\begin{equation}\label{HLpevo}
    \begin{aligned}
    \int_{M_{t}}|H|^{p_0}\mid_{t_1}^{t_2}
    &\leq \int_{t_1}^{t_2}[\int_{M_{t}}(1-\frac{p_0}{3})|H|^{p_0+2}-p_0(p_0-1)|H|^{p_0-2}|\nabla H|^{2}]dt\\
      &\leq \int_{t_1}^{t_2}[(1-\frac{p_0}{3})(\int_{M_{t}}|H|^{3p_0})^{\frac{1}{3}}(\int_{M_{t}}|H|^{3})^{\frac{2}{3}}-4\frac{p_0-1}{p_0}\int_{M_{t}}|\nabla |H|^{\frac{p_0}{2}}|^{2}]dt\\
      & \leq [C (1-\frac{p_0}{3})(\int_{M_{0}}|H|^{3})^{\frac{2}{3}}-4\frac{p_0-1}{p_0}]\times \int_{M_{t}}|\nabla |H|^{\frac{p_0}{2}}|^{2},
      \end{aligned}
\end{equation}
where we have used  (\ref{L6Sob}) and (\ref{HLpbd2}). 

On the other hand, from  (\ref{inftySob2}) (\ref{HL_p1}),  we have 
\begin{equation}||H||_{L^{\infty}(M_0)}\leq C||\nabla H||^{\frac{1}{2}}_{L^2(M_0)}||\nabla^2H||^{\frac{1}{2}}_{L^2(M_0)}\leq CE^{\frac{1}{2}}_1E^{\frac{1}{2}}_2,
\end{equation}
\begin{equation}
 ||H||_{L^3(M_0)}\leq C (\|{H}\|_{L^{p_0}(M_{0})} (E_1E_2)^{\frac{3-p_0}{2p_0}})^{\frac{p_0}{3}}\leq C \epsilon_0^{\frac{p_0}{3}}.
\end{equation}
Consequently, 
\begin{equation}\label{Hp1}
  \int_{M_{t}}|H|^{p_0}\mid_{t_1}^{t_2}  \leq - \frac{2p_0-2}{p_0}\int_{t_1}^{t_2} [\int_{M_{t}}|\nabla |H|^{\frac{p_0}{2}}|^{2}]dt \leq 0   
  \end{equation}
 provided $\epsilon_0$ is small. 
This gives a   proof of (\ref{Hpm}).
  
\end{proof}

\begin{theorem}\label{HNash}
    Under the bootstrap assumptions(\ref{bootd1g}) (\ref{bootd2g}) (\ref{bootd3g}), we have 
    \begin{equation}\label{4.17}
        |H|(\cdot, t)\leq C ||H||_{L^{p_0}(M_0)} t^{-\frac{3}{2p_0}},
    \end{equation}
    where $C=C(\alpha_0,\alpha_1)$.
\end{theorem}
\begin{proof} According to (\ref{63}), we have 
\begin{equation}\label{Hp/2evo}
    (\frac{\partial}{\partial t}-\Delta)|H|^{\frac{p_0}{2}}\leq -\frac{p_0}{2}|H|^{\frac{p_0}{2}}|h|^{2}+\frac{p_0}{2}(1-\frac{p_0}{2})|H|^{\frac{p_0}{2}-2}|\nabla |H||^{2}.
\end{equation}

Choose a  time cutoff function $\phi$ satisfying $\phi(t)=0$ for $t\in [0,\frac{1}{8}]$, and $\phi(t)=1$ for $t\in [\frac{1}{4},1]$, and  
 choose a spatial cutoff function $\psi$ satisfying $\psi(x)=1$ for $x\leq \frac{1}{2}$,  and $\psi(x)=0$ for $x \geq \frac{3}{4}$. 
Clearly,   $|\phi^{\prime}|\leq C \sqrt{\phi}$ and $ |\psi^{\prime}|\leq C\sqrt{\psi}$ hold for some universal constant $C$.

Fix a point $x_1\in M_0$,  let   $\zeta(t, x):=\phi(\frac{t}{T})\psi(\frac{d_0(x_1,x)}{R})$, where $d_0(x_1,\cdot)$ is the distance function  from $x_1$ at initial time. 

For any   $k> 0$, by direct computations and integration by parts, we have  

\begin{equation}\label{Htde}
\begin{split}
    \frac{d}{d t} \int_{M_t}\zeta^{2}(|H|^{\frac{p_0}{2}}-k)_{+}^2=&\int_{M_{t}}2\zeta^{2}(|H|^{\frac{p_0}{2}}-k)_{+}\frac{\partial}{\partial t}|H|^{\frac{p_0}{2}}+\frac{\partial \zeta}{\partial t} (|H|^{\frac{p_0}{2}}-k)_{+}^{2}\\
    & + \zeta^{2}|H|^{2}(|H|^{\frac{p_0}{2}}-k)_{+}^{2},   \end{split} \end{equation}
and 
\begin{equation}\label{Hspde}
    \int_{M_{t}}\zeta^{2}(|H|^{\frac{p_0}{2}}-k)_{+}\triangle  |H|^{\frac{p_0}{2}}=-\int_{M_{t}}\zeta^{2}|\nabla(|H|^{\frac{p_0}{2}}-k)_{+}|^{2}+(|H|^{\frac{p_0}{2}}-k)_{+}\nabla \zeta^2\cdot \nabla |H|^{\frac{p_0}{2}}.
\end{equation}

Combining (\ref{Hp/2evo})(\ref{Htde}) and (\ref{Hspde}), we have 

\begin{equation}\label{Hp-k cut}
\begin{split}
    &  \frac{d}{d t} \int_{M_t}\zeta^{2}(|H|^{\frac{p_0}{2}}-k)_{+}^2  \leq \int_{M_{t}}\frac{\partial \zeta}{\partial t} (|H|^{\frac{p_0}{2}}-k)_{+}^{2}+ \zeta^{2}|H|^{2}(|H|^{\frac{p_0}{2}}-k)_{+}^{2}
    \\& \ \ \ -2\zeta^{2}|\nabla(|H|^{\frac{p_0}{2}}-k)_{+}|^{2} +2 (|H|^{\frac{p_0}{2}}-k)_{+}\nabla \zeta^2\cdot \nabla |H|^{\frac{p_0}{2}}\\
    &\ \ \ \ +2\zeta^{2}(|H|^{\frac{p_0}{2}}-k)_{+}[-\frac{p_0}{2}|H|^{\frac{p_0}{2}}|h|^{2}+\frac{p_0}{2}(1-\frac{p_0}{2})|H|^{\frac{p_0}{2}-2}|\nabla |H||^{2}].\end{split} \end{equation}

    Note that 
    \begin{equation}
    \begin{split}
&\frac{p_0}{2}(1-\frac{p_0}{2})(|H|^{\frac{p_0}{2}}-k)_{+}|H|^{\frac{p_0}{2}-2}|\nabla |H||^{2}\\ &= (\frac{2}{p_0}-1)(|H|^{\frac{p_0}{2}}-k)_{+} |H|^{-\frac{p_0}{2}}|\nabla |H|^{\frac{p_0}{2}}|^{2}\\
& \leq (\frac{2}{p_0}-1)|\nabla(|H|^{\frac{p_0}{2}}-k)_{+}|^{2}. \end{split} \end{equation}

Applying  Cauchy-Schwarz inequality in (\ref{Hp-k cut}), for any $0<\sigma <4-\frac{4}{p_0}$,  we have 
\begin{equation}\label{Nashevo1}
\begin{split}
    \frac{d}{d t} \int_{M_t}\zeta^{2}(|H|^{\frac{p_0}{2}}-k)_{+}^2\leq & C_{\sigma}\int_{M_{t}}(|\frac{\partial \zeta}{\partial t}|+|\nabla \zeta|^2) (|H|^{\frac{p_0}{2}}-k)_{+}^{2}\\& -(4-\frac{4}{p_0}-\sigma)\int_{M_{t}}|\nabla (\zeta(|H|^{\frac{p_0}{2}}-k)_{+})|^{2}\\
    &+\int_{M_{t}}\zeta^{2} H^{2}(|H|^{\frac{p_0}{2}}-k)_{+}^2. \end{split} \end{equation}

On the other hand, by Sobolev embedding theorem, we have 
\begin{equation}
\begin{split}
    \int_{M_{t}}\zeta^{2}H^{2}(|H|^{\frac{p_0}{2}}-k)_{+}^{2}
    &\leq (\int_{M_{t}}\zeta^{6}(|H|^{\frac{p_0}{2}}-k)_{+}^{6})^{\frac{1}{3}}(\int_{M_{t}}|H|^{3})^{\frac{2}{3}}\\
    & \leq  C \mathop{sup}\limits_{t}(\int_{M_{t}}|H|^{3})^{\frac{2}{3}} \int_{M_{t}}|\nabla (\zeta (|H|^{\frac{p_0}{2}}-k)_{+})|^{2}\\
    &\leq  C\epsilon_0^{\frac{2p_0}{3}} \int_{M_{t}}|\nabla (\zeta (|H|^{\frac{p_0}{2}}-k)_{+})|^{2}.
    \end{split}
\end{equation}
Since $\epsilon_0$ is small,  (\ref{Nashevo1}) becomes

\begin{equation}\label{Nashevo2}
\begin{split}
    \frac{d}{d t} \int_{M_t}\zeta^{2}(|H|^{\frac{p_0}{2}}-k)_{+}^2\leq & C \int_{M_{t}}(|\frac{\partial \zeta}{\partial t}|+|\nabla \zeta|^2) (|H|^{\frac{p_0}{2}}-k)_{+}^{2}\\& -\delta\int_{M_{t}}|\nabla (\zeta(|H|^{\frac{p_0}{2}}-k)_{+})|^{2},
    \end{split} \end{equation}
for some  $\delta>0$.  Integrating (\ref{Nashevo2}) over $[0,T]$ yields  
\begin{equation}\label{HV2}
\begin{split}
    &\int_{M_{T}}\zeta^{2}(|H|^{\frac{p_0}{2}}-k)_{+}^{2}+\int_{\frac{T}{8}}^{T}[\int_{M_{t}}|\nabla (\zeta(|H|^{\frac{p_0}{2}}-k)_{+})|^{2}]dt\\
    &\leq C\int_{\frac{T}{8}}^{T}[\int_{M_{t}}(|\nabla \zeta|^{2}+|\partial_{t}\zeta|)(|H|^{\frac{p_0}{2}}-k)_{+}^{2}]dt.
    \end{split}
\end{equation}

In what follows,  we will employ  the De Giorgi-Nash-Moser iteration technique. 

Let $R_{0}=\sqrt{T},  R_{l}=\frac{\sqrt{T}}{2}+\frac{\sqrt{T}}{2^{l+1}},  k_{l}=k(2-\frac{1}{2^{l}})$, $Q_{l}=[T-R_{l}^{2},  T]\times B_{0}(x_1, R_{l})$, for $l\geq 1$. Let   $\zeta_{l}$ be a  cutoff function on $Q_l$,   which is equal to  1 on $Q_{l+1}$,  and  zero outside  $Q_{\frac{R_{l}+R_{l+1}}{2}}$,  satisfying  $|\nabla \zeta_{l}|^{2}+|\partial_{t}\zeta_{l}|\leqslant \frac{C}{(R_{l}-R_{l+1})^{2}}$. 

Replacing  $\zeta$ with   $\zeta_{l}$ in (\ref{HV2}) gives   
\begin{equation}
    \int_{M_{T}}\zeta_{l}^{2}(|H|^{\frac{p_0}{2}}-k_{l})^2_{+}+\int_{\frac{T}{8}}^{T}[\int_{M_{t}}|\nabla (\zeta_{l}|H|^{\frac{p_0}{2}})|^{2}]dt
    \leq  C\frac{4^{l}}{T}\|(|H|^{\frac{p_0}{2}}-k_{l})_{+}\|_{L^{2}(Q_{l})}^{2}.
\end{equation}
Thanks to   Lemma \ref{StSob},   we obtain  

\begin{equation}\label{HStSob}
        \|\zeta_{l}(|H|^{\frac{p_0}{2}}-k_{l})_{+}\|_{L^{\frac{10}{3}}(Q_{l})}^{2}
    \leq  C\frac{4^{l}}{T}\|(|H|^{\frac{p_0}{2}}-k_{l})_{+}\|_{L^{2}(Q_{l})}^{2}.
\end{equation}

Let $\beta_{l}=\|(|H|^{\frac{p_0}{2}}-k_{l})_{+}\|_{L^{2}(Q_{l})}^{2},  A_{l}(k)=|Q_{l}\cap \{|H|^{\frac{p_0}{2}}>k\}|$. 

Applying (\ref{HStSob}) and H$\ddot{o}$lder inequality deduces 
$$
\beta_{l+1}
\leq \|\zeta_{l}(|H|^{\frac{p_0}{2}}-k_{l})_{+}\|_{L^{\frac{10}{3}}(Q_{l})}^{2}A_{l}(k_{l+1})^{\frac{2}{5}}\leq C\frac{4^{l}}{T}\beta_{l}A_{l}(k_{l+1})^{\frac{2}{5}}.
$$
In view of 
$$
\beta_{l}
\geq (k_{l+1}-k_{l})^{2}A_{l}(k_{l+1})=\frac{k^{2}}{2^{2l+2}}A_{l}(k_{l+1}),
$$
we have 
\begin{equation}
\beta_{l+1}
\leq C (2^{\frac{14}{5}})^l\frac{\beta_{l}^{\frac{7}{5}}}{Tk^{\frac{4}{5}}}.
\end{equation}
If we denote $y_{l}=\frac{\beta_{l}}{k^{2}T^{\frac{5}{2}}}$, then 
\begin{equation}\label{yl ind}
y_{l+1}
\leq C (2^{\frac{14}{5}})^ly_{l}^{\frac{7}{5}}.
\end{equation}

\textbf{Claim}: 
\begin{equation}\label{y_lbd}
y_l\leq y_0 B^l
\end{equation} 
 holds for all $l\geq 0$, for $B=2^{-\frac{14}{5}}$, provided that   $y_0$ is suitably small.

The argument is as follows.  Suppose (\ref{y_lbd}) is true for $l=l_0$.  By (\ref{yl ind}), we have 
 \begin{equation}
 y_{l_0+1}\leq  C(2^{\frac{14}{5}})^{l_0} (y_0 B^{l_0})^{\frac{7}{5}}
 \leq  (CB^{-1}y_{0}^{\frac{2}{5}})(2^{\frac{14}{5}}B^{\frac{2}{5}})^{l_0} y_0 B^{l_0+1}
 \leq  y_0 B^{l_0+1},
 \end{equation}
 if  $y_0$ satisfies  $CB^{-1} y_0^{\frac{2}{5}}\leq 1$. 

 If we  choose $k= A (T^{-\frac{5}{2}}\int_{B_{0}(x_1,\sqrt{T})\times [\frac{T}{4}, T]} |H|^{p_0})^{\frac{1}{2}}$, then 
 $y_0\leq CA^{-2}$ can  be  small if  $A$ is large. \textbf{Claim}  is proved.  From (\ref{y_lbd}), it follows  that $\lim_{l\rightarrow \infty}y_{l}=0$, which implies 

\begin{equation}
\begin{split}
  \sup_{ B_{0}(x_1,\frac{\sqrt{T}}{2})\times  [\frac{T}{2},T]}|H|^{\frac{p_0}{2}}(x,t)& \leq k
   \leq C (T^{-\frac{5}{2}}\int_{ B_{0}(x_1, \sqrt{T})\times [\frac{T}{4},  T]}|H|^{p_0})^{\frac{1}{2}}\\
  &\leq C||H||_{L^{p_0}(M_0)}^{\frac{p_0}{2}} T^{-\frac{3}{4}}.
  \end{split}
\end{equation}

Therefore, 
\begin{equation}
 \sup_{B_{0}(x_1, \frac{\sqrt{T}}{2})\times  [\frac{T}{2},T]}|H|
\leq C||H||_{L^{p_0}(M_0)} T^{-\frac{3}{2p_0}}.
\end{equation}
The proof is completed. 

\end{proof}

\section{The estimates of the derivatives  of the mean curvature}

Since the mean curvature and its derivatives appear in the evolution equations for the curvature and the second fundamental form,  the  estimate of mean curvatures  is  of great importance. The purpose  of this section is to  establish  $L^2$ estimates for  the mean curvature  up to third-order derivatives  by means of  the  energy estimates.  
 
\begin{lemma}\label{dH st L2}
    Under the bootstrap assumptions,  we have
    \begin{equation}
    \label{5.1}\quad\int_{\frac{T}{4}}^{T}[\int_{M_{t}} |\nabla H|^{2}]dt\leq C||H||_{L^{p_0}(M_0)}^{2} T^{\frac{3}{2}-\frac{3}{p_0}}, \end{equation} 
 \begin{equation}\label{5.2} 
 \quad\int_{M_{T}}|\nabla H|^{2}+\int_{\frac{3}{8}T}^{T}[\int_{M_{t}} |\nabla^{2} H|^{2}]dt
    \leq  C P^2  E_1^{2}||H||_{L^{p_0}(M_0)}^{2} T^{1-\frac{3}{p_0}}+C||H||_{L^{p_0}(M_0)}^{2} T^{\frac{1}{2}-\frac{3}{p_0}}.\end{equation}
    \end{lemma}
\begin{proof}
    Thanks to   (\ref{Hp1}),   we have

    \begin{equation}\label{HLpbd3}
    \int_{0}^{T}[\int_{M_{t}} |H|^{p_0-2}|\nabla |H||^{2}]dt \leq   C \int_{M_0}|H|^{p_0},
    \end{equation}
which together  with (\ref{4.17}) implies 
    \begin{equation}
    \begin{split}
        \int_{\frac{T}{8}}^{T}\int_{M_{t}}|\nabla |H||^{2}
    &= \int_{\frac{T}{8}}^{T}\int_{M_{t}} |H|^{2-p_0}|H|^{p_0-2}|\nabla H|^{2}\\
    &\leq C||H||_{L^{p_0}(M_0)}^{2-p_0} T^{\frac{3}{2}-\frac{3}{p_0}}\int_{0}^{T}\int_{M_{t}} |H|^{p_0-2}|\nabla H|^{2} \\
    &\leq C ||H||_{L^{p_0}(M_0)}^{2}T^{\frac{3}{2}-\frac{3}{p_0}}.
    \end{split}
    \end{equation}

   Thus,    (\ref{5.1}) follows by using the fact  that  $|\nabla H|=|\nabla |H||$ a.e..

To prove (\ref{5.2}), we need to compute the evolution equation of $\nabla H$ by differentiating the last equation in  (\ref{412}): 
\begin{equation} \label{75}
\begin{split}
  (\frac{\partial}{\partial t}-\triangle)\nabla_{l}H&=-\nabla_{l}H|h|^{2}-H\nabla_{l}|h|^{2}-R_{lm}\nabla^{m}H\\
  &=-\nabla_{l}H|h|^{2}-2H\nabla_{l}h_{ij}h^{ij}-\bar{R}_{l0m0}\nabla^{m}H+Hh_{lm}\nabla^{m}H-h^{2}_{lm}\nabla^{m}H.
\end{split}
\end{equation}

By using similar cutoff function technique as in (\ref{HLpcut}) (\ref{66}), 
\begin{equation}\label{dH L2 evo}
    \begin{split}
       \frac{d^{+}}{d t}\int_{M_{t}}|\nabla H|^{2}
       \leq & \int_{M_{t}}-2Hh_{ij}\nabla^{i}H\nabla^{j}H-2|\nabla^{2}H|^{2} -2|\nabla H|^{2}|h|^{2}\\& -2H\nabla_{l}|h|^{2}\nabla^{l}H -2\bar{R}_{m0l0}\nabla^{m}H\nabla^{l}H-2\nabla^{l}Hh_{lm}^{2}\nabla^{m}H\\& +2Hh_{ij}\nabla^{i}H\nabla^{j}H  +H^{2}|\nabla H|^{2}.
    \end{split}
\end{equation}

Note that
\begin{equation}\nabla^{m}\bar{R}_{m0l0}=D^{m}\bar{R}_{m0l0}+h^{mn}\bar{R}_{mnl0}+h^{mn}\bar{R}_{m0ln}, 
\end{equation}  which together   with the Bianchi identity gives 
\begin{equation}\nabla^{m}\bar{R}_{m0k0}=h^{ml}\bar{R}_{mlk0}+h^{ml}\bar{R}_{m0kl}.
\end{equation} 
By integration by parts, 
\begin{equation}\label{dH L2 est1}
\begin{split}
    |\int_{M_{t}}\bar{R}_{m0l0}\nabla^{m}H\nabla^{l}H |
    &=|-\int_{M_{t}}\bar{R}_{m0l0}H\nabla^{m}\nabla^{l}H+ h^{mn}\bar{R}_{m0ln}H\nabla^{l}H |\\
    &\leq \frac{1}{100}\int_{M_{t}}|\nabla^{2}H|^{2}+|\nabla H|^{2}|h|^{2}+C \int_{M_{t}} H^{2}|\bar{R}|^{2}.
    \end{split}
\end{equation}

The  Cauchy-Schwarz inequality asserts 
\begin{equation}\label{dH L2 est2}
|\int_{M_{t}}H\nabla_{l}h_{ij}h^{ij}\nabla^{l}H  |\leq  C\int_{M_{t}}H^{2}|\nabla h|^{2}d\mu_{t}+C^{-1}\int_{M_{t}}|h|^{2}|\nabla H|^{2} .\end{equation}  

Combining (\ref{dH L2 evo}) (\ref{dH L2 est1}) (\ref{dH L2 est2}), we have
\begin{equation} \label{5.11}
 \frac{d^{+}}{d t}\int_{M_{t}}|\nabla H|^{2} +\int_{M_{t}}|\nabla^{2}H|^{2}+|\nabla H|^2 |h|^2 
 \leq  C \int_{M_{t}}|H|^2(|\nabla h|^{2}+|\bar{R}|^2).
\end{equation}

Let  $\phi:\mathbb{R}\rightarrow \mathbb{R}$ be a nonnegative and non-decreasing function,   satisfying $\phi(t)=0$ when  $t\leq \frac{1}{4}$, and $\phi(t)=1$ when  $t\geq \frac{3}{8}$. 
Using (\ref{5.11}), a direct computation yields 
\begin{equation}\label{5.12}
\begin{split}
& \frac{d^{+}}{d t}\int_{M_{t}}\phi(\frac{t}{T})|\nabla H|^{2}
+\phi(\frac{t}{T})\int_{M_{t}}|\nabla^{2}H|^{2}+|\nabla H|^2 |h|^2 
 \\ \leq & C \int_{M_{t}}\phi(\frac{t}{T})|H|^2(|\nabla h|^{2}+|\bar{R}|^2)+\frac{C}{T}\sqrt{\phi(\frac{t}{T})}\int_{M_t}|\nabla H|^2\\
 \leq & C P^2E_1^2||H||_{L^{p_0}(M_0)}^{2} t^{-\frac{3}{p_0}}+\frac{C}{T}\sqrt{\phi(\frac{t}{T})}\int_{M_t}|\nabla H|^2,
 \end{split}\end{equation}
where we have used (\ref{4.17}) and bootstrap assumption (\ref{bootd2g}). 

The desired  estimate  (\ref{5.2}) follows from integrating  (\ref{5.12})  from $\frac{T}{4}$ to $T$ and using (\ref{5.1}). 

\end{proof}

\begin{lemma}
Under the bootstrap assumptions, we have:
\begin{equation}\label{5.13}
\begin{aligned}
    &\int_{M_{T}}|\nabla^{2}H|^{2}+ \int_{\frac{3T}{8}}^{T}\int_{M_{t}}|\nabla^{3}H|^{2} \\ & \leq  C ||H||_{L^{p_0}(M_0)}^2 (P||H||_{L^{p_0}(M_0)} E^{\frac{1}{2}}_1E^{\frac{1}{2}}_2{T^{-\frac{3}{2p_0}}}+T^{-1} + P^4E_1^4) (P^2E_1^2T^{1-\frac{3}{p_0}}+ T^{\frac{1}{2}-\frac{3}{p_0}})\\& \ \ \ + C||H||_{L^{p_0}(M_0)}^2 P^4E_1^3E_2 T^{1-\frac{3}{p_0}}.  \end{aligned}
\end{equation}

\end{lemma}

\begin{proof} We need to  calculate the evolution equation of $\nabla^2 H$ by differentiating  the equation (\ref{75}): 
    \begin{equation}\label{d2H evo}
    \begin{split}
        (\frac{\partial}{\partial t}-\triangle)\nabla_{i}\nabla_{j}H
       =& 2R_{ipjq}\nabla_p\nabla_qH-R_{ip}\nabla_{p}\nabla_jH-R_{jp}\nabla_{p}\nabla_iH\\& 
       +(\nabla_kR_{ij}-\nabla_iR_{jk}-\nabla_jR_{ik})\nabla_kH-\nabla_i\nabla_j(|h|^2H)-(\Gamma^{p}_{ij})^{\prime}\frac{\partial H}{\partial x^p} .  \end{split}
\end{equation}
Using (\ref{d2H evo})  and a  cutoff function technique yields   
\begin{equation}\label{d2H L2 evo}
\begin{aligned}
   \frac{d^{+}}{d t}\int_{M_{t}}|\nabla^{2}H|^{2} \leq& \int_{M_{t}}-2|\nabla^{3}H|^{2}+ \nabla Rm \ast \nabla^2 H \ast \nabla H+Rm\ast \nabla^2H \ast \nabla^2H\\ 
           &-2\nabla_i\nabla_jH \nabla_i\nabla_j(|h|^2H)+\nabla H\ast \nabla (Hh) \ast \nabla^2 H\\
        &-4 Hh_{pq}\nabla_{p}\nabla_{l}H\nabla_{q}\nabla_{l}H+H^{2}|\nabla^{2}H|^2.
\end{aligned}
\end{equation}
Now we handle the  terms in the right hand side of  (\ref{d2H L2 evo}). 
After integration by parts and applying Cauchy-Schwarz inequality, we have 
\begin{equation}\label{d2H est1}
\begin{split}
    \int_{M_{t}} \nabla Rm \ast \nabla H \ast \nabla^{2}H &\leq \int_{M_{t}} Rm \ast \nabla H \ast \nabla^{3}H+ Rm \ast \nabla^2 H \ast \nabla^{2}H\\
    & \leq \int_{M_{t}}\delta  |\nabla^3 H|^2+C_{\delta} |Rm|^2|\nabla H|^2+ Rm \ast \nabla^2 H \ast \nabla^{2}H,\end{split}
    \end{equation}
\begin{equation}\label{d2H est2}
    -2 \int_{M_{t}} \nabla_i\nabla_jH\nabla_i\nabla_j(|h|^2H)
     \leq  \int_{M_{t}}\delta |\nabla^3 H|^2+C_{\delta} (|\nabla H|^2 |h|^4+|\nabla h|^2|h|^2|H|^2),
    \end{equation}
\begin{equation}\label{d2H est3}
    \int_{M_{t}} \nabla H\ast \nabla (Hh)\ast \nabla^2H
     \leq  \int_{M_{t}}\delta |\nabla^3 H|^2+C_{\delta} |\nabla H|^2 |h|^2|H|^2+C|h||H||\nabla^2 H|^2.
    \end{equation}

On the other hand, by H$\ddot{o}$lder inequality and Sobolev embedding theorem, we have 
\begin{equation}\label{d2H est4}
\begin{split}
    |\int_{M_{t}}Rm \ast \nabla^2H \ast \nabla^2H| & \leq  C (\int_{M_t}|Rm|^2)^{\frac{1}{2}}(\int_{M_t} |\nabla^{2}H|^{2})^{\frac{1}{4}}\int_{M_t} |\nabla^{2}H|^{6})^{\frac{1}{4}} \\ & \leq CPE_1 (\int_{M_t}|\nabla^{2}H|^{2})^{\frac{1}{4}}(\int_{M_t}|\nabla^{3}H|^{2})^{\frac{3}{4}}\\
    & \leq \delta \int_{M_t}|\nabla^3 H|^2 +C_{\delta}P^4E_1^4 \int_{M_t}|\nabla^2 H|^2.
    \end{split}
\end{equation}

Combining (\ref{d2H L2 evo})(\ref{d2H est1})(\ref{d2H est2})(\ref{d2H est3}) and (\ref{d2H est4}), we obtain 
\begin{equation}\label{d2H L2 est}
\begin{aligned}
   \frac{d^{+}}{d t}\int_{M_{t}}|\nabla^{2}H|^{2} \leq & -\frac{3}{2}||\nabla^{3}H||_{L^2}^{2}+C\max_{M_t}|\nabla H|^2 \cdot ||Rm|+|h|^2||_{L^2}^2\\
   &+C (\max_{M_t}|H||h|+P^4E_1^4)\cdot  ||\nabla^2 H||_{L^2}^2+C \max_{M_t}|H|^2|h|^2 \cdot||\nabla h||_{L^2}^2\\
   \leq &-||\nabla^{3}H||_{L^2}^{2}+C||H||_{L^{p_0}(M_0)}^2 P^4E_1^3E_2 t^{-\frac{3}{p_0}}\\
   &+C(P||H||_{L^{p_0}(M_0)} E^{\frac{1}{2}}_1E^{\frac{1}{2}}_2{T^{-\frac{3}{2p_0}}}+P^4E_1^4)||\nabla^2 H||_{L^2}^2,\end{aligned}
\end{equation}
where we have used (\ref{bootd2g}) (\ref{inftySob2}) and 
\begin{equation} \label{5.21}
||h||_{L^{\infty}(M_t)}\leq C||\nabla h||^{\frac{1}{2}}_{L^2(M_t)}||\nabla^2 h||^{\frac{1}{2}}_{L^2(M_t)} \leq CPE^{\frac{1}{2}}_1E^{\frac{1}{2}}_2, 
\end{equation}
\begin{equation} \label{5.22}
||\nabla H||^2_{L^{\infty}(M_t)}\leq C||\nabla^2 H||_{L^2(M_t)}||\nabla^3 H||_{L^2(M_t)}.
\end{equation} 

The result  (\ref{5.13}) follows from  integrating  (\ref{d2H L2 est}) and exploiting   similar cutoff function technique as in (\ref{5.12}).  

\end{proof}

\begin{lemma} \label{l5.3}
 Under bootstrap assumptions, we have: 

\begin{equation} \label{5.23}
\int_{0}^{\hat{T}} ||H||_{L^{\infty}}||h||_{L^{\infty}} \leq CP \epsilon_0^{\frac{p_0}{3-p_0}},\end{equation}
\begin{equation}\label{5.24}
    \int^{\hat{T}}_{0} ||\nabla H||_{L^\infty}ds \leq  CP^{5}\epsilon_0.
      \end{equation}
     
\end{lemma}

\begin{proof}
    Note that from (\ref{61})(\ref{4.17}),  we have 
\begin{equation}\label{5.25}
\int_{0}^{\hat{T}} ||H||_{L^{\infty}}\leq C\int_{0}^{a}{\frac{ds}{\sqrt{\frac{2s}{3}}}}+C||H||_{L^{p_0}(M_0)}  \int_{a}^{\infty} s^{-\frac{3}{2p_0}}ds\leq C \sqrt{a}+C ||H||_{L^{p_0}(M_0)} a^{1-\frac{3}{2p_0}}
\end{equation}
since we require  $p_0\in (1,\frac{3}{2})$.   Let $a= ||H||_{L^{p_0}(M_0)}^{\frac{2p_0}{3-p_0}}$ in (\ref{5.25}), we obtain:

\begin{equation}\label{5.26}
\int_{0}^{\hat{T}} ||H||_{L^{\infty}}\leq C||H||_{L^{p_0}(M_0)}^{\frac{p_0}{3-p_0}},
\end{equation}
which together with (\ref{5.21})(\ref{HL_p1}) gives (\ref{5.23}).

We  now turn  to   the proof of  (\ref{5.24}).    We need to use (\ref{5.22}). 

We divide the argument into two cases. Case i)  $\hat{T}\geq (PE_1)^{-4}$, and Case ii)   $\hat{T} < (PE_1)^{-4}$.

In Case i),  for any $a \leq  \hat{T}$, there exists $k\in \mathbb{Z}_{+}$ with $\hat{T}\in [2^{k}a, 2^{k+1}a)$. Then    
\begin{equation}
   \int^{\hat{T}}_{0} ||\nabla^2 H||^{\frac{1}{2}}_{L^2}||\nabla^3 H||^{\frac{1}{2}}_{L^2}ds\leq I_1+I_2
      \end{equation}
      where 
\begin{equation} \label{5.28}
   I_1= a^{\frac{1}{2}} (\int^a_{0} ||\nabla^2 H||^{2}_{L^2}ds)^{\frac{1}{4}}(\int_{0}^a||\nabla^3 H||^{2}_{L^2}ds)^{\frac{1}{4}},
      \end{equation}   
\begin{equation} \label{5.29}
  \begin{split}
 I_2= &  \sum_{i=0}^{k-1} 2^{\frac{i}{2}}a^{\frac{1}{2}} (\int^{2^{i+1}a}_{2^ia} ||\nabla^2 H||^2_{L^2}ds)^{\frac{1}{4}} (\int^{2^{i+1}a}_{2^ia} |\nabla^3 H||^2_{L^2} ds)^{\frac{1}{4}}\\& + 2^{\frac{k}{2}}a^{\frac{1}{2}}(\int^{\hat{T}}_{2^ka} ||\nabla^2 H||^2_{L^2}ds)^{\frac{1}{4}} (\int^{\hat{T}}_{2^ka} |\nabla^3 H||^2_{L^2} ds)^{\frac{1}{4}}.
      \end{split}
      \end{equation}      
      
  On the other hand, from (\ref{dH L2 evo}) we have 

\begin{equation}\label{5.30}
 \frac{d^{+}}{d t}\int_{M_{t}}|\nabla H|^{2} +\int_{M_{t}}|\nabla^{2}H|^{2}+|\nabla H|^2 |h|^2 
 \leq  C P^2 E_1^2 ({||H||^{-2}_{L^{\infty}(M_0)}+t})^{-1}.
 \end{equation}

  Choose $a=(PE_1)^{-4}$.   Integrating (\ref{5.30}) from $0$ to $a$, 
 \begin{equation}\label{5.31} 
 \begin{split}
 ||\nabla H||^2_{L^2(M_a)}+\int_{0}^{a} ||\nabla^{2}H||_{L^2}^{2} \leq & C ||\nabla H||^2_{L^2(M_0)}+ CP^2E_1^2 \log (1+||H||^{2}_{L^{\infty}(M_0)}a)\\
 \leq & C ||\nabla H||^2_{L^2(M_0)}+C ||H||_{L^{\infty}(M_0)},
 \end{split}
 \end{equation}
where we have used $\log (1+x) \leq \frac{1}{\alpha} x^{\alpha}$ for $x\geq 0$ and $0<\alpha\leq 1$. 

From (\ref{d2H L2 est}), we have 
\begin{equation} \label{5.32}
\begin{aligned}
  & \frac{d^{+}}{d t}[ e^{-C\int_0^t||H h||_{L^{\infty}(M_s)}ds-CP^4E_1^4t }\int_{M_{t}}|\nabla^{2}H|^{2}]\\ 
  &  \leq e^{-C\int_0^t||H h||_{L^{\infty}(M_s)}ds-CP^4E_1^4t }[-\frac{1}{2}\int_{M_{t}}|\nabla^{3}H|^{2}+CP^4E_1^3E_2 ({||H||^{-2}_{L^{\infty}(M_0)}+t})^{-1}].\end{aligned}
\end{equation}  

Integrating (\ref{5.32}) from $0$ to $a$ gives 
\begin{equation} \label{5.33}
\begin{aligned}
   \int_{0}^a||\nabla^{3}H||_{L^2}^{2} & \leq C ||\nabla^2 H||^2_{L^2(M_0)}+CP^4E_1^3E_2 \log (1+||H||^{2}_{L^{\infty}(M_0)}a)\\
   & \leq  C ||\nabla^2 H||^2_{L^2(M_0)}+CP^2E_1E_2||H||_{L^{\infty}(M_0)}.   \end{aligned}
\end{equation} 
Combining (\ref{5.31}) and (\ref{5.33}), we have 
\begin{equation} \label{5.34}
  I_1 \leq P^2 (||\nabla H||^2_{L^2(M_0)}+ ||H||_{L^{\infty}(M_0)})^{\frac{1}{4}}( ||\nabla^2 H||^2_{L^2(M_0)}+E_1E_2||H||_{L^{\infty}(M_0)})^{\frac{1}{4}}E_1^{-2}.
      \end{equation}

Now we proceed to handle $I_2$. 

Owing to  (\ref{5.2}) and (\ref{5.13}), we  obtain  

\begin{equation} \label{5.35}
\begin{aligned}
& (\int_{2^ia}^{2^{i+1 }a} ||\nabla^{2}H||^2_{L^2}dt)^{\frac{1}{4}}(\int_{2^ia}^{2^{i+1}a}||\nabla^{3}H||_{L^2}^{2}dt)^{\frac{1}{4}} (2^i a)^{\frac{1}{2}}\\ & \leq   C(P||H||_{L^{p_0}(M_0)} E^{\frac{1}{2}}_1E^{\frac{1}{2}}_2 a^{-\frac{3}{2p_0}}+ P^4E_1^4)^{\frac{1}{4}}\\ & \ \ \times (||H||_{L^{p_0}(M_0)}^2 P^{2}E_1^{2}2^{(1-\frac{3}{p_0})i}a^{1-\frac{3}{p_0}}+||H||_{L^{p_0}(M_0)}^2 2^{i(\frac{1}{2}-\frac{3}{p_0})}a^{\frac{1}{2}-\frac{3}{p_0}})^{\frac{1}{2}}(2^i a)^{\frac{1}{2}}\\ & \ + C(||H||_{L^{p_0}(M_0)}^2 P^4E_1^3E_2 2^{i(1-\frac{3}{p_0})}a^{1-\frac{3}{p_0}})^{\frac{1}{4}}\\ & \ \ \times (||H||_{L^{p_0}(M_0)}^2 P^{2}E_1^{2}2^{(1-\frac{3}{p_0})i}a^{1-\frac{3}{p_0}}+||H||_{L^{p_0}(M_0)}^2 2^{i(\frac{1}{2}-\frac{3}{p_0})}a^{\frac{1}{2}-\frac{3}{p_0}})^{\frac{1}{4}} (2^i a)^{\frac{1}{2}}\\
& \leq  C ||H||_{L^{p_0}(M_0)} [(P^{\frac{1}{4}}||H||_{L^{p_0}(M_0)}^{\frac{1}{4}}E_1^{\frac{1}{8}}E_2^{\frac{1}{8}}a^{-\frac{3}{8p_0}}+PE_1)PE_1+P^{\frac{3}{2}}E_1^{\frac{5}{4}}E_2^{\frac{1}{4}}]a^{1-\frac{3}{2p_0}}2^{(1-\frac{3}{2p_0})i}\\
&\ +C||H||_{L^{p_0}(M_0)}(P^{\frac{1}{4}}||H||_{L^{p_0}(M_0)}^{\frac{1}{4}}E_1^{\frac{1}{8}}E_2^{\frac{1}{8}}a^{-\frac{3}{8p_0}}+PE_1)a^{\frac{3}{4}-\frac{3}{2p_0}}2^{(\frac{3}{4}-\frac{3}{2p_0})i}\\ & \ +C ||H||_{L^{p_0}(M_0)} P E_1^{\frac{3}{4}}E_2^{\frac{1}{4}}a^{\frac{7}{8}-\frac{3}{2p_0}}2^{(\frac{7}{8}-\frac{3}{2p_0})i}.\end{aligned}
\end{equation}  
Because of   $p_0\in (1,\frac{3}{2})$, we have 
\begin{equation} \label{5.36}
\begin{aligned}
I_2 \leq & C ||H||_{L^{p_0}(M_0)} [(P^{\frac{1}{4}}||H||_{L^{p_0}(M_0)}^{\frac{1}{4}}E_1^{\frac{1}{8}}E_2^{\frac{1}{8}}a^{-\frac{3}{8p_0}}+PE_1)PE_1+P^{\frac{3}{2}}E_1^{\frac{5}{4}}E_2^{\frac{1}{4}}]a^{1-\frac{3}{2p_0}}\\
& +C||H||_{L^{p_0}(M_0)}(P^{\frac{1}{4}}||H||_{L^{p_0}(M_0)}^{\frac{1}{4}}E_1^{\frac{1}{8}}E_2^{\frac{1}{8}}a^{-\frac{3}{8p_0}}+PE_1)a^{\frac{3}{4}-\frac{3}{2p_0}}\\ & +C ||H||_{L^{p_0}(M_0)} P E_1^{\frac{3}{4}}E_2^{\frac{1}{4}}a^{\frac{7}{8}-\frac{3}{2p_0}}\\
\leq & C P^{5}(||H||_{L^{p_0}(M_0)}^{\frac{5}{4}}E_1^{\frac{15}{2p_0}-\frac{23}{8}}E_2^{\frac{1}{8}}+||H||_{L^{p_0}(M_0)}E_1^{\frac{6}{p_0}-2}+||H||_{L^{p_0}(M_0)}E_1^{\frac{6}{p_0}-\frac{11}{4}}E_2^{\frac{1}{4}}).
\end{aligned}
\end{equation}

Consequently,  from  (\ref{5.34}) (\ref{5.36}) (\ref{HL_p1}) (\ref{HL_p2}) and (\ref{dHL_2}),  we obtain \begin{equation}
  \begin{split}
  &  \int^{\hat{T}}_{0} ||\nabla^3 H||^{\frac{1}{2}}_{L^2}||\nabla^2 H||^{\frac{1}{2}}_{L^2}ds\\& \leq C P^2 (||\nabla H||^2_{L^2(M_0)}+ ||H||_{L^{\infty}(M_0)})^{\frac{1}{4}}( ||\nabla^2 H||^2_{L^2(M_0)}+E_1E_2||H||_{L^{\infty}(M_0)})^{\frac{1}{4}}E_1^{-2}\\& \ \ \ +CP^5 (||H||_{L^{p_0}(M_0)}^{\frac{5}{4}}E_1^{\frac{15}{2p_0}-\frac{23}{8}}E_2^{\frac{1}{8}}+||H||_{L^{p_0}(M_0)}E_1^{\frac{6}{p_0}-2}+||H||_{L^{p_0}(M_0)}E_1^{\frac{6}{p_0}-\frac{11}{4}}E_2^{\frac{1}{4}})\\
  & \leq CP^{5}(\epsilon_0+\epsilon_0^{\frac{5}{4}}) +C P^2 \epsilon_0\\
  & \leq CP^5 \epsilon_0.\end{split}
      \end{equation} 
 In   Case ii), the integral $\int_0^{\hat{T}}||\nabla H||_{L^{\infty}}dt$ can be bounded by the right hand side of (\ref{5.34}).  
   This completes the proof.
\end{proof}

\section{Curvature and second fundamental form estimates}

In this section,  using the estimates obtained in Sections 4 to 7,  we improve the estimates  for the   first-order derivatives  of  the curvature and the  second-order derivatives  of the second fundamental form.  

\subsection{Curvature estimates}
\begin{proposition}\label{R0 L2}
    Under the bootstrap assumptions,  we have 
    \begin{equation} 
    ||\bar{R}_{i0j0}||^2_{L^{\infty}_{t}L^{2}(M_{t})}+||\bar{R}_{ijk0}||^2_{L^{\infty}_{t}L^{2}(M_{t})}\leq C(1+P^{7}\epsilon_0+P^3 \epsilon_0^{\frac{p_0}{3-p_0}})E_1^2.\end{equation}
\end{proposition}

\begin{proof}
By using  Bianchi identity and direct computations, we find 

\begin{equation}\label{R00 evo}
        \begin{aligned}
        \frac{\partial}{\partial t}\bar{R}_{i0j0}=& H \nabla_k\bar{R}_{ikj0}+ (2\bar{R}_{k0j0}h_{ik}+\bar{R}_{i0k0}h_{kj}-\bar{R}_{ikjl}h_{kl}-\bar{R}_{i0j0}H)H\\& + (\bar{R}_{ikj0}+\bar{R}_{i0jk})\nabla_k H,       \end{aligned}
\end{equation}

\begin{equation}\label{R0 evo}
    \begin{aligned}
     \frac{\partial}{\partial t}\bar{R}_{ijk0}=& H (\nabla_j \bar{R}_{i0k0}-\nabla_i\bar{R}_{j0k0})+(\bar{R}_{j0kp}h_{ip}-\bar{R}_{i0kp}h_{jp}-\bar{R}_{ijp0}h_{kp})H\\& +  \bar{R}_{i0k0} \nabla_jH - \bar{R}_{j0k0} \nabla_iH+ \bar{R}_{ijkl}\nabla_lH.\end{aligned}
\end{equation}

 Then,  we have  
\begin{equation} \label{6.4}
        \begin{aligned}
        \frac{\partial}{\partial t}(2 |\bar{R}_{i0j0}|^2+|\bar{R}_{ijk0}|^2) =& \nabla_k (4 H\bar{R}_{ikj0}\bar{R}_{i0j0})+\bar{R}m\ast \bar{R}m \ast \nabla H+ \bar{R}m\ast \bar{R}m \ast H h,         \end{aligned}
\end{equation}

which implies 

\begin{equation}\label{R0 L2 est}
  \begin{aligned}
  \frac{d^{+}}{d t} \int_{M_{t}} 2 |\bar{R}_{i0j0}|^2+|\bar{R}_{ijk0}|^2  \leq C (||\nabla H||_{L^{\infty}}+||H||_{L^{\infty}} ||h||_{L^{\infty}}) ||\bar{R}m||^2_{L^{2}}.     \end{aligned}
\end{equation}

Integrating (\ref{R0 L2 est})  over $[0,  t]$ and employing   Lemma  \ref{l5.3}, we obtain 

\begin{equation}
\begin{split}
  \int_{M_{t}}2 |\bar{R}_{i0j0}|^2+|\bar{R}_{ijk0}|^2 & \leq CE_1^2+C P^2 E_1^2 \int_{0}^t||\nabla H||_{L^{\infty}}+||H||_{L^{\infty}} ||h||_{L^{\infty}}ds\\
  & \leq CE_1^2+  C(P^{7}\epsilon_0+P^3 \epsilon_0^{\frac{p_0}{3-p_0}})E_1^2,
      \end{split}
\end{equation} 
which completes the proof.

    \end{proof}

\begin{proposition}\label{dR0 L2}
    Under the bootstrap assumptions,  we have 
    \begin{equation} \label{6.7}
        \|\nabla \bar{R}_{0***}\|^2_{L^{\infty}_{t}L^2(M_{t})}\leq C(1+P^{7}\epsilon_0+P^3 \epsilon_0^{\frac{p_0}{3-p_0}})E_2^2.
    \end{equation} 
\end{proposition}
\begin{proof}
Taking covariant derivatives  on  both sides of  equations (\ref{R00 evo}) and (\ref{R0 evo}) yields 
\begin{equation}
 \begin{aligned}
        \frac{\partial}{\partial t}\nabla_{p}\bar{R}_{i0j0}
        =&H\nabla_{p}\nabla_{l}\bar{R}_{ilj0}+\nabla[\bar{R}m \ast (\nabla H+ Hh)] +\nabla \Gamma^{\prime}\ast \bar{R}m\\ =& H\nabla_{l}\nabla_{p}\bar{R}_{ilj0}+\nabla\bar{R}m \ast (\nabla H+ Hh)\\ &  +\bar{R}m \ast (\nabla^2 H+ \nabla H\ast h+H\ast \nabla h+Rm\ast H),
        \end{aligned}
\end{equation}

\begin{equation}
    \begin{aligned}
     \frac{\partial}{\partial t}\nabla_p\bar{R}_{ijk0}=& H (\nabla_p\nabla_j \bar{R}_{i0k0}-\nabla_p\nabla_i\bar{R}_{j0k0})+\nabla\bar{R}m \ast (\nabla H+ Hh)\\& + \bar{R}m \ast (\nabla^2 H+ \nabla H\ast h+H\ast \nabla h+Rm \ast H).
     \end{aligned}
\end{equation}

Hence, we have 

\begin{equation}
        \begin{aligned}
       &  \frac{\partial}{\partial t}(2 |\nabla_p\bar{R}_{i0j0}|^2+|\nabla_p\bar{R}_{ijk0}|^2)\\ =& \nabla_k (4 H\nabla_p\bar{R}_{ikj0}\nabla_p\bar{R}_{i0j0})+\nabla \bar{R}m\ast \nabla \bar{R}m \ast (\nabla H+H h)\\& +\nabla \bar{R}m\ast \bar{R}m \ast[\nabla^2 H+ \nabla H\ast h+H (\bar{R}m+\nabla h+h\ast h)].    
        \end{aligned}
\end{equation}

This implies 
\begin{equation}\label{6.11}
        \begin{aligned}
       &  \frac{d^{+}}{d t}\int_{M_t}2 |\nabla_p\bar{R}_{i0j0}|^2+|\nabla_p\bar{R}_{ijk0}|^2\\ 
       & \leq C ||\nabla \bar{R}m||^2_{L^2}(||\nabla H||_{L^\infty}+||H||_{L^\infty}||h||_{L^\infty}) + C \|\nabla \bar{R}m||_{L^2} ||\bar{R}m\ast \nabla^2 H||_{L^2}\\& \ \ \ +C ||\nabla \bar{R}m||_{L^2}( ||\nabla H||_{L^\infty}+||h||_{L^{\infty}}||H||_{L^{\infty}}) || h\ast \bar{R}m||_{L^2} \\ & \ \ \ + C\|\nabla \bar{R}m||_{L^2}  ||H||_{L^\infty} (||\bar{R}m||^{2}_{L^4}+||\nabla h||^2_{L^4}).
       \end{aligned}
\end{equation}

To estimate the right hand side of (\ref{6.11}), we need  the following inequalities
\begin{equation}
    ||\bar{R}m\ast \nabla^2 H||_{L^2}\leq C ||\bar{R}m||_{L^6}||\nabla^2 H||_{L^3}\leq C||\nabla \bar{R}m||_{L^2}||\nabla^{2}H||_{L^2}^{\frac{1}{2}}||\nabla^{3}H||_{L^2}^{\frac{1}{2}},
\end{equation}
\begin{equation}
|| h\ast \bar{R}m||_{L^2}\leq C ||h||_{L^6}||\bar{R}m||_{L^3}\leq C PE_2,  
\end{equation}
\begin{equation}
||\bar{R}m||^2_{L^4}\leq C ||\bar{R}m||^{\frac{1}{2}}_{L^2}||\nabla\bar {R}m||^{\frac{3}{2}}_{L^2}, \ \ \ 
||\nabla h||^2_{L^4}\leq C ||\nabla h||^{\frac{1}{2}}_{L^2}||\nabla^2 h||^{\frac{3}{2}}_{L^2}. 
\end{equation}
In conjunction with (\ref{bootd2g}) and  (\ref{bootd3g}), (\ref{6.11}) becomes 
\begin{equation}\label{6.15}
        \begin{aligned}
       &  \frac{d^{+}}{d t}\int_{M_t}2 |\nabla_p\bar{R}_{i0j0}|^2+|\nabla_p\bar{R}_{ijk0}|^2\\ 
       & \leq C P^2E^2_2 (||\nabla H||_{L^\infty}+||H||_{L^\infty}||h||_{L^\infty}+ ||\nabla^{2}H||_{L^2}^{\frac{1}{2}}||\nabla^{3}H||_{L^2}^{\frac{1}{2}}+P ||H||_{L^\infty}E_1^{\frac{1}{2}}E_2^{\frac{1}{2}}). 
       \end{aligned}
\end{equation}

 Integrating (\ref{6.15})  and using the proof of Lemma \ref{l5.3},  we obtain
\begin{equation}
    \int_{M_t}2 |\nabla_p\bar{R}_{i0j0}|^2+|\nabla_p\bar{R}_{ijk0}|^2\leq CE_2^2+ C(P^{7} \epsilon_0+P^3 \epsilon_0^{\frac{p_0}{3-p_0}})E_2^2.
\end{equation}
The proof is completed. 

\end{proof}

\subsection{The second fundamental form estimates}

\begin{lemma}
    Under the bootstrap assumptions,  we have 
    \begin{equation} \label{6.17}
    \begin{split} & \|h\|^4_{L^{\infty}_{t}L^{4}(M_{t})}\leq  C(1+P^{7}\epsilon_0+P^3 \epsilon_0^{\frac{p_0}{3-p_0}})E_1^2,\\
     &  ||\bar{R}_{ijkl}||_{L^2}^2\leq C(1+P^{7}\epsilon_0+P^3 \epsilon_0^{\frac{p_0}{3-p_0}})E_1^2,\\
   &  \|\nabla h\|^2_{L^{\infty}_{t}L^{2}(M_{t})}\leq C(1+P^{7}\epsilon_0+P^3 \epsilon_0^{\frac{p_0}{3-p_0}}) E_1^2.
      \end{split}\end{equation}
\end{lemma}

\begin{proof}  We need the evolution equation of the second fundamental form,  i.e.,  the second equation of  (\ref{412}). It follows from (\ref{412}) that  
\begin{equation} \label{6.18}
\begin{split}
\frac{d^{+}}{dt}\int_{M_t}|h|^4\leq &\int_{M_t}4 \nabla_{i}\nabla_j H h_{ij}|h|^2-4 H \bar{R}_{i0j0}h_{ij}|h|^2-4H tr (h^3)|h|^2+H^2|h|^4\\
\leq & - 4|||\nabla H||h|||_{L^2}^2+C\int_{M_t}|\nabla H||\nabla h||h|^2\\
&+C ||H||_{L^{\infty}}||h||_{L^{\infty}}(||\bar{R}_{i0j0}||_{L^2}||h||^2_{L^4}+||h||_{L^{4}}^4)\\
\leq & - 2|||\nabla H||h|||_{L^2}^2+CP^2E_1^2 (||H||_{L^{\infty}}||h||_{L^{\infty}}+||\nabla H||_{L^{\infty}}).
\end{split}
\end{equation}

Integrating (\ref{6.18}) yields 
\begin{equation} \label{6.19}
\begin{split}
||h||^4_{L^4(M_t)}+ 2 \int_{0}^t|||\nabla H||h|||^2_{L^2}ds 
\leq &E_1^2+ C(P^{7}\epsilon_0+P^3 \epsilon_0^{\frac{p_0}{3-p_0}})E_1^2,
\end{split}
\end{equation}
which implies  the first inequality in (\ref{6.17}).

Since  $\dim M_t=3$, we have 
    \begin{equation}\label{R=Ric}
  R_{ijkl}=R_{ik}g_{jl}+R_{jl}g_{ik}-R_{il}g_{jk}-R_{jk}g_{il}-\frac{R}{2}(g_{ik}g_{jl}-g_{il}g_{jk}),
\end{equation} 
 the second inequality follows from  (\ref{Gauss}) (\ref{tr Gauss}) (\ref{R=Ric}) and  (\ref{R0 L2}):  \begin{equation} \label{6.21}
\begin{split}
      ||\bar{R}_{ijkl}||_{L^2}^2\leq& C||\bar{R}_{i0j0}||^2_{L^2}+C||h||^{4}_{L^4}\\
     \leq& CE_1^2+ C(P^{7}\epsilon_0+P^3 \epsilon_0^{\frac{p_0}{3-p_0}})E_1^2.
\end{split}
\end{equation}

To prove the third  inequality in (\ref{6.17}), we need to use  Simon's identity:
 \begin{equation}\label{Simon}
    \begin{split}
    \triangle h_{ij}=&\nabla_i\nabla_jH+R_{il}h_{lj}-R_{ikjl}h_{kl}+\nabla_k(\bar{R}_{kij0})\\
    =& \nabla_i\nabla_jH-(\bar{R}_{ikjl}-h_{ij}h_{kl}+h_{il}h_{kj})h_{kl}+(\bar{R}_{i0k0}-Hh_{ik}+h^2_{ik})h_{kj}+\nabla_k(\bar{R}_{kij0}).\end{split}
        \end{equation}
Multiplying  (\ref{Simon}) with $h_{ij}$ and integrating by parts,   we have 
         \begin{equation} \label{6.23}
\int_{M_t}|\nabla h|^2
\leq C ||\bar{R}_{0ijk}||^2_{L^2}+C ||\nabla H||^2_{L^2}+C||\bar{R}_{ijkl}||_{L^2}||h||_{L^4}^2+C||h||_{L^4}^4.
        \end{equation}

In order to estimate the term $||\nabla H||^2_{L^2}$ in (\ref{6.23}),  we integrate     (\ref{5.11}), it follows that  
\begin{equation} \label{6.24}
||\nabla H||^2_{L^2(M_t)}\leq CE_1^2+CP^3 \epsilon_0^{\frac{p_0}{3-p_0}}E_1^2.
\end{equation}

Substituting (\ref{R0 L2}) (\ref{6.24})  (\ref{6.21}) and (\ref{6.19}) in (\ref{6.23}), we obtain the third inequality in (\ref{6.17}). 
The proof is completed.

\end{proof}

\begin{lemma} \label{l6.4}
    Under the bootstrap assumptions,  we have 
\begin{equation} \label{6.25}
\begin{split}
     &   \int_{M_t}|\nabla \bar{R}m|^2\leq  C(1+P^7\epsilon_0+P^3 \epsilon_0^{\frac{p_0}{3-p_0}})E_2^2,\\
      &   \int_{M_t}|\nabla^2 h|^2+|h|^6+|\bar{R}m|^3+|\nabla h|^3\leq C(1+P^{7}\epsilon_0+P^3 \epsilon_0^{\frac{p_0}{3-p_0}})E_2^2.\\
\end{split}
\end{equation}

\end{lemma}

\begin{proof}
 To prove the first inequality of (\ref{6.25}), in view of (\ref{6.7}), we only need to handle the purely spacial components of the curvature, i.e. $|| \nabla \bar{R}_{ijkl}||_{L^2}$. Because of   (\ref{Gauss}) (\ref{tr Gauss})  (\ref{R=Ric}) and (\ref{6.7}), 
\begin{equation} \label{6.26}
\begin{split}
|| \nabla \bar{R}_{ijkl}||^2_{L^2}& \leq C  || \nabla \bar{R}_{0i0j}||^2_{L^2}+ C ||\nabla h\ast h||^2_{L^2}\\
& \leq C(1+P^7 \epsilon_0+ P^3 \epsilon_0^{\frac{p_0}{3-p_0}}) E^2_2+ C ||h||^2_{L^6} ||\nabla h||^2_{L^3}.
\end{split}
\end{equation}
Hence, the first inequality of (\ref{6.25}) is reduced to the estimates of  $||h||_{L^6}$ and $||\nabla h||_{L^3}$.

By direct computations, we have 
\begin{equation}\label{6.27}
\begin{split}
    \frac{d^{+}}{dt}\int_{M_t}|h|^6\leq  &3\int_{M_{t}}|h|^4(\nabla^2H+H(h\ast h-\bar{R}))\ast h+\int_{M_t}H^2|h|^6\\
    \leq & C\int_{M_t}|h|^4|\nabla h||\nabla H|+|H|(|h|^7+|h|^5|\bar{R}|)\\
    \leq & C||\nabla H||_{L^{\infty}}(\int_{M_t}|h|^6)^{\frac{2}{3}}(\int_{M_{t}}|\nabla h|^3)^{\frac{1}{3}}\\
    &+C||H||_{L^{\infty}}||h||_{L^{\infty}}(\int_{M_t}|h|^6+(\int_{M_t}|h|^6)^{\frac{2}{3}}(\int_{M_{t}}|\bar{R}|^3)^{\frac{1}{3}})\\
    \leq & C(||H||_{L^{\infty}}||h||_{L^{\infty}}+||\nabla H||_{L^{\infty}})P^2E_2^2,
\end{split}
\end{equation}

and 

\begin{equation}\label{6.28}
\begin{split}
    \frac{d^{+}}{dt}\int_{M_t}|\nabla h|^3\leq  &\int_{M_{t}}|\nabla h|(\nabla^{3}H+\nabla(H(h^2+\bar{R}))+\Gamma^\prime h)\ast \nabla h+\int_{M_t}H^2|\nabla h|^3\\
   \leq & -\int_{M_{t}}|\nabla h||\nabla^2 H|^2+C\int_{M_t}|\nabla^2 H|(|\nabla h||\nabla^2 h|+|h||\bar{R}||\nabla h|)\\
   &+ C\int_{M_t}|\nabla H||\nabla h|(|h|^2|\nabla h|+|\bar{R}||h|^2)+|H||h||\nabla h|^3\\
   \leq & -\int_{M_{t}}|\nabla h||\nabla^2 H|^2+C||\nabla^2 H||_{L^{2}}^{\frac{1}{2}}||\nabla^3 H||_{L^{2}}^{\frac{1}{2}}[||\nabla^2 h||_{L^2}\\
   &+(\int_{M_t}|\nabla h|^3)^{\frac{1}{3}}(\int_{M_t}|h|^6)^{\frac{1}{6}}||\nabla \bar{R}||_{L^2}]\\
   &+C||\nabla H||_{L^{\infty}}[(\int_{M_t}|\nabla h|^2|h|^2\int_{M_t}|\bar{R}|^2|h|^2)^{\frac{1}{2}}+\int_{M_t}|\nabla h|^2|h|^2]\\
   \leq& C(||H||_{L^{\infty}}||h||_{L^{\infty}}+||\nabla^2 H||_{L^{2}}^{\frac{1}{2}}||\nabla^3 H||_{L^{2}}^{\frac{1}{2}})P^2E_2^2.
\end{split}
\end{equation}

Integrating (\ref{6.27}) and (\ref{6.28}) yields 
\begin{equation} \label{6.29}
||h||_{L^6}^6+||\nabla h||^3_{L^3} \leq C(1+P^{7}\epsilon_0+P^3 \epsilon_0^{\frac{p_0}{3-p_0}})E_2^2.
\end{equation}
This implies 
\begin{equation}
||h||_{L^6}^2||\nabla h||^2_{L^3} \leq C(1+P^{7}\epsilon_0+P^3 \epsilon_0^{\frac{p_0}{3-p_0}})E_2^2,
\end{equation}
which together with (\ref{6.26}) gives   the first inequality of (\ref{6.25}).

Now we proceed to estimate $||\bar{R}_{0\ast\ast\ast}||_{L^3}$. 
It follows from (\ref{6.4}) that 
\begin{equation} 
\begin{split}
    \frac{d^{+}}{dt}\int_{M_t}|\bar{R}_{0\ast\ast\ast}|^3\leq  &\frac{3}{2}\int_{M_{t}}|\bar{R}|(\nabla (H\bar{R})+Hh\ast \bar{R}+\nabla H \ast\bar{R})\ast \bar{R}+\int_{M_t}H^2|\bar{R}|^3\\
    \leq & C (||H||_{L^{\infty}}||h||_{L^{\infty}}+||\nabla H||_{L^{\infty}}) ||\bar{R}||_{L^3}^3\\& +C||H||_{L^{\infty}}||\nabla \bar{R}||_{L^2}||\bar{R}||^{\frac{1}{2}}_{L^2}||\bar{R}||^{\frac{3}{2}}_{L^6}.\end{split}
\end{equation}

Applying  the Sobolev inequality  $||\bar{R}||_{L^6}\leq C||\nabla \bar{R}||_{L^2}$ and bootstrap assumptions  
yields 
\begin{equation}\label{6.32}
\begin{split}
    \frac{d^{+}}{dt}\int_{M_t}|\bar{R}_{0\ast\ast\ast}|^3\leq  C(||H||_{L^{\infty}} PE_1^{\frac{1}{2}}E^{\frac{1}{2}}_{2}+||\nabla H||_{L^{\infty}})P^2E_2^2.
    \end{split}
\end{equation}

Integrating (\ref{6.32}) gives \begin{equation} \label{6.33}
\int_{M_t}|\bar{R}_{0\ast\ast\ast}|^3 \leq   C(1+P^7\epsilon_0+P^3 \epsilon_0^{\frac{p_0}{3-p_0}})E_2^2.\end{equation}

From  (\ref{Gauss}) (\ref{tr Gauss})  (\ref{R=Ric}) (\ref{6.29}) and (\ref{6.33}), the  spacial component  $\bar{R}_{ijkl}$ of the spacetime curvature can be estimated as follows   

\begin{equation}\label{6.34}
    \int_{M_t}|\bar{R}_{ijkl}|^3\leq C\int_{M_t}|\bar{R}_{0\ast\ast\ast}|^3+|h|^6 \leq  C(1+P^7\epsilon_0+P^3 \epsilon_0^{\frac{p_0}{3-p_0}})E_2^2. \end{equation}

Now it only remains to bound $||\nabla^2 h||_{L^2}$.

Note that it is convenient to apply  (\ref{Simon}) directly:  
\begin{equation}\label{6.35}
    \begin{split}
    ||\triangle h_{ij}||^2_{L^2}\leq & C ||\nabla^2H||^2_{L^2}+C||\nabla\bar{R}||_{L^2}^2+ C\int_{M_t}|\bar{R}m|^2|h|^2+|h|^6\\
     \leq &  C (1+P^7\epsilon_0+P^3 \epsilon_0^{\frac{p_0}{3-p_0}})E_2^2+C ||\nabla^2H||^2_{L^2}. \end{split}
        \end{equation}
To estimate the term $||\nabla^2H||^2_{L^2}$,  we need to use (\ref{d2H L2 evo}):  
 \begin{equation} \label{6.36}
\begin{aligned}
   & \frac{d^{+}}{d t}\int_{M_{t}}|\nabla^{2}H|^{2} \leq   \int_{M_{t}}-|\nabla^{3}H|^{2}+\nabla Rm \ast \nabla^2 H \ast \nabla H+C |Rm\ast \nabla H|^2\\ 
           &\ \ \ \ \ \ \  \ \ \ \ \ \ \ \ \ \ \ \ \ \ \ +C |\nabla(|h|^2H)|^2+\nabla H\ast \nabla (Hh) \ast \nabla^2 H  + H\ast h \ast \nabla^2 H \ast \nabla^2 H\\
                \leq&-||\nabla^{3}H||_{L^2}^{2}+C ||\nabla H||_{L^{\infty}}(||\nabla \bar{R}m||_{L^2}||\nabla^2 h||_{L^2}+||\bar{R}m||_{L^3}^{2}||\nabla h||_{L^3}\\ & +||h||_{L^6}^{4}||\nabla h||_{L^3}+ ||\nabla^2 h||_{L^2}||\nabla h \ast h||_{L^2}) 
    +C ||H||_{L^{\infty}}||h||_{L^{\infty}}(||\nabla h \ast h||^2_{L^2}+||\nabla^2 h||_{L^2}^2)\\
           \leq & \ C(||H||_{L^{\infty}}||h||_{L^{\infty}}+||\nabla H||_{L^{\infty}})P^2E_2^2.
\end{aligned}
\end{equation}
Integrating (\ref{6.36}) gives 
\begin{equation}\label{6.37}
    ||\nabla^{2}H||_{L^2}^{2}\leq C(1+P^{7}\epsilon_0+P^3 \epsilon_0^{\frac{p_0}{3-p_0}})E_2^2, 
\end{equation}
which together with (\ref{6.35}) implies 
 \begin{equation}\label{6.38}
    \begin{split}
    ||\triangle h_{ij}||^2_{L^2} \leq &  C (1+P^7\epsilon_0+P^3 \epsilon_0^{\frac{p_0}{3-p_0}})E_2^2.\end{split}
        \end{equation}

On the other hand, it follows from (\ref{3.13}) that 
 \begin{equation} \label{6.39}
    \begin{split}
            ||\nabla^{2} h||^{2}_{L^2} 
            & \leq C ||\triangle h||^2_{L^2}+C (||\bar{R}m||_{L^3}+||h||^2_{L^6})||\nabla h||^2_{L^3}+C ||\bar{R}m||_{L^3}||h||^2_{L^6}||\nabla^{2} h||_{L^2}. 
    \end{split}
    \end{equation}        
   Substituting (\ref{6.38}) (\ref{6.34}) (\ref{6.33}) (\ref{6.29}) into (\ref{6.39}), we obtain 
   \begin{equation}
      ||\nabla^2h||^2_{L^2}\leq C(1+P^7\epsilon_0+P^3 \epsilon_0^{\frac{p_0}{3-p_0}})E_2^2+C(1+P^7\epsilon_0+P^3 \epsilon_0^{\frac{p_0}{3-p_0}})E_2||\nabla^2h||_{L^2},
   \end{equation}   
      which implies 
\begin{equation}
      ||\nabla^2h||^2_{L^2}\leq C(1+P^7\epsilon_0+P^3 \epsilon_0^{\frac{p_0}{3-p_0}})E_2^2.
   \end{equation}   

This completes   the proof of Lemma \ref{l6.4}. 
  \end{proof}

 \section{Improving bootstrap assumptions}

 Based on the preceding estimates,  we can show that the  bootstrap assumptions can be improved or verified. 
 
 To  begin with,   we come to the estimate of  the evolving  metric $g_{ij}(x,t)$.   For any tangent vector $V\in T_{x}M_0$,  according to  the evolution equation $\partial_t g=2Hh$,  we have  
$$
 -2||Hh||_{L^{\infty}}g(V,V) \leq \partial_tg(V,V) \leq 2||Hh||_{L^{\infty}}g(V,V).$$
Integrating the above inequalities and using (\ref{5.23}), we have 

\begin{equation}\label{7.1}
e^{-CP \epsilon_0^{\frac{p_0}{3-p_0}}} g_0(V,V) \leq g(V,V)\leq e^{CP \epsilon_0^{\frac{p_0}{3-p_0}}}g_0(V,V),
\end{equation}
which implies 
\begin{equation}\label{7.2}
e^{-\frac{1}{100}} g_0(V,V) \leq g(V,V)\leq e^{\frac{1}{100}}g_0(V,V),
\end{equation}
if we  choose $P$ so that  $CP \epsilon_0^{\frac{1}{2}}\leq \frac{1}{100}$.  
Since $e^{\frac{1}{100}}<2$,  (\ref{bootd1g})  always  holds  as long as     $CP \epsilon_0^{\frac{1}{2}}\leq \frac{1}{100}$.

Now we summarize what we have proved under bootstrap assumptions.

Suppose the mean curvature flow satisfies  the bootstrap assumptions (\ref{bootd1g})(\ref{bootd2g})(\ref{bootd3g}) on a time interval $[0,\bar{T})$.

 In Propositions \ref{R0 L2}, \ref{dR0 L2},  we prove that there exists a positive  constant $C_1$ depending only on $\alpha_0$ and $\alpha_1$ such that 

\begin{equation}\label{7.3}
        \int_{M_t}|\bar{R}m|^2+|\nabla h|^2+|h|^4\leq  C_1(1+P^{7}\epsilon_0+P^3 \epsilon_0^{\frac{p_0}{3-p_0}})E_1^2,
\end{equation}
\begin{equation}\label{7.4}
    \int_{M_t}|\nabla \bar{R}m|^2+|\nabla^2 h|^2+|h|^6+|\bar{R}m|^3+|\nabla h|^3\leq C_1(1+P^{7}\epsilon_0+P^3 \epsilon_0^{\frac{p_0}{3-p_0}})E_2^2.
\end{equation}
By  choosing suitably large $P$, (e.g. $P=\epsilon_0^{-\frac{1}{8}}$),   we have 

\begin{equation}
    C_1(1+P^{7}\epsilon_0+P^3 \epsilon_0^{\frac{p_0}{3-p_0}})< \frac{1}{2}P^2,
\end{equation}
as $\epsilon_0$ is small. If $P=\epsilon_0^{-\frac{1}{8}}$, then $CP \epsilon_0^{\frac{1}{2}}\leq C \epsilon_0^{\frac{3}{8}}\leq \frac{1}{100}$ also holds.

Therefore, (\ref{bootd1g}) (\ref{bootd2g}) and (\ref{bootd3g}) will always  hold  for  $P=\epsilon_0^{-\frac{1}{8}}$ on an interval $[0,T)$, as long as the solution exists on $[0, T)$.

\begin{theorem} \label{t7.1}
    Under the assumptions of Theorem \ref{maxfol},  there exists a positive  constant $C=C(\alpha_0,\alpha_1)$ such that the solution to the mean curvature flow  (\ref{mcf}) satisfies 
    \begin{equation} \label{7.6}
    \begin{split}
     &  \ \ \ \ \ \ \frac{1}{2} g_0    \leq g_t \leq 2 g_0,\\
      &  \int_{M_t}|\bar{R}m|^2+|\nabla h|^2+|h|^4  \leq   C \int_{M_0}|\bar{R}m|^2+|\nabla h|^2+|h|^4,\\
 &  \int_{M_t}|\nabla \bar{R}m|^2+|\nabla^2 h|^2+|h|^6+|\bar{R}m|^3+|\nabla h|^3
 \\& \leq C \int_{M_0}|\nabla \bar{R}m|^2+|\nabla^2 h|^2+|h|^6+|\bar{R}m|^3+|\nabla h|^3.
\end{split}
\end{equation}
\end{theorem}

\begin{proof}
When   $E_1\leq \epsilon_0$ and $E_2\leq \epsilon_0$,   the result follows from the above arguments.  Otherwise, we may scale the solution $(g,h)$ to $(\lambda g,\sqrt{\lambda}h)$ so that the conditions $E_1\leq \epsilon_0$ and $E_2\leq \epsilon_0$ hold for $(\lambda_0 g,\sqrt{\lambda_0}h)$  for some large  $\lambda_0$. In view of the scaling invariance of (\ref{7.6}), this completes the proof of Theorem \ref{t7.1}. 
\end{proof}

\section{Higher order estimates}
  Note that we have obtained a uniform $W^{1,2}$-estimate for the spacetime curvature tensor $\bar{R}m$ in Theorem \ref{t7.1}, whereas this is not enough to deduce a uniform  $L^{\infty}$-bound for $\bar{R}m$,  which was used in the proof of short time existence of mean curvature flow in Section 2.2.  In view of the Sobolev embedding theorem, the boundedness of $||\bar{R}m||_{L^{\infty}}$ can be   guaranteed by  a uniform $W^{2,2}$-estimate of $\bar{R}m$. We therefore  need  to derive some  higher-order estimates of the  curvature and the second fundamental form.

\begin{lemma} \label{l8.1}
    Under  the assumptions of Theorem \ref{maxfol},  the mean curvature flow (\ref{mcf}) satisfies 
\begin{equation} \label{8.1}
 \int_{M_{T}}|\nabla^{3}H|^{2}+\int_{0}^{T}(\int_{M_{t}} |\nabla^{4} H|^{2}) dt\leq  ||\nabla^3 H||^2_{L^2(M_0)}+C(E_2^{\frac{4}{3}}+E_1^4)E_2^2,
\end{equation}
and 
\begin{equation} \label{8.2}
    \begin{split}
& \int_{M_{T}}|\nabla^{3}H|^{2}+\int_{\frac{15}{32}T}^{T}(\int_{M_{t}} |\nabla^{4} H|^{2}) dt\\ 
     \leq &C(E_2^{\frac{4}{3}}+E_1^4+T^{-1})||H||_{L^{p_0}(M_0)}^2\{(||H||_{L^{p_0}(M_0)} E^{\frac{1}{2}}_1E^{\frac{1}{2}}_2{T^{-\frac{3}{2p_0}}}+T^{-1} + E_1^4) \\
&\times( E_1^2T^{1-\frac{3}{p_0}}+ T^{\frac{1}{2}-\frac{3}{p_0}})+ E_1^3E_2 T^{1-\frac{3}{p_0}}\}+C E_1E_2^3||H||_{L^{p_0}(M_0)}^2T^{1-\frac{3}{p_0}}\\
&+C(E_1^2E_2^2+E_2^{\frac{8}{3}})||H||_{L^{p_0}(M_0)}^2( E_1^2T^{1-\frac{3}{p_0}}+T^{\frac{1}{2}-\frac{3}{p_0}}).
    \end{split}
\end{equation}
        \end{lemma}
\begin{proof}  First of all, we need to derive the evolution equation of $\nabla^3H$. 

Differentiating  the formula  $$\nabla_i\nabla_j\nabla_k H=\partial_i(\nabla_j\nabla_kH)-\Gamma^{p}_{ij}\nabla_p\nabla_kH-\Gamma^{p}_{ik}\nabla_j\nabla_pH$$ with respect to $t$, we obtain 

$$\frac{\partial}{\partial t}\nabla_i\nabla_j\nabla_k H=\nabla_i (\frac{\partial }{\partial t} \nabla_j\nabla_k H)+\Gamma^{\prime}\ast \nabla^2 H.$$

Due to the  Ricci formula,
\begin{equation}
\triangle \nabla_i\nabla_j\nabla_kH=\nabla_i \triangle \nabla_j\nabla_k H+Rm\ast \nabla^3H+\nabla Rm\ast \nabla^2H,
\end{equation}
we have 
\begin{equation}\label{8.4}
(\frac{\partial}{\partial t}-\triangle) \nabla^3 H=\nabla(\frac{\partial}{\partial t}-\triangle) \nabla^2 H+Rm\ast \nabla^3H+(\nabla Rm+\Gamma^{\prime})\ast \nabla^2H.
\end{equation}
On the other hand, one can rewrite (\ref{d2H evo}) as follows: 
\begin{equation}\label{8.5}
\begin{split}
(\frac{\partial}{\partial t}-\triangle) \nabla^2 H= Rm\ast \nabla^2H+(\nabla Rm+\Gamma^{\prime})\ast \nabla H+\nabla^2(H|h|^2).
\end{split}
\end{equation}
Substituting (\ref{8.5}) into (\ref{8.4}) yields 
\begin{equation}
\begin{split}
(\frac{\partial}{\partial t}-\triangle) \nabla^3 H=& Rm\ast \nabla^3H+(\nabla Rm+\Gamma^{\prime})\ast \nabla^2 H+(\nabla^2 Rm+\nabla \Gamma^{\prime})\ast \nabla H\\ & +\nabla^3(H|h|^2),
\end{split}
\end{equation}
which implies 
\begin{equation}\label{8.7}
\begin{aligned}
   \frac{d^{+}}{d t}\int_{M_{t}}|\nabla^{3}H|^{2} \leq  & -2\int_{M_{t}}|\nabla^{4}H|^{2}+\int_{M_{t}} Rm\ast \nabla^3H \ast \nabla^3H+(\nabla Rm+\Gamma^{\prime}) \ast \nabla^2 H \ast \nabla^3 H\\ 
           &\ \ \ \ +(\nabla^2 Rm+\nabla \Gamma^{\prime})\ast \nabla H\ast \nabla^3H+\nabla^3 H \ast \nabla^3 (|h|^2H) \\
        &\ \ \ \ + \int_{M_{t}} H  h\ast  \nabla^3 H\ast \nabla^3 H.
\end{aligned}
\end{equation}

Now we  estimate the right hand side of (\ref{8.7}).  By integration by parts, applying  H$\ddot{o}$lder  and Sobolev inequalities,  we obtain 
\begin{equation}\label{8.8}
\begin{aligned}
   & \int_{M_t}(\nabla^2 Rm+\nabla \Gamma^{\prime})\ast \nabla H\ast \nabla^3H\\ = & \int_{M_t}(\nabla Rm+\Gamma^{\prime}) \ast \nabla^2 H \ast \nabla^3 H +(\nabla Rm+\Gamma^{\prime}) \ast \nabla H \ast \nabla^4 H\\
     \leq & \  {\delta} ||\nabla^4H||^2_{L^2}+C_{\delta} E_2^2||\nabla^2H||_{L^2}||\nabla^3H||_{L^2}+C_{\delta} E_2^{\frac{4}{3}}||\nabla^3H||_{L^{2}}^{2},\end{aligned}
\end{equation}
\begin{equation}\label{8.9}
\begin{aligned}    \int_{M_t} \nabla^3H \ast \nabla^3 (H|h|^2)= &  \int_{M_t} \nabla^2(H|h|^2)\ast \nabla^4 H\\
     \leq &  \  \delta  ||\nabla^4H||^2_{L^2}+C_{\delta}E_1^2E_2^2||\nabla^2 H||^2_{L^2}  +C_{\delta}E_1E_2^3||H||^2_{L^{\infty}}\\ & +C_{\delta}E_1^3E_2||\nabla^2 H||_{L^2}||\nabla^3 H||_{L^2},
    \end{aligned}
\end{equation}
and 
\begin{equation}\label{8.10}
\begin{aligned}
 &  \int_{M_{t}} (Rm+Hh) \ast \nabla^3H \ast \nabla^3H  \\ \leq & \ C ||Rm+Hh||_{L^2}||\nabla^3H||^{\frac{1}{2}}_{L^2}||\nabla^4H||^{\frac{3}{2}}_{L^2}\\  \leq & \ 
 \delta ||\nabla^4H||^2_{L^2}+ C_{\delta}  ||Rm+Hh||^4_{L^2}||\nabla^3H||^2_{L^2}.\end{aligned}
\end{equation}
Combining (\ref{8.7})(\ref{8.8})(\ref{8.9}) and (\ref{8.10}), we have 
\begin{equation}\label{8.11}
\begin{aligned}
   \frac{d^{+}}{d t}\int_{M_{t}}|\nabla^{3}H|^{2} \leq  & -||\nabla^{4}H||_{L^2}^{2}+C(E_1^4+E_2^{\frac{4}{3}}) ||\nabla^3H||^2_{L^2}\\
   &+C(E_1^2E_2^2+E_2^{\frac{8}{3}})||\nabla^2H||^2_{L^2}+CE_1E_2^3 ||H||^2_{L^{\infty}}.\end{aligned}
\end{equation}

By similar technique as in (\ref{5.12}) and using (\ref{4.17})(\ref{5.2})(\ref{5.13}),  we obtain (\ref{8.2}).

To derive  (\ref{8.1}), we integrate (\ref{5.11}) and (\ref{6.36}) from 0 to $T$: 
\begin{equation}\label{8.12}
\begin{split}
       &\int_{M_{T}}|\nabla H|^{2}+\int_{0}^{T}(\int_{M_{t}} |\nabla^{2} H|^{2}) dt\leq C E_1^{2},\\
        &\int_{M_{T}}|\nabla^{2}H|^{2}+\int_{0}^{T}(\int_{M_{t}} |\nabla^{3} H|^{2}) dt\leq C E_2^{2}.
        \end{split}
\end{equation}
Integrating (\ref{8.11}) and using (\ref{8.12}) (\ref{5.23}) gives (\ref{8.1}). 
 Lemma \ref{l8.1} thus follows

\end{proof}

\begin{proposition}\label{p8.2}
Under the assumptions of Theorem \ref{maxfol},  if at time $t=0$, $E_1\leq \epsilon_0$, $E_2\leq \epsilon_0$  and 
\begin{equation}\label{8.13}
\begin{split}
        ||H||_{L^p(M_0)}  E_2^{\frac{2}{p_0}-\frac{1}{3}} \leq  \epsilon_0,
\end{split}
\end{equation} 
\begin{equation}
\begin{split}
 ||\nabla^2 \bar{R}m||^2_{L^2}+||\nabla^3 h ||^2_{L^2} \leq \Lambda^2,\end{split}
\end{equation}
hold,  then for any $t>0$, we have 
\begin{equation} \label{8.15}
 ||\nabla^2 \bar{R}m||^2_{L^2}+||\nabla^3 h ||^2_{L^2} \leq C_2(\Lambda^2+\epsilon_0^2), 
\end{equation} 
for some constant $C_2=C_2(\alpha_0,\alpha_1)$. 

\end{proposition}

\begin{proof}  We assume 
\begin{equation} \label{8.16}
||\nabla^2 \bar{R}m||^2_{L^2}+||\nabla^3 h ||^2_{L^2} \leq Q(\Lambda^2+\epsilon_0^2), \end{equation} 
holds for some $Q\geq 1$ and  all $0\leq t<T$.

By differentiating (\ref{R00 evo}) and (\ref{R0 evo}) twice  and commuting derivatives, let $I$ be a multi-index with $|I|=2$,  we have 
\begin{equation}
 \begin{aligned}
        \frac{\partial}{\partial t}\nabla_{I}\bar{R}_{i0j0}
        =& H\nabla_{l}\nabla_{I}\bar{R}_{ilj0}+\sum_{i\leq 2} \nabla^{i}\bar{R}m \ast \nabla^{3-i} H+\sum_{i+j\leq2} \nabla^{i}\bar{R}m \ast \nabla^{2-i-j}H \ast \nabla^j h\\
        & +\sum_{i+j\leq 1} \nabla^{i}\bar{R}m \ast \nabla^{1-i-j}H \ast \nabla^j\bar{R}m +\sum_{i+j\leq 1} \nabla^{i}\bar{R}m \ast \nabla^{1-i-j}H \ast  \nabla^j(h\ast h),\end{aligned}
\end{equation}
\begin{equation}
    \begin{aligned}
     \frac{\partial}{\partial t}\nabla_{I}\bar{R}_{ijk0}=& H (\nabla_j\nabla_I \bar{R}_{i0k0}-\nabla_i\nabla_I\bar{R}_{j0k0}) + \sum_{i\leq 2} \nabla^{i}\bar{R}m \ast \nabla^{3-i} H+\sum_{i+j\leq 2} \nabla^{i}\bar{R}m \ast \nabla^{2-i-j}H \ast \nabla^j h\\
        & +\sum_{i+j\leq 1} \nabla^{i}\bar{R}m \ast \nabla^{1-i-j}H \ast \nabla^j\bar{R}m +\sum_{i+j\leq 1} \nabla^{i}\bar{R}m \ast \nabla^{1-i-j}H \ast  \nabla^j(h\ast h).     \end{aligned}
\end{equation}

Then, we have 
\begin{equation} \label{8.19}
\begin{split}
 & \frac{d}{d t} \int_{M_{t}} 2 |\nabla_{I}\bar{R}_{i0j0}|^2+|\nabla_{I}\bar{R}_{ijk0}|^2\\   \leq  &  C ||\nabla^{2}\bar{R}m||_{L^{2}}^2(||\nabla H||_{L^{\infty}}+||Hh||_{L^{\infty}}+||H||_{L^{\infty}}||\bar{R}m||_{L^3})\\
 &+C||\nabla^2 H||_{L^2}^{\frac{1}{2}}||\nabla^3H||^{\frac{1}{2}}_{L^2}||\nabla^2\bar{R}m||_{L^2}(|| \bar{R}m||_{L^{\infty}}||\nabla h||_{L^2}+||h||_{L^{\infty}}||\nabla \bar{R}m||_{L^2})\\
& + C||\nabla^{3} H||_{L^2}||\nabla\bar{R}m||_{L^2}^{\frac{1}{2}}||\nabla^2\bar{R}m||_{L^2}^{\frac{3}{2}}+\int_{M_t}H\nabla^{2}(Rm\ast h)+H\nabla(Rm\ast Rm)\ast\nabla^{2}\bar{R}m.
  \end{split}
\end{equation}

Combining (\ref{8.19}) and the following estimates (\ref{8.20}) (\ref{8.21}) and (\ref{8.22})
\begin{equation} \label{8.20}
    \begin{split}
    \int H\nabla \bar{R}m\ast \nabla h\ast \nabla^2\bar{R}m\leq & ||H||_{L^{\infty}}||\nabla^2\bar{R}m||_{L^2}||\nabla\bar{R}m||_{L^6}||\nabla h||_{L^3}\\
    \leq & C||H||_{L^{\infty}}||\nabla^2\bar{R}m||_{L^2}^2||\nabla h||_{L^2}^{\frac{1}{2}}||\nabla^2 h||_{L^2}^{\frac{1}{2}}\\
    \leq & C||H||_{L^{\infty}}E_1^{\frac{1}{2}}E_2^{\frac{1}{2}}Q^2(\Lambda+\epsilon_0)^2,
    \end{split}
\end{equation}
\begin{equation} \label{8.21}
    \begin{split}
        \int H\bar{R}m\ast\nabla^2h\ast\nabla^2\bar{R}\leq & ||H||_{L^{\infty}}||\nabla^2\bar{R}m||_{L^2}||\bar{R}m||_{L^3}||\nabla^2 h||_{L^6}\\
        \leq & C ||H||_{L^{\infty}}E_1^{\frac{1}{2}}E_2^{\frac{1}{2}}Q^2(\Lambda+\epsilon_0)^2,
    \end{split}
\end{equation}
\begin{equation} \label{8.22}
    \begin{split}
        \int H\nabla \bar{R}m\ast \bar{R}m\ast\nabla^2\bar{R}\leq & ||H||_{L^{\infty}}||\nabla^2\bar{R}m||_{L^2}||\bar{R}m||_{L^3}||\nabla \bar{R}m||_{L^6}\\ \leq & C ||H||_{L^{\infty}}E_1^{\frac{1}{2}}E_2^{\frac{1}{2}}Q^2(\Lambda+\epsilon_0)^2,
    \end{split}
\end{equation}
we have 
\begin{equation} \label{8.23}
\begin{split}
 & \frac{d}{d t} \int_{M_{t}} 2 |\nabla_{I}\bar{R}_{i0j0}|^2+|\nabla_{I}\bar{R}_{ijk0}|^2\\ 
 \leq & \  C Q^2(\Lambda+\epsilon_0)^2 (||\nabla H||_{L^{\infty}}+||Hh||_{L^{\infty}}+||H||_{L^{\infty}}E_1^{\frac{1}{2}}E_2^{\frac{1}{2}})+C Q(\Lambda+\epsilon_0)E_1^{\frac{1}{2}}E_2^{\frac{3}{2}}||\nabla H||_{L^{\infty}}\\
 &+C Q^{\frac{3}{2}}(\Lambda+\epsilon_0)^{\frac{3}{2}}E_1E_2^{\frac{1}{2}}||\nabla^2 H||_{L^2}^{\frac{1}{2}}||\nabla^3H||^{\frac{1}{2}}_{L^2}+C Q^{\frac{3}{2}}(\Lambda+\epsilon_0)^{\frac{3}{2}}E_2^{\frac{1}{2}}||\nabla^{3}H||_{L^2}. 
 \end{split}
\end{equation}

To estimate the time integral of $||\nabla^{3}H||_{L^2}$, we need to employ   (\ref{8.1}) (\ref{8.2}) to obtain 
\begin{equation} \label{8.24}
    \begin{split}
    \int_{0}^{T}||\nabla^3 H||_{L^2}\leq & C||H||_{L^{p_0}(M_0)} (||H||_{L^{p_0}(M_0)} E^{\frac{1}{2}}_1E^{\frac{1}{2}}_2{a^{-\frac{3}{2p_0}}}+a^{-1} + E_1^4)^{\frac{1}{2}} ( E_1a^{1-\frac{3}{2p_0}}+a^{\frac{3}{4}-\frac{3}{2p_0}})\\ & +C||H||_{L^{p_0}(M_0)} E_1^{\frac{3}{2}}E_2^{\frac{1}{2}} a^{1-\frac{3}{2p_0}}+CE_2{a}^{\frac{1}{2}},
    \end{split}
    \end{equation}
    which holds for all $a\leq T$.   If $T\geq E_2^{-\frac{4}{3}}$, we choose $a=E_2^{-\frac{4}{3}}$ in (\ref{8.24}). If $T< E_2^{-\frac{4}{3}}$, one can prove   
     \begin{equation}
    \begin{split}
    \int_{0}^{T}||\nabla^3 H||_{L^2}\leq & CE_2{a}^{\frac{1}{2}},
    \end{split}
    \end{equation}    
 which    holds for $a= E_2^{-\frac{4}{3}}$. For both cases,  we have  \begin{equation}  
    \begin{split}
  \int_{0}^{T}||\nabla^3 H||_{L^2}\leq & C\tilde{E}_1, 
    \end{split}
\end{equation}
where $\tilde{E}_1=\tilde{E}_1(g,h)$ is defined as follows:
\begin{equation} \label{8.26}  
    \begin{split}
  \tilde{E}= & E_2^{\frac{1}{3}}+||H||_{L^{p_0}(M_0)} E_1^{\frac{3}{2}}E_2^{\frac{2}{p_0}-\frac{5}{6}}\\
    &+||H||_{L^{p_0}(M_0)}(||H||_{L^{p_0}(M_0)} E^{\frac{1}{2}}_1E^{\frac{2}{p_0}+\frac{1}{2}}_2+E_2^{\frac{4}{3}} + E_1^4)^{\frac{1}{2}} ( E_1E_2^{\frac{2}{p_0}-\frac{4}{3}}+ E_2^{\frac{2}{p_0}-1}), 
    \end{split}
\end{equation}
which  satisfies   the following scaling invariant  property:
\begin{equation}
\tilde{E}_1(\lambda g, \sqrt{\lambda}h)=\lambda^{-\frac{1}{4}}\tilde{E}_1(g,h).
\end{equation}
Integrating (\ref{8.23}) and using (\ref{8.26}) (\ref{5.23}) (\ref{5.24}), we obtain 
\begin{equation} \label{8.29}
\begin{split}
& \int_{M_{T}} 2 |\nabla_{I}\bar{R}_{i0j0}|^2+|\nabla_{I}\bar{R}_{ijk0}|^2\\ 
 \leq & C(Q^2(\Lambda+\epsilon_0)^2+Q(\Lambda+\epsilon_0) E_1^{\frac{1}{2}}E_2^{\frac{3}{2}}) \int_{0}^{T} (||\nabla H||_{L^{\infty}}+||Hh||_{L^{\infty}})\\
 &+CQ^2(\Lambda+\epsilon_0)^2||H||^{\frac{p_0}{3-p_0}}_{L^{p}(M_0)}E_1^{\frac{1}{2}}E_2^{\frac{1}{2}}+C Q^{\frac{3}{2}}(\Lambda+\epsilon_0)^{\frac{3}{2}}E_2^{\frac{1}{2}}(E_1+\tilde{E}_1)\\
 \leq & C(\tilde{E}_1+E_1+\epsilon_0^{\frac{3}{8}}) Q^2(\Lambda+\epsilon_0)^2 .
 \end{split}
\end{equation}

On the other hand, combining (\ref{Gauss})(\ref{tr Gauss})(\ref{R=Ric}) and the following estimate 
\begin{equation} \label{8.30}
        ||\nabla^2(h\ast h)||_{L^2}\leq C||h||_{L^{\infty}}||\nabla^2h||_{L^{2}}+C||\nabla h||_{L^4}^2\leq CE_1^{\frac{1}{2}}E_2^{\frac{3}{2}},
\end{equation}
we obtain 

\begin{equation}\label{8.31}
\int_{M_{t}}  |\nabla^2\bar{R}m|^2
 \leq  C(\tilde{E}_1+E_1+\epsilon_0^{\frac{3}{8}}) Q^2(\Lambda+\epsilon_0)^2+C E_1E_2^3.
\end{equation}

Now we turn to the estimate of the second fundamental form.

Differentiating (\ref{Simon}) twice, we obtain 
\begin{equation}
\triangle \nabla^2h=\nabla^{4}H+ \sum_{i\leq 2}\nabla^{i}(\bar{R}m+h\ast h) \ast \nabla^{2-i}h+\nabla^{3}\bar{R}m,
\end{equation}
which implies 
\begin{equation}
\int_{M_t}|\nabla^{3}h|^2=\int_{M_t} \nabla^{2}h\ast \nabla^{4}H+ \sum_{i\leq 2}\nabla^{2}h\ast\nabla^{i}(\bar{R}m+h\ast h) \ast \nabla^{2-i}h+\nabla^{2}h\ast\nabla^{3}\bar{R}m.
\end{equation}

By integration by parts and Cauchy-Schwarz inequality,
\begin{equation}\label{8.34}
\begin{split}
\int_{M_t}|\nabla^{3}h|^2 \leq& C\int_{M_t} |\nabla^{3}H|^2+|\nabla^{2}\bar{R}m|^2+\sum_{i\leq2}\nabla^{2}h\ast\nabla^{i}(\bar{R}m+h\ast h) \ast \nabla^{2-i}h\\
\leq & C ||\nabla^{3}H||_{L^2}^2+C ||\nabla^{2}\bar{R}m||_{L^2}^2+C||\nabla^{2}(\bar{R}m+h\ast h)||_{L^2}||h||_{L^{\infty}}||\nabla^{2}h||_{L^{2}}\\
&+ C||(\bar{R}m+h\ast h)||_{L^{\infty}}||\nabla^{2}h||^2_{L^{2}}+C||\nabla^2 h||_{L^2}||\nabla(\bar{R}m+h\ast h)||_{L^{6}}||\nabla h||_{L^3}\\
\leq &C(\Lambda+\epsilon_0)^2+C ||\nabla^{2}\bar{R}m||_{L^2}^2+CQ(\Lambda+\epsilon_0)E_{1}^{\frac{1}{2}}E_2^{\frac{3}{2}}+CE_{2}^{\frac{5}{2}}Q^{\frac{1}{2}}(\Lambda+\epsilon_0)^{\frac{1}{2}}.
\end{split}
\end{equation}

Summing up the estimates (\ref{8.31})(\ref{8.34})(\ref{8.30}), we have 
\begin{equation}\label{8.35}
  ||\nabla^{2} \bar{R}m||^2_{L^2}+||\nabla^{3} h ||^2_{L^2}
\leq C(\tilde{E}_1+E_1+E_2+\epsilon_0^{\frac{3}{8}}) Q^2(\Lambda+\epsilon_0)^2+C(\Lambda+\epsilon_0)^2.
\end{equation}

In view of   (\ref{HL_p2}) and (\ref{8.13}),  we have  $ \tilde{E}_1\leq C(\epsilon_0)$,  
where $C(\epsilon_0)$ satisfies $\lim\limits_{\epsilon_0\rightarrow 0} C(\epsilon_0)=0$. Consequently, we have 
\begin{equation}
 ||\nabla^{2} \bar{R}m||^2_{L^2}+||\nabla^{3} h ||^2_{L^2}
\leq \frac{1}{2}Q^2(\Lambda+\epsilon_0)^2,
 \end{equation}
if  $Q$ is suitably large. This completes the proof of Proposition \ref{p8.2}. 
\end{proof}

\begin{theorem}\label{t8.3}
Under the assumptions of Theorem \ref{maxfol}, for any $t>0$, we have 
\begin{equation}\label{8.37}
\begin{split}
 & ||\nabla^2 \bar{R}m||^2_{L^2(M_t)}+||\nabla^3 h ||^2_{L^2(M_t)} \\ & \leq C \{||\nabla^2 \bar{R}m||^2_{L^2(M_0)}+||\nabla^3 h ||^2_{L^2(M_0)}+(||H||_{L^{p_0}(M_0)}  E_2^{\frac{2}{p_0}-\frac{1}{3}})^{10} +E^{10}_1+E^{\frac{10}{3}}_2\}. \end{split}
\end{equation} 

\end{theorem}

\begin{proof}
First of all,  we claim the following estimate holds: 
\begin{equation}\label{8.38}
  ||\nabla^2 \bar{R}m||^2_{L^2}+||\nabla^3 h ||^2_{L^2}  \leq C  \{\Lambda^2 + E^{10}_1+ E^{\frac{10}{3}}_2+ \epsilon_0^2 +(||H||_{L^{p_0}(M_0)}  E_2^{\frac{2}{p_0}-\frac{1}{3}})^{10}\}, 
\end{equation} 
where  $\Lambda^2= ||\nabla^2 \bar{R}m||^2_{L^2(M_0)}+ ||\nabla^3 h ||^2_{L^2(M_0)}$.

When  $E_1\leq \epsilon_0$, $E_2\leq \epsilon_0$ and (\ref{8.13}) hold, (\ref{8.38}) follows from Proposition \ref{p8.2}. 

Now we assume   $E_1\leq \epsilon_0$, $E_2\leq \epsilon_0$,  but
\begin{equation}
        ||H||_{L^{p_0}(M_0)}  E_2^{\frac{2}{p_0}-\frac{1}{3}} \geq \epsilon_0.
\end{equation}

Note that $(\lambda g(t), \sqrt{\lambda} h(t))$ is still a solution to the mean curvature flow (\ref{mcf}), and 

\begin{equation}
        ||H||_{L^{p_0}(M_0, \lambda g_0)}  E_2(\lambda g_0, \sqrt{\lambda}h_0)^{\frac{2}{p_0}-\frac{1}{3}}=\lambda^{-\frac{1}{4}} ||H||_{L^{p_0}(M_0)}  E_2^{\frac{2}{p_0}-\frac{1}{3}}.
\end{equation}

If we  choose  $$\lambda_0=(||H||_{L^{p_0}(M_0)}  E_2^{\frac{2}{p_0}-\frac{1}{3}}\epsilon_0^{-1})^4 \geq 1$$
 then  (\ref{8.13}) holds for $(\lambda_0 g_0, \sqrt{\lambda_0} h_0)$. Moreover,   under this  scaling, we know  $$E_1\rightarrow \lambda_0^{-\frac{1}{4}}E_1 \leq \epsilon_0; \ \ E_2\rightarrow \lambda_0^{-\frac{3}{4}}E_2\leq \epsilon_0;  \ \ \ \Lambda\rightarrow \lambda_0^{-\frac{5}{4}} \Lambda.$$

Then,  the conclusion of Proposition \ref{p8.2}   holds for $(\lambda_0 g(t), \sqrt{\lambda_0} h(t))$, i.e. 

\begin{equation}\label{8.41}
||\nabla^{2} \bar{R}m_{\lambda_0}||^2_{L^2}+||\nabla^{3} h_{\lambda_0} ||^2_{L^2} \leq C (\lambda_0^{-\frac{5}{4}} \Lambda+\epsilon_0)^2,
\end{equation}
 which  implies 
\begin{equation} \label{8.42}
\begin{split}
 ||\nabla^2 \bar{R}m||^2_{L^2}+||\nabla^3 h ||^2_{L^2} & =   \lambda_0^{\frac{5}{2}} (||\nabla^2 \bar{R}m_{\lambda_0}||^2_{L^2}+||\nabla^3 h_{\lambda_0} ||^2_{L^2})\\
 & \leq C(\Lambda^2+  \lambda_0^{\frac{5}{2}} \epsilon_0^2). \end{split}
\end{equation} 
This proves  (\ref{8.38}) in this case.

Now we treat the claim when  $E_1\geq \epsilon_0$ or $E_2\geq \epsilon_0$.

We consider  two decreasing functions $\lambda\rightarrow \lambda^{-\frac{1}{4}}E_1$ and $\lambda \rightarrow \lambda^{-\frac{3}{4}}E_2$ for $\lambda\geq 1$. Clearly, there exists a unique  $\lambda_0\geq 1$ so that either $(E_1)_{\lambda_0}\leq (E_2)_{\lambda_0}=\epsilon_0$ or  $(E_2)_{\lambda_0}\leq (E_1)_{\lambda_0}=\epsilon_0$ holds. For  the former case, 
a scaling  argument as in (\ref{8.41})(\ref{8.42}) gives  
\begin{equation}
\begin{split}
 ||\nabla^2 \bar{R}m||^2_{L^2}+||\nabla^3 h ||^2_{L^2} \leq C (\Lambda^2+(||H||_{L^{p_0}(M_0)}  E_2^{\frac{2}{p_0}-\frac{1}{3}})^{10}+E^{\frac{10}{3}}_2). \end{split}
\end{equation} 

For the latter case,  we obtain 
\begin{equation}
\begin{split}
 ||\nabla^2 \bar{R}m||^2_{L^2}+||\nabla^3 h ||^2_{L^2} \leq C (\Lambda^2+(||H||_{L^{p_0}(M_0)}  E_2^{\frac{2}{p_0}-\frac{1}{3}})^{10}+E^{10}_1). \end{split}
\end{equation}

Now the claim (\ref{8.38}) is proved for all cases. 

To get rid of the $\epsilon_0^2$-term in (\ref{8.38}),  since (\ref{8.38}) holds for all scalings, we scale the solution and apply the same argument as in (\ref{8.41}) (\ref{8.42}) to obtain 
\begin{equation}
 ||\nabla^2 \bar{R}m||^2_{L^2}+||\nabla^3 h ||^2_{L^2} \leq C \{\Lambda^2+(||H||_{L^{p_0}(M_0)}  E_2^{\frac{2}{p_0}-\frac{1}{3}})^{10}+E^{10}_1+E^{\frac{10}{3}}_2+\lambda^{\frac{5}{2}}\epsilon_0^2\}, 
\end{equation} 
which holds for any $\lambda>0$.  Let $\lambda\rightarrow 0$, we obtain (\ref{8.37}). The proof is completed. 
\end{proof}

\begin{theorem} \label{t8.4} Under the assumptions of Theorem \ref{maxfol}, if the following holds at time $t=0$ for some $k\geq 2$:   
\begin{equation}
 \sum_{i=0}^{k}||\nabla^i \bar{R}m||^2_{L^2(M_0)}+\sum_{i=1}^{k+1}||\nabla^i h ||^2_{L^2(M_0)}< \infty,
\end{equation}
 
then there exists  $\Lambda_k>0$ depending only  on the initial data $(g_0,h_0)$ such that   for any $t>0$, we have 
\begin{equation}
\begin{split}
 \sum_{i=0}^{k}||\nabla^i \bar{R}m||^2_{L^2(M_t)}+\sum_{i=1}^{k+1}||\nabla^i h ||^2_{L^2(M_t)} \leq \Lambda_k.
 \end{split}
 \end{equation}
\end{theorem}

\begin{proof} We have proved the result for  $k=2$ in Theorem \ref{t8.3}. When $k\geq 3$, the proof is similar. We omit the details. 
\end{proof}

\section{Maximal foliations}

In this section, we aim to complete the proofs of Theorem \ref{t1.5} and  Theorem \ref{maxfol}.

\begin{proof} of Theorem \ref{t1.5}.

Let $\mathcal{M}$ be the maximal development of $M_0$ in the sense of \cite{CG69}. 
Let $\bar{M}_1 \subset \mathcal{M}$ be an open submanifold constructed in Section \ref{Pre} such that $\bar{M}_1\overset{diff}{\approx} M_0\times (-\delta_1,\delta_1)$ and a time function $\tau$ on $\bar{M}_1$ defines a smooth foliation  as in (\ref{tau})  and $M_0=\{\tau=0\}$.  It should be noted that $\delta_1$ depends only on the $L^{\infty}$-bounds of the spacetime curvature and the second fundamental form of  $\tau-$foliation.    Let $[0, t_1)$ ($0<t_1\leq \infty$) be the maximal time interval  so that the mean curvature flow exists within $\bar{M}_1$ and $\nu=-\langle N, T\rangle\leq 2$, where  $N$ and $T$ are   timelike unit normals of  $\{\tau=const.\}$ and $M_{t}$ respectively (see Section \ref{Pre}). We have two cases. 

Case 1, $t_1=\infty$.

In this case,  let  $f(x,t)$ be a graphical solution of the mean curvature flow equation (\ref{mcfg}) in the above $\tau$-foliation. In view of  $\nu\leq 2$,  we have
\begin{equation} \label{9.1}
|\nabla f|^2=\nu^2-1\leq 3.
\end{equation}

From $\frac{\partial f}{\partial t}=\nu^{-1} H$ and $\int_{0}^{\infty} |H|dt<\infty$, we know the limit $ f_{\infty}(x)\triangleq \lim\limits_{t\rightarrow \infty}f(x,t)$  exists. 
Combining   (\ref{9.1}), Theorem \ref{8.4} and (\ref{61}),  we know $f(x,t)$ converges to $f_{\infty}(x)$ in $C^{k+1}$ topology ($k\geq 2$) and the graph $\{(x,f_{\infty}(x)): x\in M_0\}$ is a  maximal spacelike hypersurface.  We complete the proof in  Case 1.

Case 2, $t_{1}<\infty$.  

First of all, we remark that Theorem  \ref{t2.4} and Theorem \ref{t2.5} actually hold on the whole time interval $[0,t_1)$. We even do not need to modify the arguments.    
 
By the same argument as in Case 1, we know $f(x,t_1)\triangleq \lim\limits_{t\rightarrow t_1}f(x,t)$ exists. Moreover, the convergence can be in $C^{k+1}$-topology.  In view of Theorem \ref{t2.4} and Theorem \ref{t2.5}, we also have two cases. There exists a point $x_1\in M_0$ such that either $|f(x_1,t_1)-f(x_1,0)|=\delta_1$ or $ \nu(x_1,t_1)=2$.  In the former case, integrating $\frac{\partial f}{\partial t}=\nu^{-1} H$, we have 
\begin{equation}\label{9.2}
\delta_1=|f(x_1,t_1)-f(x_1,0)|\leq C \int_{0}^{t_1}|H|dt.
\end{equation}
For the latter case,  direct computation gives  
\begin{equation}\label{9.3}
\begin{split}
\frac{\partial}{\partial t} \nu&=H D^2\tau(N,N)+\langle T, \nabla H\rangle\\
&=HD^2\tau(N-\nu T,N-\nu T)+\langle T-\nu N, \nabla H\rangle,
\end{split}
\end{equation}
where we have used $D^2\tau(T, \ast)=0$.

Note that  \begin{equation}\label{9.4}
\begin{split}
& D^2\tau(N-\nu T, N-\nu T)=II_{\Sigma_{\tau}}(N-\nu T, N-\nu T)\leq C |N-\nu T|^2=C(\nu^2-1)\\
& \langle T-\nu N, \nabla H\rangle \leq |\nabla H|\sqrt{\nu^2-1}.\end{split}
\end{equation}
Integrating (\ref{9.3}) yields 
\begin{equation}\label{9.5}
\begin{split}
2\leq & 1 +C\int_{0}^{t_1} ||H||_{L^{\infty}}dt+C\int_{0}^{t_1}||\nabla H||_{L^{\infty}}dt.
\end{split}
\end{equation}
Combining (\ref{9.2}) and (\ref{9.5}), we have 

\begin{equation} \label{9.6}
\min\{ \delta_1, 1\}\leq C( \int_{0}^{t_1}||H||_{L^{\infty}}dt+\int_{0}^{t_1}||\nabla H||_{L^{\infty}}dt)\leq C\sqrt{t_1},
\end{equation}
where we have used  (\ref{61}) and 
\begin{equation}\int_{0}^{t_1}||\nabla H||_{L^{\infty}}dt \leq C \int^{t_1}_{0}||\nabla^2 H||^{\frac{1}{2}}_{L^2}||\nabla^3 H||^{\frac{1}{2}}_{L^2}dt \leq CE^{\frac{1}{2}}_1E^{\frac{1}{2}}_2 \sqrt{t_1}, \end{equation}
which follows from (\ref{5.28}) and (\ref{8.12}).

(\ref{9.6}) implies that $t_1$ has a positive lower bound depending only on the initial data. 

Let  $M_{t_1}$ be the new initial Cauchy surface, we solve the Cauchy problem of vacuum Einstein equation. Then  we construct  $\bar{M}_2 \subset \mathcal{M}$ as  in Section \ref{Pre} such that $\bar{M}_2\overset{diff}{\approx} M_{t_1}\times (-\delta_2,\delta_2)$, and  a time function $\tau_2$ on $\bar{M}_2$ defines a smooth foliation  as in (\ref{tau})  and $M_{t_1}=\{\tau_2=0\}$.

One can continue to  solve the graphical mean curvature flow equation (\ref{mcfg}) in $\tau_2$-foliation.

Actually, since  Theorem \ref{t2.4} and Theorem \ref{t2.5} hold at time $t_1$,   one can  use the original auxiliary function $\tilde{r}$  to define a  sequence of  exhausting domains $\Omega^2_{i}\subset \subset \bar{M}_2$ with $\partial_{s} \Omega^2_i=\{\tilde{r}=i\}\cap \bar{M}_2$. At this point, the boundary gradient estimate can be easily obtained.  Then the   Nash-Moser iteration argument gives a uniform gradient estimate on $\Omega^2_{i}$.  By extracting a convergent subsequence, we obtain  the short time existence of global graphical mean curvature flow in $\tau_{2}$-foliation.  

let $[t_{1},t_{2})$ be the maximal time interval such that the graphical mean curvature flow exists within  $\tau_2$-foliation as in (\ref{tau}) and $\nu\leq 2$.

 If $t_2=\infty$,  we are done.  If $t_2<\infty$,  we will  continue the mean curvature flow beyond time $t_2$. 
 
Note that the function $w=w^{\sigma}$  constructed in (\ref{253}) (\ref{254}) is a decreasing function of $\sigma$. In view of (\ref{256}),  we  choose   smaller $\sigma_2<\sigma_0$ in (\ref{2.53}) (\ref{2.54})  which can ensure  \begin{equation}\label{98}
\partial (M_{w^{\sigma_2}}\cap \bar{M}_2) \subset \partial \bar{M}_2. 
\end{equation}
Taking (\ref{98}) into account,  employing the  same argument of Theorem \ref{t2.4}, one can prove that when $r(x)$ is large and $t\in [t_1,t_2)$,  the mean curvature flow  lies between  $M_{w^{\sigma_2}}$ and $M_{-w^{\sigma_2}}$. In particular, the exterior part  of the mean curvature flow remains to be a graph in $t_1$-foliation and satisfies $|f(x,t)|\leq C r(x)^{-\sigma_2}$, for all $t\in [t_1,t_2)$.  Moreover, this results that the function $r(x)$ can still  be used to construct cut-off functions  when extending  the argument of Theorem \ref{t2.5} to time interval  $ [t_1,t_2)$.  Hence,  Theorem \ref{t2.5} remains to be true when $r(x)$ is large and $t\in [t_1,t_2)$.  Then the previous argument at time $t_1$ can be applied, we conclude that  the mean curvature flow can be extended beyond time $t_2$.

In general, if we are in Case 2 in the i-th step,  let $[t_{i-1},t_{i})$ be the maximal time interval such that the graphical mean curvature flow exists within  a foliation  $\bar{M}_i\overset{diff}{\approx} M_{t_{i-1}}\times (-\delta_i,\delta_i)$  as in (\ref{tau}) and $\nu\leq 2$.  According to Theorem \ref{t8.3}, $\delta_i$ can be bounded from below by a constant depending only on the initial data.

 Inductively, choosing smaller $\sigma_i<\sigma_{i-1}$ so that $\partial (M_{w^{\sigma_i}})\cap \bar{M}_i \subset \partial \bar{M}_i$,    the arguments  of  Theorems \ref{t2.4} and \ref{t2.5} can still go through.  The consequence is that there exists a  large  $r_i>r_{i-1}$  such that  when  $t\in [t_{i-1}, t_{i})$,  the exterior  part of the mean curvature flow ($r(x)\geq r_i$)  remains to be a graph in $\tau_1$-foliation, and the height function $f$ satisfies 
 \begin{equation} \label{119}
 \begin{split}
 & |f(x,t)|  \leq C r(x)^{-\sigma_i}\\
&  |\nabla f|(x,t) \leq C r(x)^{-e^{-\frac{5}{8}}(\frac{2229}{2048}+\frac{21}{32}\epsilon)} (t+1)^{\frac{1177}{2028}}, 
 \end{split}
 \end{equation}
for all $t\in [t_{i-1},t_i)$  and $r(x)\geq r_i$. 

Therefore, applying the same argument at time $t_1$,   one  can continue the mean curvature flow beyond time $t_i$.

In view of Theorem \ref{t7.1},  the  same arguments as in (\ref{9.2})(\ref{9.5})(\ref{9.6}) give us  $t_i-t_{i-1} \geq C^{-1}$, which implies \begin{equation} \label{9.8}
t_i \geq C^{-1}i, 
\end{equation}
 where $C$ is some constant depending only on the initial data.

On the other hand,   let $a=t_{i-1}$ in (\ref{5.29}) (\ref{5.35}) (\ref{5.36}), we find 

\begin{equation} \label{1111}
\int_{t_{i-1}}^{t_i}||\nabla H||_{L^{\infty}}dt \leq C t_{i-1}^{1-\frac{3}{2p_0}},
\end{equation}
and the estimate (\ref{4.17}) tells us  

\begin{equation} \label{9.10}
\int_{t_{i-1}}^{t_i}||H||_{L^{\infty}}dt\leq C \int_{t_{i-1}}^{\infty} t^{-\frac{3}{2p_0}}\leq C t_{i-1}^{1-\frac{3}{2p_0}}.
\end{equation}

 The  same arguments as in (\ref{9.2})(\ref{9.5})(\ref{9.6}) 
 imply  \begin{equation} \label{9.11}
C^{-1}\leq \min\{ \delta_i,1\} \leq  C \int_{t_{i-1}}^{t_i}||H||_{L^{\infty}}dt+C \int_{t_{i-1}}^{t_i}||\nabla H||_{L^{\infty}}dt \leq C t_{i-1}^{1-\frac{3}{2p_0}}\leq C (i-1)^{1-\frac{3}{2p_0}}.
\end{equation}

Since $1-\frac{3}{2p_0}<0$, (\ref{9.11}) is invalid if $i$ is too large.   That means, after finite number of steps, we must be in Case 1, that is to say, the mean curvature flow will converge to a maximal hypersurface. We denote it by $S$.

Note that we have to prove that $S$ is asymptotically flat in some sense.

To this end,  we note that  outside a compact set, $S$ can be written as a graph in $\tau_1$-foliation, and the hight function $f$ satisfies 

 \begin{equation}
|f|(x)\leq C {r}(x)^{-\sigma},
\end{equation}
for some $0<\sigma<\epsilon$.  Moreover, the vanishing of the mean curvature of $S$ implies 
\begin{equation}
\begin{split}
\triangle f& =-div_{S} (T)\\
&=-H_0-\frac{h_0(Df,Df)}{\sqrt{1-|Df|^2}}.
\end{split}
\end{equation}

We need to derive a gradient estimate of $f$. 

The  approach  is to apply  a  gradient estimate  for Poisson equations (see Proposition \ref{P9.1}). 
Note that the Ricci curvature of $S$ satisfies $ R_{ij}=g^{kl}h_{ik}h_{jl}-\bar{R}_{i0j0}\geq -C r^{-2-\epsilon}$.   For all $a\leq r(x)/ 2$, we have 
\begin{equation} \label{ge}
\sup_{B(x,\frac{a}{2})}|\nabla f| \leq C (\frac{Osc (f)\mid_{B(x,a)}}{a}+a \sup_{B(x,a)}(|H_0|+|h_0||\nabla f|^2)).
\end{equation}

Since $|H_0|\leq C r^{-2-\epsilon}$, $|h_0| \leq Cr^{-1-\epsilon} $, $|\nabla f|\leq C$, by taking $a=r(x)^{\frac{1}{2}}$ in (\ref{ge}), we find $|\nabla f|\leq C r(x)^{-\frac{1}{2}-\sigma}$,  which further implies 

\begin{equation} \label{F}
|H_0|+|h_0| |\nabla f|^2 \leq C r^{-2-\epsilon}.
\end{equation}

Substituting (\ref{F}) into (\ref{ge}) and choosing $a=\frac{r(x)}{2}$, it follows  that 
\begin{equation} \label{ge1}
|\nabla f|(x)\leq Cr(x)^{-1-\sigma}.
\end{equation}

Now we change to  wave coordinates $\{z^{\alpha}\}$.  Outside a compact set,  $S$ lies between $M_{w^{\sigma}}$ and $M_{-w^{\sigma}}$, and it is also a graph $z^0=u(z^1,z^2,z^3)$  with   \begin{equation}
|u|\leq C \tilde{r}^{-\sigma}.
\end{equation}

From chain rule, (\ref{ge1}) and i) in Lemma \ref{l2.2}, we have 
\begin{equation} \label{920}
\frac{\partial u}{\partial z^i}=\frac{\partial z^0}{\partial x^k}\frac{\partial x^k}{\partial z^i}+\frac{\partial z^0}{\partial \tau}\frac{\partial f}{\partial x^k}\frac{\partial x^k}{\partial z^i}=O(r^{-1-\epsilon})+(r^{-1-\delta})=O(r^{-1-\delta}).\end{equation}

Substituting $z^0=u(z^1,z^2,z^3)$ in (\ref{2.57}),  denoting  the induced metric  on $S$  in coordinates $\{z^1,z^2,z^3\}$ by $g_{ij}dz^idz^j$, we have \begin{equation} \label{921}
g_{ij}=g^{z}_{ij}+\beta_iu_j+\beta_ju_i-(\alpha^2-\beta^2)u_iu_j.
\end{equation}
It follows from (\ref{AC est}) that  $g_{ij}^{z}=\delta_{ij}+O(r^{-\epsilon})+O(r^{-1-\epsilon})=\delta_{ij}+O(r^{-\epsilon})$, $\alpha=1+O(r^{-\epsilon})$,  $\beta=O(r^{-\epsilon})$,  combining them with   (\ref{920}) deduces 
\begin{equation}
g_{ij}=\delta_{ij}+O(r^{-\epsilon}). 
\end{equation}

To obtain higher regularity of $g_{ij}$, we need the elliptic  equation (see (\ref{2.59})) of $u$,   which is derived from $H_{S}=0$:  
\begin{equation}\label{923}
div_{M_{z^0}} (U)+\frac{1}{2}(1-|U|^2)^{-1}D_{U+T}|U|^2=-H_{M_{z^0}}-\langle D_{T}T, U\rangle,
\end{equation}
where $U=\frac{\alpha D u}{1+Du\cdot \beta}$, $T=\alpha^{-1}(\frac{\partial}{\partial z^0}-\beta)$.

Note that the inverse matrix of $(\bar{g}_{\alpha\beta})$ in (\ref{2.57}) can be computed explicitly:
\begin{equation} \label{924}
\bar{g}^{00}=-\alpha^{-2}, \bar{g}^{0i}=\alpha^{-2}\beta^{i}, \bar{g}^{ij}=(g^{z})^{ij}-\alpha^{-2}\beta^i\beta^j.
\end{equation}

Because of the chain rule  and Lemma \ref{l2.2} (see (\ref{920})),  we have 
\begin{equation}\label{925}
\begin{split}
\frac{\partial \alpha\mid_{z^0=u}}{\partial z^i}&=\frac{1}{2}\alpha^{-1}[\frac{\partial}{\partial z^i}(-\bar{g}_{00}+|\beta|^2)+\frac{\partial}{\partial z^0}(-\bar{g}_{00}+|\beta|^2)\frac{\partial u}{\partial z^i}]\\
&=\alpha^{-1} (\partial \bar{g}+\partial \bar{g}\ast \partial u) \ast [\bar{g} \ast  (g^{z})^{-1}+\bar{g}^2\ast (g^{z})^{-2}] \\&=O(r^{-1-\epsilon})+O(r^{-2-\epsilon-\delta})=O(r^{-1-\epsilon}),\\
\frac{\partial} {\partial z^i}(\beta\mid_{z^0=u})&= \partial \bar{g}+\partial u\ast \partial \bar{g}=O(r^{-1-\epsilon}),\\
 \frac{\partial} {\partial z^i}(g^{z}_{kl}\mid_{z^0=u})& = \partial \bar{g}+\partial u\ast \partial \bar{g}=O(r^{-1-\epsilon}).
\end{split}
\end{equation}
We rewrite  (\ref{923})  as 
\begin{equation} \label{926}
a^{ij} \partial^2_{ij}u=F
\end{equation}
with  \begin{equation} \label{927}
\begin{split}
a^{ij}=& \frac{\alpha}{1+Du\cdot \beta}[(g^{z})^{ij}+(\frac{U_kU_l}{1-|U|^2}-\alpha^{-1}U_k\beta_l-\alpha^{-1}U_l\beta_k+\frac{\alpha^{-2}|U|^2}{1-|U|^2}\beta_k\beta_l)(g^{z})^{ik}(g^{z})^{jl}]
\end{split}
\end{equation} and 
\begin{equation}\label{928}
\begin{split}
F=& -H_{M_{z^0}}+ \partial u\ast \partial u \ast Q_{1}+\partial u\ast \partial \bar{g} \ast Q_2,
\end{split} 
\end{equation}
where $Q_1,Q_2$ are polynomials of $\alpha, \alpha^{-1}, \beta, g^{z}, (g^{z})^{-1}, \bar{g}, \bar{g}^{-1}, \partial u, \partial \bar{g},  (1-|U|^2)^{-1}, (1+Du\cdot \beta)^{-1}$.

Since $a^{ij}=\delta_{ij}+O(r^{-\epsilon})$, $F=O(r^{-2-\sigma})$, applying $L^p$-estimate of elliptic equations (see \cite{GT77}, Chapter 9),  for large $x$, we obtain  
\begin{equation}\label{929}
\begin{split}
||\partial^2 u||_{L^p (|x^{\prime}-x|\leq \frac{r}{4})}& \leq C (r^{-2} || u||_{L^p(|x^{\prime}-x|\leq \frac{r}{2})}+ r^{-1} || \partial u||_{L^p(|x^{\prime}-x|\leq \frac{r}{2})}+   || F ||_{L^p(|x^{\prime}-x|\leq \frac{r}{2})})\\
& \leq C r^{-2-\sigma} r^{\frac{3}{p}}.
\end{split}
\end{equation}

Differentiating (\ref{923}) and  making use of  (\ref{925})(\ref{927})(\ref{928}) and same $L^p$-estimate  as above, we obtain 
\begin{equation}\label{930}
\begin{split}
||\partial^3 u||_{L^p (|x^{\prime}-x|\leq \frac{r}{8})}\leq C r^{-3-\sigma} r^{\frac{3}{p}}.
\end{split}
\end{equation}

Applying Sobolev embedding theorem and (\ref{929}) (\ref{930}) gives 
\begin{equation}\label{931}
|\partial^2 u|\leq C r^{-2-\sigma}.
\end{equation}

In view of (\ref{925})(\ref{AC est}), one can further differentiate (\ref{923}) up to  three times and make use of $L^2$-estimates as above to obtain:
\begin{equation}\label{933}
\begin{split}
& ||\partial^4 u||_{L^2 (|x^{\prime}-x|\leq \frac{r}{16}))}\leq C r^{-4-\sigma} |\{|x^{\prime}-x|\leq \frac{r}{8}\}|^{\frac{1}{2}},\\  & ||\partial^5 u||_{L^2 (|x^{\prime}-x|\leq \frac{r}{32}))}\leq C r^{-5-\sigma} |\{|x^{\prime}-x|\leq \frac{r}{16}\}|^{\frac{1}{2}}, \end{split}
\end{equation}
which together with (\ref{930}) and   Sobolev embedding theorem implies 
\begin{equation}\label{934}
|\partial^3 u| \leq  C r^{-3-\sigma}.
\end{equation}
Combining (\ref{921}) (\ref{927}) (\ref{934}) and (\ref{AC est}), we obtain 
\begin{equation}
g_{ij}=\delta_{ij}+O_2(\tilde{r}^{-\epsilon}),\ \ \  k_{ij}=O_1(\tilde{r}^{-1-\epsilon}). \end{equation}

Clearly, the regularity can be improved if we impose higher regularity on the initial data (\ref{AC}). 

The proof is completed. 
\end{proof}

\begin{proposition}\label{P9.1}
Let $(M^n,g)$ be a Riemannian manifold  with $Ric \geq -(n-1)K$ for some $K>0$. Then for any $B(x_0,a)\subset\subset M^n$ with  $0<a\leq \frac{1}{\sqrt{K}}$, we have 
\begin{equation} \label{935}
\sup_{B(x_0,\frac{q}{2})} |\nabla u|\leq C (\frac{Osc (u)\mid_{B(x_0,a)}}{a}+a \sup_{B(x_0,a)}|\triangle u|),
\end{equation}
where $C$ depends  only on the isoperimetric constant of $B(x_0,a)$. 
\end{proposition}

By scaling invariance,  we only need to prove the result when $K=1$, $a=1$,$ Osc (u)\mid_{B(x_0,1)}\leq 1$,  $\sup_{B(x_0,1)} |\triangle u|\leq 1$. The proof is a standard application of Nash-Moser iteration and Bochner formula. We omit the details.

Now we are in a position to prove Theorem \ref{maxfol}

\begin{proof} of Theorem \ref{maxfol}.

Recall that the first variation $\mathcal{L}H(X)$ of the mean curvature function along an arbitrary vector field $X$ reads
\begin{equation} \label{936}
    \mathcal{L}H(X)=-\triangle\langle X,N\rangle+(|h|^2+\bar{R}ic(N,N))\langle X,N\rangle+\langle X,\nabla H\rangle,
\end{equation}
where $N$ is the unit  normal vector field.    

Let $S$ be the maximal hypersurface constructed in Theorem \ref{t1.5}. Since the curvature of $S$ is bounded, one can perform the same construction in Section 2.1 to give a foliation  as in (\ref{tau}) 
\begin{equation}\label{937}
    S_{\tau}\triangleq \{\tau=const.\}, \tau\in (-\delta^{\prime}_1, \delta^{\prime}_1), \ \ \ S=S_0. 
\end{equation}
 Let $\{(x,f(x)):x\in S\}$ be a spacelike submanifold which is a  graph over $S$ in the above foliation, we  denote its  mean curvature function by  $H_{f}$. Then,   $H_{f}$ is a smooth map from  $W^{3,q}(S_0)\oplus \mathbb{R}$ to $W^{1,q}(S_0)$ ($q>3$), where  $W^{3,q}(S_0)\oplus \mathbb{R}$ consists of those functions $\xi+c$, $\xi\in W^{3,q}(S_0)$, $c\in \mathbb{R}$.  

From  (\ref{936}), the linearized operator  of the  mean curvature for graphical variations  at $S$  is given by 
\begin{equation} \label{938}
LH\mid_{f=0}(w)=\triangle w-|h|^{2} w. \end{equation}

By standard elliptic estimate, $LH\mid_{f=0}:W^{3,q}(S_0) \rightarrow W^{1,q}(S_0)$ is invertible. 

By implicit function theorem, there exists $\epsilon>0$, for any $\tau\in (-\epsilon,\epsilon)$, there exists a unique $f_{\tau}\in W^{3,q}(S_0)$ satisfying  $H_{f_{\tau}+\tau}=0$.  Let $\Sigma_{\tau}=\{(x,f_{\tau}+\tau):x\in S_0\}$. 

In view of (\ref{936}) and (\ref{938}), differentiating the equation $H_{f_{\tau}+\tau}=0$ with respect to $\tau$ gives 

\begin{equation} \label{939}
0=(-\triangle+|h|^2) [(1+\frac{\partial f_{\tau}}{\partial \tau})\langle N, \frac{\partial}{\partial \tau}\rangle].
\end{equation}

If we set  $u= (1+\frac{\partial f_{\tau}}{\partial \tau})\langle N, \frac{\partial}{\partial \tau}\rangle$, then $\triangle u=|h|^2 u$. One can prove that $u\rightarrow -1$ at spatial infinity.   Multiplying  both sides of (\ref{939}) with $u_{+}$ (the positive part of $u$) and integrating by parts,  we obtain 
\begin{equation}
\int_{\Sigma_{\tau}} |\nabla u_{+}|^2+|h|^2 |u_{+}|^2=0
\end{equation}
which implies $u_{+}=0$, i.e., $u\leq 0$.  Note that the Ricci curvature of $\Sigma_{\tau}$ is bounded from below (see (\ref{tr Gauss})). Applying  Yau's  gradient estimate to the equation $\triangle u=|h|^2u$ (see \cite{Y75} Theorem $3^{\prime}$ or  Theorem 3), we have $|\nabla u|\leq C (-u)$,  which implies that $u$ has no zeros, i.e., $u<0$, from which it follows 
 \begin{equation}
 \frac{\partial}{\partial \tau}(\tau+f_{\tau})>0, \ \ \text{on} \ S_0\times (-\epsilon, \epsilon). 
 \end{equation}
 
 That means, $\Sigma_{\tau}=\{(x,f_{\tau}+\tau):x\in S_0\}$, $\tau\in (-\epsilon,\epsilon)$ forms a smooth maximal foliation.  

 After  scaling, we may assume 
 
 \begin{equation} \label{942}
 E_1=\int_{\Sigma_0} |\bar{R}m|^2+|\nabla h|^2+|h|^4=1.
 \end{equation}

 Let $n$ be the lapse function defined by $n^{-2}=-\bar{g}(D\tau, D\tau)$.   After reparametrization, we may assume $n$ approaches to  1 at spatial infinity.  By exploiting the main result in \cite{KRS15},  there exist  constants $\tau_0$, $n_0$ depending only $\alpha_0$ and $\alpha_1$ such that the maximal foliation can be extended to a larger and uniform  interval  $(-\tau_0, \tau_0)$, and  the lapse function $n$ satisfies 
 \begin{equation} \label{943}
 n_0 \leq n \leq n^{-1}_0
 \end{equation}
 on $\Sigma_{\tau}$, for all $\tau\in (-\tau_0,\tau_0)$.

For large $t$, the mean curvature flow $M_t$ can be written as a graph $
M_t=\{(x,f(x,t)):x\in \Sigma_0\}
$ over $\Sigma_0$ with $\tau$ as the hight function $\tau\mid_{M_t}\triangleq f$.  The mean curvature flow (\ref{mcf}) is equivalent to 

\begin{equation}
  \frac{\partial f}{\partial t}= \frac{H}{n\sqrt{1+n^2|\nabla f|^2}},
  \end{equation}
where $H$ is the mean curvature of $M_t$. 
 
Owing to the significant result (\ref{943}), we obtain 
\begin{equation}
|\frac{\partial f}{\partial t}(x,t)| \leq n_0^{-1}|H|,
\end{equation}
 which implies 

 \begin{equation}
 |f(x,t_0)|\leq n_0^{-1} \int_{t_0}^{\infty}||H||_{L^{\infty}}dt \leq  C||H||_{L^{p_0}(M_0)}^{\frac{p_0}{3-p_0}},
 \end{equation}
taking (\ref{5.26}) into account. In view of (\ref{HL_p2}) and (\ref{942}), we obtain 
 \begin{equation} \label{947}
 |f(x,t_0)| \leq  C\epsilon_0^{\frac{p_0}{3-p_0}} \leq \tau_0, 
 \end{equation}
 if we choose a smaller $\epsilon_0$. Consequently,  as the right hand side estimate of (\ref{947}) is uniform (independent of $t_0$), one can take $t_0$ to be $0$.  In other words, the maximal foliation constructed above covers the initial hypersurface $M_0$. The estimate (\ref{1.16})  follows  from Theorem \ref{t7.1}.

 The proof of Theorem \ref{maxfol} is completed. 
\end{proof}


\end{document}